\documentclass[11pt]{amsart}

\usepackage{amsmath, amsthm, amssymb, mathrsfs}
\usepackage[hidelinks]{hyperref}
\usepackage[parfill]{parskip}    \parskip = 0.2cm    
\usepackage[margin=1.12in]{geometry} 
\usepackage{tikz}
\usepackage{tikz-cd}
\usetikzlibrary{matrix}
\usepackage{comment} 

\usepackage{mathtools}

\graphicspath{ {images/} }

\numberwithin{equation}{section}

\title{\textbf{Cut and project sets with polytopal window II: linear repetitivity}}
\date{\today}

\author{Henna Koivusalo}
\address{School of Mathematics, Fry Building, Woodland Road, Bristol BS8 1UG, United Kingdom}
\email{henna.koivusalo@bristol.ac.uk }
\urladdr{https://research-information.bris.ac.uk/en/persons/henna-l-l-koivusalo}

\author{James J.\ Walton} 
\address{School of Mathematical Sciences, Mathematical Sciences Building, University Park, Nottingham, NG7 2RD, United Kingdom}
\email{James.Walton@nottingham.ac.uk}
\urladdr{https://www.nottingham.ac.uk/mathematics/people/james.walton}

\theoremstyle{plain}
\newtheorem{theorem}{Theorem}[section]
\newtheorem{definition}[theorem]{Definition}
\newtheorem{lemma}[theorem]{Lemma}

\newtheorem{proposition}[theorem]{Proposition}
\newtheorem{corollary}[theorem]{Corollary}

\newtheorem{Mainthm}{Theorem}

\theoremstyle{definition}
\newtheorem{remark}[theorem]{Remark}
\newtheorem{example}[theorem]{Example}
\newtheorem{notation}[theorem]{Notation}

\newcommand{\R}{\mathbb R}
\newcommand{\C}{\mathbb C}
\newcommand{\Z}{\mathbb Z}
\newcommand{\Q}{\mathbb Q}

\newcommand{\N}{\mathbb N}

\newcommand{\sH}{\mathscr{H}}
\newcommand{\sA}{\mathscr{A}}
\newcommand{\sC}{\mathscr{C}}
\newcommand{\sF}{\mathscr{F}}
\newcommand{\cV}{\mathcal{V}}

\newcommand{\E}{\mathbb E}
\newcommand{\spintl}{X_i}

\renewcommand{\epsilon}{\varepsilon}

\newcommand{\iW}{\mathring{W}}      

\DeclareMathOperator{\rk}{rk}

\DeclareMathOperator{\conv}{\mathrm{conv}}

\DeclareMathOperator{\covol}{covol}

\DeclareMathOperator{\tot}{\E}      
\DeclareMathOperator{\phy}{\E_\vee}      
\DeclareMathOperator{\intl}{\E_<}     
\DeclareMathOperator{\cps}{\Lambda}          


\thanks{We gratefully acknowledge the support of the London Mathematical Society, Scheme 2. The research of JW was partially support by EPSRC grant EP/R013691/1.}

\begin{document}

\begin{abstract}
This paper gives a complete classification of linear repetitivity ({\bf LR}) for a natural class of aperiodic Euclidean cut and project schemes with convex polytopal windows. Our results cover those cut and project schemes for which the lattice projects densely into the internal space and (possibly after translation) hits each supporting hyperplane of the polytopal window. Our main result is that {\bf LR} is satisfied if and only if the patterns are of low complexity (property {\bf C}), and the projected lattice satisfies a Diophantine condition (property {\bf D}). Property {\bf C} can be checked by computation of the ranks and dimensions of linear spans of the stabiliser subgroups of the supporting hyperplanes, as investigated in Part I to this article. To define the correct Diophantine condition {\bf D}, we establish new results on decomposing polytopal cut and project schemes to factors, developing concepts initiated in the work of Forrest, Hunton and Kellendonk. This means that, when {\bf C} is satisfied, the window splits into components which induce a compatible splitting of the lattice. Then property {\bf D} is the requirement that, for any suitable decomposition, these factors do not project close to the origin in the internal space, relative to the norm in the total space. On each factor, this corresponds to the usual notion from Diophantine Approximation of a system of linear forms being badly approximable. This extends previous work on cubical cut and project schemes to a very general class of cut and project schemes. We demonstrate our main theorem on several examples, and derive some further consequences of our main theorem, such as the equivalence {\bf LR}, positivity of weights and satisfying a subadditive ergodic theorem for this class of polytopal cut and project sets.
\end{abstract}


\maketitle

\section*{Introduction}
There are currently two main approaches to systematically constructing aperiodically ordered patterns: by \emph{tiling substitution} and by the \emph{cut and project method}. Whilst there is no single mathematical characterisation of what makes a pattern `aperiodically ordered', there are several important indicators of it. Some of the most natural attributes considered in this context are: \emph{diffraction}, \emph{complexity} and \emph{repetitivity}.

An aperiodic pattern can satisfy ideal qualitative properties with respect to these attributes. For example, a completely disordered pattern has no distinguishable features in its diffraction, but some aperiodic patterns can have \emph{pure point diffraction}, making the diffraction spectrum a countable sum of $\delta$-peaks. The study of spectral properties is in part motivated by the discovery of \emph{quasicrystals} \cite{SchBleGraCah84}, physical materials which have sharp peaks in their diffraction patterns, indicating order, but with rotational symmetry prohibiting them having periodically arranged structure. On measures of complexity, one may ask if the \emph{complexity function} $p(r)$, given by the number of distinct patches of radius $r$ up to translation equivalence, is `minimal'. For example, in the setting of polytopal cut and project sets (see Corollary \ref{cor: minimal complexity}) one always has $p(r) \gg r^d$, where $d$ is the dimension of the pattern; in fact, it is conjectured that this holds for any non-periodic pattern \cite[Conjecture 1.1]{LP03} (it is known to hold for all linearly repetitive patterns \cite{Len04}). Turning to repetitivity, it is of great interest to know when a pattern is \emph{linearly repetitive} ({\bf LR}), so that $\rho(r) \leq Cr$ for some $C > 0$ and sufficiently large $r$. Here, $\rho(r)$ is the \emph{repetitivity function}, defined by letting $\rho(r)$ be the smallest value of $R$ so that every $r$-patch which appears somewhere in the pattern in fact appears within radius $R$ of every point, up to translation. The work of Lagarias and Pleasants \cite{LP03} showed that linear repetitivity is a restrictive and distinctive property of aperiodic patterns which implies other indicators of high order, such as fast convergence to uniform patch frequencies. Furthermore, Lagarias and Pleasants showed that linear repetitivity is the lowest possible aperiodic repetitivity \cite[Theorem 8.1]{LP03}. 

It is easily shown that substitution tilings (with some standard restrictions, such as finite local complexity and primitivity) are {\bf LR}, from which it follows that their complexity is $p(r) \ll r^d$, where $d$ is the dimension of the pattern. However, the problem of identifying substitution tilings with pure point diffraction is extremely difficult. The \emph{Pisot Conjecture} states that $1$-dimensional irreducible, primitive, substitution tilings always have pure point diffraction. Although some cases have been resolved (for $\beta$-substitutions \cite{Bar18} and two-letter substitutions \cite{BarDia02,HolSol03,BMST16}), the full conjecture is currently still open, with the situation in higher dimensions seeming far out of reach.

For cut and project patterns, our understanding is in some sense the reverse. It is known for the very general class of \emph{regular} cut and project schemes (those whose windows have measure zero boundary) that the resulting patterns have pure point diffraction \cite{Sch00}. However, determining the nature of their complexity and repetitivity is more difficult. To make progress, it is natural to begin with the class of cut and project sets with convex polytopal windows, which includes a good number of examples of interest, such as the Ammann--Beenker \cite{Bee82} and the Penrose tilings \cite{Pen79,GruShe87}. We note that for the Penrose tilings, the window may be taken as the projection of the unit hypercube in $\R^5$, although in this case the scheme has non-dense projected lattice in the internal space. However, a more suitable model set scheme with dense projection is available \cite{AOI,BKSZ90}, for which the window may be taken as a union of pentagons, and this produces cut and project schemes which are mutually locally derivable (MLD, see \cite{AOI}) to those with a single decagon window. In Part I of the current work \cite{I} a general formula for the polynomial growth rate of the complexity function was given in terms of the cut and project data (see Theorem \ref{thm: generalised complexity} below), building on the work of \cite{Jul10}. In this paper we settle the problem of determining exactly which such polytopal cut and project sets, satisfying one extra mild restriction ({\it weak homogeneity}, see Definition \ref{def: weakly homogeneous}), satisfy linear repetitivity.

It has been known for some time in the $2$-to-$1$ canonical case that linear repetitivity is equivalent to the slope of the physical space being \emph{badly approximable} or, equivalently, having continued fraction expansion with bounded entries \cite{MH38, MH40}. Higher dimensional generalisations of this have been studied from the perspective of substitutions (see, for example \cite{BF11}). More generally, the case of `cubical cut and project sets', was solved in \cite{HaynKoivWalt2015a}. However, the class of patterns given by cubical windows is somewhat limited. Whilst the techniques of that paper could be pushed in an ad-hoc way to treat some non-cubical examples (see \cite[Section 7]{HaynKoivWalt2015a}), it remained unclear how to deal with general polytopal schemes.

In this paper we give a complete characterisation of linear repetitivity for all polytopal cut and project sets which also satisfy a property which we call \emph{weakly homogeneous}. The result is sharp, in the sense that it fails without assuming this property, see Section \ref{sec: codimension 1} (although in the same section we show that there are linearly repetitive examples which are not weakly homogeneous). Our characterisation is in terms of two properties {\bf C} (low complexity) and {\bf D} (Diophantine), which shall be explained in more detail shortly.

\begin{Mainthm} \label{thm: main}
For an aperiodic, polytopal, weakly homogeneous cut and project scheme, the following are equivalent:
\begin{enumerate}
	\item {\bf LR};
	\item {\bf C} and {\bf D}.
\end{enumerate}
\end{Mainthm}

Property {\bf C} is the requirement that the complexity function satisfies $p(r) \ll r^d$, where $d$ is the `physical' dimension of the patterns. As already mentioned, the complexity exponent may be determined from the cut and project data in the way described in Part I \cite{I}, from the hyperplane stabiliser ranks and the dimensions of their linear spans. We derive further important consequences of this condition in Section \ref{sec: minimal complexity}. In particular, we show that property {\bf C} implies a \emph{hyperplane spanning} condition of the pattern (see Definition \ref{def: hyperplane spanning} and Theorem \ref{thm: C => hyperplane spanning}).

The Diophantine condition {\bf D} (see Definition \ref{def: diophantine lattice}) is a higher dimensional analogue of the projected lattice being \emph{badly approximable}. Loosely speaking, it says that any lattice point $\gamma \in \Gamma$ in the total space does not get too close to the physical space $\phy$, relative to the point's distance $\|\gamma\|$ from the origin. Equivalently, any point of the projected lattice $\Gamma_<$ in the internal space must be reasonably far from the origin, relative to the norm of the lattice point it projected from.

It is well-known for general patterns (not just cut and project sets) that property {\bf C} is necessary for {\bf LR}. So our main theorem shows that property {\bf D} is the crux to establishing linear repetitivity for this class of patterns: if {\bf C} is assumed (as well as weak homogeneity) then it is surprisingly \emph{only the projected lattice} $\Gamma_<$, and none of the remaining data, which determines whether or not the scheme is {\bf LR}. So to find linearly repetitive schemes, one needs to look for Diophantine projections of lattices. The next step is to identify codimension $1$ subspaces in internal space which intersect $\Gamma_<$ with high rank. If there are enough such subspaces to constitute the boundary of a polytope, with the subspaces intersecting the projected lattice with sufficiently large ranks, then translates of these subspaces can be used as supporting hyperplanes of windows of {\bf LR} cut and project schemes.

\subsection{Further details on properties {\bf C} and {\bf D}}

To precisely quantify what is meant by a projected lattice point being `close' to the origin, one needs to use the correct Diophantine exponent. When the cut and project set has what is known as `constant hyperplane stabiliser ranks', property {\bf D} says the following:

\emph{There is some $c > 0$ so that for all $\gamma \in \Gamma \setminus \{0\}$ we have}
\[
\|\gamma_<\| \geq c \cdot \|\gamma\|^{-\delta}, \text{ where } \delta = \frac{d}{n}.
\]
In the above, $d$ is the `dimension' $\dim(\phy)$ and $n$ is the `codimension' $\dim(\intl)$. The notation of $\Gamma$ for the lattice in the total space $\tot$, $\phy$ for the physical space, $\intl$ for the internal space and $\gamma_<$ for the projection of $\gamma$ to $\intl$ follows the notation of Part I \cite{I}. The Ammann--Beenker example below is a cut and project set for which the condition of constant hyperplane stabiliser ranks holds. 

\begin{example}

\begin{figure}
    \centering
        \includegraphics[width=0.8\textwidth]{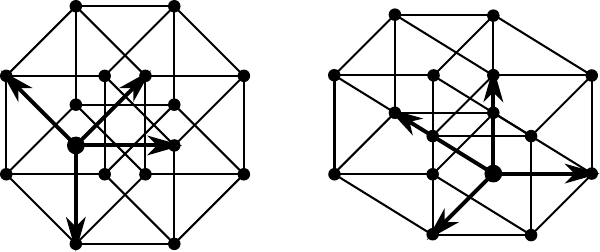}
    \caption{Windows for the Ammann--Beenker (left) and Golden Octagonal (right) cut and project schemes. Dots and edges are the projections to $\intl$ of the vertices and edges, respectively, of the unit hypercube in $\R^4$. The origin is indicated with a larger dot, along with the projections of the $4$ standard basis vectors.}
    \label{fig: octagonal}
\end{figure}

To illustrate briefly how the above theorem may be applied in practice, we consider the Ammann--Beenker cut and project scheme (which is known to satisfy {\bf LR} because it can also be generated through substitution). The octagonal window in the internal space is depicted in Figure \ref{fig: octagonal}, together with the origin and projections of the $4$ standard basis vectors of $\R^4$ in the total space, which generate the lattice $\Gamma$. There are $4$ distinct supporting hyperplanes up to translation, with each parallel subspace intersecting the projected lattice $\Gamma_<$ with rank $2$ (for example, the vertical subspace contains the projected lattice element pointing South, as well as the element given by the sum of those pointing NE and NW). Using the main result of \cite{I} (see Theorem \ref{thm: generalised complexity} below), or indeed the earlier result of Julien from \cite{Jul10}, it follows that these patterns satisfy {\bf C}, that is, $p(r) \ll r^2$. This also shows that the window has constant stabiliser rank, which also follows (Corollary \ref{cor: equal ranks}) from {\bf C} and the fact that the window is not indecomposable (Definition \ref{def: decomp}), which in codimension 2 simply means it is not a parallelogram (Example \ref{ex: decomp codim 2}).

Since the scheme has constant hyperplane rank (see Section \ref{sec: decompositions}), to establish {\bf D} one just needs to show that $\Gamma_<$ is Diophantine, as defined above. Given $\gamma_< \in \Gamma_<$, we have
\[
\gamma_< = \begin{pmatrix} (n_4-n_2) \sqrt{2} + n_1 \\ (n_2 +n_4)\sqrt{2} - n_3 \end{pmatrix},
\]
where the $n_i \in \Z$ are the coordinates of $\gamma \in \Gamma$ which projects to $\gamma_<$. Since $\sqrt{2}$ is badly approximable (it is a quadratic irrational, so has periodic and hence bounded continued fraction expansion), the above vector has norm at least $c/\|\gamma\|$, for $\gamma \neq 0$. This establishes that $\Gamma_<$ is Diophantine, so {\bf D} holds and the Ammann--Beenker tilings are linearly repetitive by Theorem \ref{thm: main}. More details of this computation, and similar ones such as for the \emph{golden octagonal tilings} \cite{BedFer} (with window on the right of Figure \ref{fig: octagonal}) are given in Section \ref{sec: octagonals}.
\end{example}

\subsection{Organisation of paper}
In Section \ref{sec: basic notions} we set out some important notation and recall the basic setup for polytopal cut and project schemes. We also introduce here some new conditions for a cut and project schemes important for our main results: being \emph{hyperplane spanning} (Definition \ref{def: hyperplane spanning}), \emph{homogeneous} (Definition \ref{def: homogeneous}) and \emph{weakly homogeneous} (Definition \ref{def: weakly homogeneous}). The weakly homogeneous condition is not necessary for one direction of our main theorem, in showing that {\bf LR} implies {\bf C} and {\bf D}, but it is needed in the other direction. The hyperplane spanning condition will be used extensively, especially when we come to analyse the shapes of cut regions in Section \ref{sec: vertices}. However, we are able to remove this assumption from our main theorem by establishing that it always holds for cut and project schemes satisfying {\bf C}, in Theorem \ref{thm: C => hyperplane spanning}. This is done in Section \ref{sec: minimal complexity}, which contains results on low complexity schemes.

Unlike in the Ammann--Beenker example above, cut and project sets can behave like sums of schemes, with factors of varying dimension and codimension. Because of such examples, finding the correct Diophantine property sometimes requires decomposing the scheme into factors. After decomposing enough, the window always splits to subsystems with constant stabiliser ranks, allowing us to derive the correct Diophantine exponents for them.  We develop a theory for cut and project scheme decompositions in Section \ref{sec: decompositions}. This extends work of Forrest, Hunton and Kellendonk \cite{FHK02}, who first identified one definition of such a decomposition. We provide a similar, equivalent definition, which does not require a choice of inner product (as well as defining the notion of a decomposition of more components). We also demonstrate that there is an equivalent geometric description: a decomposition corresponds to a direct sum decomposition of the window. We also give a simpler proof (using a different combinatorial graph associated to the flags) that when {\bf C} is satisfied, indecomposable windows have constant stabiliser rank (Corollary \ref{cor: equal ranks}). The geometric factorisation of the window is difficult to apply in isolation, since the dynamics under consideration rely on the interaction between the acting lattice and the window. Fortunately, when property {\bf C} is assumed and the window has a decomposition, the lattice must also split compatibly as a direct sum, up to finite index (Theorem \ref{thm: sum decomposition}). For a summary of these results, see Corollary \ref{cor: decompose}.

In Section \ref{sec: accs and cuts in subsystems} we show how properties of the cut and project scheme relate to those of the subsystems of a decomposition. In particular, we show how the cut regions and acceptance domains are related to products of those in the subsystems (Lemmas \ref{lem: cut region products} and \ref{lem: acceptance domain products}).

In Section \ref{sec: diophantine} we introduce the Diophantine condition {\bf D}. We develop a base-free approach for a Diophantine condition for a general, densely embedded lattice. It is analogous to the badly approximable condition for a real number or, more generally, a set of vectors in some Euclidean space (see Example \ref{exp: badly approximable}). Before applying the condition to cut and project schemes, we establish its basic properties and implications, such as that of `transference', that is, that Diophantine lattices fill space densely (Theorem \ref{thm: transference}).

In Section \ref{sec: LR => C and D} we prove that {\bf LR} implies {\bf C} and {\bf D}. This is the easier direction of the main theorem, in part because only consideration of patch frequencies (but not uniform spacing of patches) is required. We recall the \emph{positivity of weights} ({\bf PW}) condition (see \cite{BBL}) and show that, in fact, {\bf PW} is sufficient for {\bf C} and {\bf D}. This is done by constructing patches of low frequency if either of these conditions fails; in the case that {\bf C} fails, the growth rate of patches is too high for all patches to have low frequencies, and if {\bf D} fails then we use short projections of lattice vectors to construct small acceptance domains.

In Section \ref{sec: vertices} we establish a close connection between the set of possible locations of vertices of cut regions with the projected lattice, for (weakly) homogeneous schemes. This has the implication that if the projected lattice vectors cannot be too small, then neither are displacements between vertices of cut regions. Hence, weakly homogeneous cut and project schemes satisfying {\bf C} and {\bf D} have `large' acceptance domains. In Section \ref{sec: C and D => LR}, we combine this observation with transference from Section \ref{sec: diophantine}, to ensure that sufficiently large orbits are dense enough in internal space to hit all required acceptance domains, establishing that {\bf C} and {\bf D} together imply {\bf LR} for weakly homogeneous schemes.

Finally, in Section \ref{sec: further results and examples}, we consider some implications of our main theorem. We also demonstrate it on a range of examples.

\filbreak
\subsection{Summary of notation}
\begin{itemize}
	\item $\mathcal{S} = (\tot,\phy,\intl,\Gamma,W)$ : Euclidean cut and project scheme $\mathcal{S}$, with total space $\tot$ of dimension $k$, physical space $\phy$ of dimension $d$, complementary internal space $\intl$, lattice $\Gamma \leqslant \tot$ and window $W \subset \intl$ (always here a convex polytope).
	\item $\pi_\vee \colon \tot \to \phy$ : the projection from the total space to the physical space with respect to $\intl$. For $x \in \tot$ we let $x_{\vee} \coloneqq \pi_{\vee}(x)$. Analogous notation for $\intl$ in place of $\phy$.
	\item $y^\wedge$ : for $y = \pi_\vee(\gamma)$, with $\gamma \in \Gamma$, we let $y^\wedge \coloneqq \gamma$.
	\item $y^*$ : the star map applied to $y \in \Gamma_{\vee}$ is $y^* \coloneqq (y^\wedge)_<$.
	\item $\cps$ : the cut and project set associated to the scheme, $\cps \coloneqq \{y \in \Gamma_{\vee} \mid y^* \in W\}$.
	\item $A_P$ : the acceptance domain in $W$ associated to (pointed) patch $P$ of $\cps$.
	\item $\Gamma(r)$: those lattice points that are within radius $r$ from the origin. 
	\item $\sA(r)$ : the set of acceptance domains of patches of radius $r$.
	\item $\sH^+$ : the set of half spaces defining a polytopal window $W$, with associated collection of affine hyperplanes $\sH$, called the \emph{supporting hyperplanes} of the window.
	\item $V(H)$ : the linear subspace given by translating an affine hyperplane $H$ over the origin.
	\item $\sH_0$ : the supporting hyperplanes translated over the origin, that is, $\sH_0 = V(\sH)$.
	\item $\Gamma^H$ : the stabiliser of $H$ is the subgroup $\Gamma^H \coloneqq \{\gamma \in \Gamma \mid H = H + \gamma_<\}$ of $\Gamma$. More generally, given $I \subseteq \sH$, we have $\Gamma^I \coloneqq \{\gamma \in \Gamma \mid H = H + \gamma_< \text{ for all } H \in I\} = \bigcap_{H \in I} \Gamma^I$.		
	\item $(X_i, \Gamma_i, W_i)$: the scheme splits into $m$ subsystems $(X_i, \Gamma_i, W_i)$, where $\intl=X_1+\cdots+ X_m$, $W=W_1+\cdots+W_m$, $\Gamma_1+ \cdots + \Gamma_m$ is finite index in $\Gamma$ and each subsystem has constant hyperplane stabiliser rank. See Section \ref{sec: decompositions} and in particular Corollary \ref{cor: decompose}.
	\item $\mathcal{V}(G)$ : the set of `vertex' intersections of $G_<$ translates of hyperplanes, where $G \subseteq \Gamma$.
	\item $\mathcal{V}(G,f)$ : those vertices coming from a flag $f$ of hyperplanes. Similar notation $\mathcal{V}_i(G)$ and $\mathcal{V}_i(G,f)$ used for analogous sets in the subsystems coming from a decomposition. See Definition \ref{def: vertex sets}.
	\item $f \ll g$ : for two functions $f,g \colon \R_{>0} \to \R_{>0}$, we let $f \ll g$ mean that there exists some constant $C > 0$ for which $f(x) \leq C g(x)$ for sufficiently large $x$. That is, $f \in O(g)$.
	\item $f \asymp g$ : for two functions as above, $f \ll g$ and $g \ll f$.
\end{itemize}

\section{Basic notions for polytopal cut and project set schemes} \label{sec: basic notions}
A polytopal {\bf $k$-to-$d$ cut and project scheme} is a tuple
\[
\mathcal{S} = (\tot,\phy,\intl,\Gamma,W),
\]
of a {\bf total space}, {\bf physical space}, {\bf internal space}, {\bf lattice} and {\bf window}, respectively. The total space $\tot$ is a $k$-dimensional $\R$-vector space, with $\phy$ and $\intl$ subspaces of $\tot$ of dimensions $d$ and $n=k-d$, respectively. We assume that $d$, $n > 0$, and that $\phy$, $\intl$ are complementary, giving a direct sum decomposition $\tot = \phy + \intl$ of the total space. The lattice $\Gamma \leqslant \tot$ is a free Abelian group of rank $k$ which is co-compact in $\tot$. The window $W \subset \intl$ is chosen to be a compact, convex polytope, so it is given as the intersection of a uniquely defined irredundant set of half-spaces $\sH^+$. The set of associated codimension $1$ {\bf supporting hyperplanes} (the boundaries of the half spaces of $\sH^+$) is denoted by $\sH$.

By a {\bf hyperplane}, we mean a codimension $1$ affine hyperplane (and say `codimension $1$ subspace' if we wish to emphasise that the hyperplane contains the origin). Given a hyperplane $H$, we let $V(H)$ denote the codimension $1$ subspace given by shifting $H$ over the origin, that is $V(H) = H-H = H-h$ for any $h \in H$. By a slight abuse of notation, we sometimes use $V$ (and variations) as also the name of some codimension $1$ subspace; we use $H$ (and variations) for affine hyperplanes. We introduce the notation $\sH_0 = V(\sH)$ for the set of subspaces parallel to the supporting hyperplanes.

We recall the notation $\pi_\vee \colon \tot \to \phy$ for the projection to the physical space, with respect to the direct sum decomposition $\tot = \phy + \intl$. For a point $X \in \tot$ (or a subset $X \subseteq \tot$), we let $X_\vee \coloneqq \pi_\vee(X)$. Analogous notation is used for projections to the internal space. This notation is remembered by visualising a $2$-to-$1$ scheme in the first quadrant, where we project `down' $\vee$ to the physical space, and `left' $<$ to the internal space.

We assume that $\pi_\vee$ is injective on $\Gamma$ and that $\Gamma_<$ is dense in $\intl$. In addition, {\bf we assume throughout that $\mathcal{S}$ produces aperiodic patterns}. Equivalently, $\pi_<$ is injective on $\Gamma$, that is, $\rk(\Gamma_<) = k$. Given $x \in \Gamma_\vee$, we denote the lift of $x$ back to the lattice by $x^\wedge$, which is the unique point with $(x^\wedge)_\vee = x$. We have the {\bf star map} $x \mapsto x^*$, defined by $x^* \coloneqq (x^\wedge)_<$, mapping $\Gamma_\vee$ into $\intl$.

The cut and project set produced by $\mathcal{S}$ is denoted
\[
\cps \coloneqq \{x \in \Gamma_\vee \mid x^* \in W\}.
\]
The above is appropriate only when the scheme $\mathcal{S}$ is {\bf non-singular}, meaning that $x^* \notin \partial W$ for all $x \in \Gamma_\vee$ (equivalently, $\Gamma \cap (\phy + \partial W) = \emptyset$ or $(\Gamma_<) \cap \partial W = \emptyset$). We may always translate the window to make it non-singular, and any non-singular scheme will generate cut and project sets with the same sets of finite patches, by density of $\Gamma_<$. In particular, we need only consider the above choice of $\Lambda$ in determining whether or not the family of cut and project sets produced by $\mathcal{S}$ are linearly repetitive (this is the family of patterns given similarly using non-singular translates of the lattice, and singular patterns given as limits of non-singular ones).

It will be important to keep track of the subgroups of $\Gamma$ which stabilise various sets of supporting hyperplanes. Given $I \subseteq \sH$, we let
\[
\Gamma^I \coloneqq \{\gamma \in \Gamma \mid H + \gamma_< = H \text{ for all } H \in I\}.
\]
It is easy to see that each $\Gamma^I$ is a subgroup of $\Gamma$ and we denote the rank $\rk(\Gamma^I)$ by $\rk(I)$. For a singleton, we simplify notation to $\Gamma^H \coloneqq \Gamma^{\{H\}}$ and $\rk(H) \coloneqq \rk(\{H\})$. Clearly a translation stabilises a hyperplane $H$ if and only if it stabilises any other translate $H+x$, for example $V(H)$. So we may write
\[
\Gamma^I = \{\gamma \in \Gamma \mid V + \gamma_< = V \text{ for all } V \in V(I)\} = \{\gamma \in \Gamma \mid \gamma_< \in V \text{ for all } V \in V(I)\}.
\]
Hence, after projecting $\Gamma^I$ to the internal space (which determines $\Gamma^I$, by injectivity of $\pi_<$), we may alternatively describe it as the subset of $\Gamma_<$ contained in the subspace intersection of all $V \in V(I)$:

\begin{lemma} \label{lem: higher stabilisers}
For $I \subseteq \sH$, $\Gamma^I$ is a subgroup of $\Gamma$, whose projection to $\intl$ is $\Gamma^I_< = \Gamma_< \cap X_I$, for the subspace $X_I = \bigcap_{H \in I} V(H)$. \qed
\end{lemma}

\subsection{Patches, acceptance domains and cut regions} \label{sec: patches and acceptance domains}
The {\bf $r$-patch} centred at $y \in \Lambda$ is the subset $P(y,r) \coloneqq B_r(y) \cap \Lambda$, marked with the centre $y$. We typically consider $r$-patches up to translation equivalence. The {\bf complexity function} $p \colon \R_{>0} \to \N$ is defined by letting $p(r)$ be the number of $r$-patches, up to translation equivalence. Again, this does not depend on the choice $\cps$ of cut and project set produced by the scheme $\mathcal{S}$, as long as $\mathcal S$ is non-singular.

One of the main results of \cite{I} was that for any polytopal scheme, we have $p(r) \asymp r^\alpha$, for some $\alpha \in \N$ which can be determined by the ranks of the subgroups $\Gamma^H$, for $H \in \sH$ and the dimensions of their linear spans (in particular, we recall this as Theorem \ref{thm: generalised complexity} later). By \cite[Corollary 6.2]{I}, $\alpha \leq nd$ (where $n = k-d = \dim \intl$ denotes the `codimension'). We define the following property of minimal complexity:

\begin{definition}
We say that the scheme $\mathcal{S}$ satisfies property {\bf C}, of having {\bf minimal complexity}, if $p(r) \ll r^d$, that is, there exists some $C > 0$ for which $p(r) \leq C r^d$ for all $r\ge 1$. 
\end{definition}

We will see in Corollary \ref{cor: minimal complexity} that $\alpha \geq d$ for an aperiodic scheme $\mathcal{S}$, so we always have $p(r) \gg r^d$ and hence {\bf C} implies that $p(r) \asymp r^d$. We shall explore further useful consequences of this condition in Section \ref{sec: minimal complexity}.

Given some $r > 0$, we consider the lattice points $\gamma \in \Gamma$ for which $\|\gamma_\vee\|$, $\|\gamma_<\| \leq r$. Without loss of generality, we can choose a metric on $\tot$ so that, equivalently, $\|\gamma\| \leq r$ (and even if we do not, we have such an identification up to linear rescalings of $r$, which is all that is needed for our results here). Consider the elements of $\Gamma(r) \coloneqq B_r(0) \cap \Gamma$, and further, denote $\Gamma_<(r) \coloneqq \Gamma(r)_<$. 

Recall that to each $r$-patch $P = P(y,r)$, there is a subset $A_P \subseteq W$, called here an {\bf $r$-acceptance domain}, such that $P(y,r) = P(z,r) + (y-z)$ if and only if $z^* \in A_P$. In particular, the $r$-acceptance domain $A\subset W$ containing the point $z^*\in W$ (for $z^* \notin \partial W +\Gamma_<$) is the intersection of those $\Gamma_<(r)$ translates of the  interior $\iW$ and complement $W^c$ that contain $z^*$ \cite[Corollary 3.2]{I}. The set of $r$-acceptance domains is denoted $\sA(r)$.

We also have the {\bf cut regions}, defined as follows. Consider the set $\sH + \Gamma_<(r)$ of translates of the supporting hyperplanes. Removing them from the window splits it into convex connected components, called {\bf $r$-cut regions}. The set of $r$-cut regions is denoted $\sC(r)$. Even without further assumptions on the polytopal scheme, we have the following:

\begin{lemma}[Corollary 4.3 of \cite{I}] \label{lem: cuts refine acc}
For sufficiently large $r$, the $r$-cut regions refine the $r$-acceptance domains. That is, for all $A \in \sA(r)$, with $r$ sufficiently large, there exists some $C \in \sA(r)$ with $C \subseteq A$; equivalently, if the $r$-patches at $x$, $y \in \cps$ disagree then $x^*$ and $y^*$ are separated by a hyperplane $H+\gamma_<$, with $H \in \sH$ and $\gamma \in \Gamma \cap B_r(0)$.
\end{lemma}

We will use the above in showing that conditions {\bf C} and {\bf D} together imply {\bf LR}: since cut regions refine acceptance domains, it will be sufficient to show that all cut regions are visited frequently to derive the same for the acceptance domains, and hence that all patches appear frequently enough in our patterns. To obtain a kind of converse to the above lemma (up to linear rescalings of $r$), so as to refine cut regions by acceptance domains, one needs to assume further conditions, a very general one being the \emph{quasicanonical} condition introduced in \cite{I}. For our purposes, however, this will not be necessary, and we work directly with the acceptance domains for the converse direction of our main theorem.

\subsection{Linear repetitivity}
For $r > 0$, we define $\rho(r)$ to be the smallest value of $R$ for which, for any $z \in \Lambda$ and any $r$-patch $P$ in $\Lambda$, we have that $P = P(y,r)$, up to translation, for some $y \in B_R(z)$. That is, all $r$-patches which appear somewhere in $\Lambda$, in fact appear within radius $\rho(r)$ of any point of $\Lambda$, up to translation.

We say that $\Lambda$ is {\bf linearly repetitive} if $\rho(r) \ll r$, that is, there exists some $C > 0$ such that $\rho(r) \leq Cr$ for sufficiently large $r$. If $\Lambda$ is linearly repetitive, then so is any other cut and project set in the family produced by the scheme $\mathcal{S}$, and we say that $\mathcal{S}$ has property {\bf LR}.

\subsection{New conditions}
We introduce here some new conditions on the scheme $\mathcal{S}$ which will be useful in establishing our main theorem. The following will be vital in Section \ref{sec: vertices} when we come to compare projected lattice points with vertices of acceptance domains, as well as calculating ranks of hyperplane stabiliser groups for low complexity schemes. It will be shown in Theorem \ref{thm: C => hyperplane spanning} that this condition always holds if $\mathcal{S}$ satisfies {\bf C}.

\begin{definition}\label{def: hyperplane spanning}
Call $\mathcal{S}$ {\bf hyperplane spanning} if, for each $H\in \sH$, the linear span of $\Gamma_<^H$ is of dimension $n-1$, that is, $\langle \Gamma^H_< \rangle_\R = V(H)$. 
\end{definition}

The next property asks that all supporting hyperplanes are well-positioned with respect to the lattice, in the sense that they may be translated by lattice elements over a common origin. The term \emph{homogeneous} is motivated by the codimension 1 case (see Subsection \ref{sec: codimension 1}), where the condition allows us to analyse sizes of acceptance domains via homogeneous (rather than inhomogeneous) Diophantine properties of the lattice.

\begin{definition} \label{def: homogeneous} Call $\mathcal{S}$ {\bf homogeneous} if there exists some `origin' $o \in \intl$ for which, for each $H \in \sH$, there exists some $\gamma_H \in \Gamma$ so that $o \in H+(\gamma_H)_<$. \end{definition}

Obviously the above property is not affected by translating the window. So an equivalent condition is that, for some translate of the window, $H \cap \Gamma_< \neq \emptyset$ for each $H \in \sH$, as seen by translating $W$ by $-o$, which re-centres at the origin. For convenience, we shall usually take $o$ to be the origin of $\intl$ when assuming homogeneity, which simplifies proofs. Whilst this technically makes the scheme singular with respect to the untranslated lattice, we ultimately only require homogeneity to derive affine conditions of the window, which do not depend on a choice of `origin', such as sizes of acceptance domains.

For the statement of our main theorem, we only need the following weaker version of homogeneity. As above, the property is invariant up to translations of the window, so we may always take $o$ as the origin after an appropriate translation of $W$.

\begin{definition} \label{def: weakly homogeneous}
Call $\mathcal{S}$ {\bf weakly homogeneous} if there exists some `origin' $o \in \intl$ for which, for each $H \in \sH$, there exists some $\gamma_H \in \Gamma$ and $n_H \in \N$ so that $o \in H+(1/n_H)\cdot(\gamma_H)_<$.
\end{definition}

\begin{lemma} \label{lem: hom vs weak hom}
The scheme $\mathcal{S}$ is weakly homogeneous if and only if there exists some $N \in \N$ so that the scheme $\mathcal{S}'$, given by replacing $\Gamma$ with $(1/N) \cdot \Gamma$, is homogeneous.
\end{lemma}

\begin{proof}
Suppose that $\mathcal{S}'$ is homogeneous, and take $H \in \sH$. Then there exists $(1/N)\cdot \gamma_H \in (1/N) \cdot \Gamma$ with $o \in H + ((1/N) \cdot \gamma_H)_< = H + (1/N)\cdot (\gamma_H)_<$. It follows that $\mathcal{S}'$ is weakly homogeneous, taking each $n_H = N$.

Conversely, suppose that $\mathcal{S}$ is weakly homogeneous, and define $N \coloneqq \prod_{H \in \sH} n_H$. For each $H \in \sH$, there exists some $\gamma_H \in \Gamma$ so that $o \in H+(1/n_H)\cdot(\gamma_H)_< = H + ((1/N) \cdot (N/n_H) \cdot \gamma_H)_<$. Since $(1/N) \cdot (N/n_H) \cdot \gamma_H \in (1/N) \cdot \Gamma$ for each $H \in \sH$, we see that $\mathcal{S}'$ is homogeneous.
\end{proof}

\section{Minimal complexity schemes} \label{sec: minimal complexity}
In this section we determine consequences of the condition {\bf C} of minimal complexity. In \cite{I} we determined the exponent of the complexity function for general polytopal cut and project schemes. We recall the following definition, followed by the result \cite[Theorem 7.1]{I} (where we simplify the original notation from $\alpha'$ and $\alpha_f'$ to $\alpha$ to $\alpha_f$):

\begin{definition} \label{def: flag} A set $f = \{H_1,\ldots,H_n\}$ of affine, codimension $1$ hyperplanes is called a {\bf flag} if their intersection is a singleton.
\end{definition}

\begin{theorem} \label{thm: generalised complexity} Consider the collection $\sF$ of flags of the supporting hyperplanes $\sH$ of a given polytopal cut and project scheme. Then $p(r) \asymp r^\alpha$ where
\[
\alpha \coloneqq \max_{f \in \sF} \alpha_f, \ \ \mathrm{for} \ \ \alpha_f \coloneqq \sum_{H \in f} (d - \rk(H) + \beta_H), \ \text{and} \ \ \beta_H \coloneqq \dim \langle \Gamma_<^H \rangle_\R.
\]
\end{theorem}

Since $\Gamma_<^H \leqslant V(H)$, and the latter is a codimension $1$ subspace of $\intl$, we have that $\beta_H \leq n - 1$, with equality for each $H \in \sH$ if and only if $\mathcal{S}$ is hyperplane spanning. By the following lemma, $\alpha \geq d$ if the scheme is hyperplane spanning (recall that we always assume that $\mathcal{S}$ is aperiodic).

The main theorem of this section is that if $\mathcal{S}$ satisfies property {\bf C} (that is, $\alpha \leq d$) then the scheme is hyperplane spanning, and hence $\alpha = d$. This is to be expected, since a scheme is of low complexity when the stabilisers $\Gamma^H$ are `large', and one would hope therefore at least spanning. The difficulty is that a priori it is feasible that each $\Gamma^H$ could have large rank but possibly concentrated in lower dimensional subspaces. Indeed, schemes can be chosen so that this happens for several hyperplanes $H$. However, we will show (loosely speaking) that this cannot happen for enough hyperplanes as to constitute a flag. To establish this, we will need a few technical results.

\begin{lemma} \label{lem: minimum complexity}
Suppose that $\mathcal{S}$ is hyperplane spanning. Then $\alpha_f \geq d$ for each $f \in \sF$. In particular $\alpha \geq d$.
\end{lemma}

\begin{proof}
Let $f \subseteq \sH_0$ be any flag and consider the group
\[
T \coloneqq \bigoplus_{H \in f} (\Gamma / \Gamma^H).
\]
There is a natural homomorphism $\Gamma \to T$, given by the diagonal embedding followed by taking the quotients by $\Gamma^H$ in each component. This map is an injection. Indeed, suppose that $\gamma \in \Gamma$ maps to $0$ in each component, so that $\gamma \in \Gamma^H$ for all $H \in f$; equivalently $\gamma_< \in V$ for all $V \in V(f)$. Since $f$ is a flag, the only point in every $V \in V(f)$ is the origin, so $\gamma_< = 0$. By aperiodicity, we have that $\gamma = 0$.

It follows that $k = \rk(\Gamma) \leq \rk(T) = \sum_{H \in f} (k - \rk(H))$. By hyperplane spanning, each $\beta_H = n-1$. Hence
\[
\alpha_f = \sum_{H \in f} (k - \rk(H) - 1) = \left(\sum_{H \in f} (k - \rk(H))\right) - n = \rk(T)-n \geq k-n = d,
\]
as required. Since each $\alpha_f \geq d$, we have that $\alpha = \max_{f \in \sF} \alpha_f \geq d$.
\end{proof}

The next result provides a procedure with which to modify a scheme to make it hyperplane spanning without increasing the complexity exponent:

\begin{lemma} \label{lem: beta increase}
Suppose that $\mathcal{S}$ is not hyperplane spanning, so there exists some $H \in \sH$ with $\beta_H \neq n-1$. Then we may replace $H$ with some hyperplane $H'$, obtaining a new scheme $\mathcal{S}'$ with supporting hyperplanes $\sH' = (\sH \setminus \{H\}) \cup \{H'\}$, so that:
\begin{enumerate}
	\item $\Gamma^H < \Gamma^{H'}$;
	\item each $\alpha_f$ satisfies $\alpha_f\ge \alpha_{f'}$;
	\item $\beta_{H'}=\beta_H+1$.
\end{enumerate}
Here, for a flag $f\subset \sH$, the flag $f'\subset \sH'$ has the same hyperplanes as $f$ except that $H$ is replaced by $H'$. 
\end{lemma}

\begin{proof}
Let $X \coloneqq \langle \Gamma^H \rangle_\R$. By assumption, $\dim(X) = \beta_H < n-1$. Choose $x \in H$ and some $\gamma \in \Gamma_<$ not in $X$ which is `close' to $V(H) = H-x$ (which we make precise later). The subspace $Y \coloneqq \langle X, \gamma \rangle_\R$ is of dimension $\beta_H + 1$. Let $\Gamma_<^Y \coloneqq \Gamma_< \cap Y$. Then $\rk(\Gamma_<^Y) > \rk(\Gamma_<^H)$, since $\Gamma_<^Y$ contains all lattice elements of $\Gamma_<^H = \Gamma_< \cap (H-x)$, in addition to the ones coming from the inclusion of $\gamma$. Choose vectors $v_1$, $v_2$, \ldots, $v_\ell$, with $\ell = (n-1) - (\beta_H + 1)$, which complete $Y$ to a subspace
\[
V \coloneqq \langle Y, v_1, v_2, \ldots v_\ell \rangle_\R
\]
which is of dimension $n-1$ and contains no lattice points other than those already contained in $Y$ (since $\Gamma_<$ is countable, there is a dense set of choices for the $v_i$). We let $H' \coloneqq V+x$. Since $\Gamma_<$ is dense in $\intl$, we may choose $\gamma$, $v_1$, \ldots, $v_\ell$ so that the above conditions hold, and so that $H'$ is `close' to $H$, in the sense that $\sH'$ is still the set of supporting hyperplanes of some new window, and so that flags of $\sH$ are the same as flags in $\sH'$, up to swapping $H$ with $H'$ (since being a flag is an open condition on the set of hyperplanes).

By construction, $\Gamma^H < \Gamma^Y = \Gamma^{H'}$ and $\beta_{H'} = \beta_H + 1$. Moreover, since $\Gamma_<^{H'} \geqslant \langle \Gamma_<^H, \gamma \rangle_\R$, we see that $\rk(H') \geq \rk(H) + 1$. Hence, by the definition of $\alpha_f$ in Theorem \ref{thm: generalised complexity}, for any flag $f\subset \sH$ it is the case that $\alpha_f\ge \alpha_{f'}$, since any appearance of $H'$ which increases the corresponding $\beta_H$ by $1$ is at least offset by replacing $\rk(H)$ by $\rk(H')$.
\end{proof}

Together with Lemma \ref{lem: minimum complexity}, the above establishes the following:

\begin{corollary} \label{cor: minimal complexity}
Each $\alpha_f \geq d$ in Theorem \ref{thm: generalised complexity}. In particular, the complexity exponent $\alpha \geq d$.
\end{corollary}

\begin{proof}
Iteratively apply Lemma \ref{lem: beta increase} until $\beta_H = n-1$ for each hyperplane, which does not increase any $\alpha_f$. In this way we can always find a new scheme $\mathcal{S'}$ which is hyperplane spanning and has complexity exponents no higher than the original scheme. Then the result follows from Lemma \ref{lem: minimum complexity}.
\end{proof}

The following simple lemma will be particularly useful for low complexity schemes, when calculating the ranks of the $\Gamma^I$:

\begin{lemma} \label{lem: lattice ranks}
For $I$, $J \subseteq \sH$ we have that
\[
\rk(I \cup J) \geq \rk(I) + \rk(J) - \rk(I \cap J),
\]
with equality if and only if $\Gamma^I + \Gamma^J$ is a finite index subgroup of $\Gamma^{I \cap J}$.
\end{lemma}

\begin{proof}
We have a short exact sequence
\[
0 \to \Gamma^I \cap \Gamma^J \to \Gamma^I \times \Gamma^J \to \Gamma^I + \Gamma^J \to 0,
\]
where the second arrow is the map $x \mapsto (x,-x)$ and the third is $(x,y) \mapsto x+y$. By Lemma \ref{lem: higher stabilisers} we have that
\[
\Gamma^I \cap \Gamma^J = (\Gamma \cap V^I) \cap (\Gamma \cap V^J) = \Gamma \cap (V^I \cap V^J) = \Gamma \cap V^{I \cup J} = \Gamma^{I \cup J}.
\]
Since $\rk(\Gamma^I \times \Gamma^J) = \rk(I) + \rk(J)$, it follows from the above short exact sequence that
\[
\rk(I \cup J) = \rk(I) + \rk(J) - \rk(\Gamma^I + \Gamma^J).
\]
Any $\gamma \in \Gamma^I + \Gamma^J$ certainly stabilises any hyperplane which is in both $I$ and $J$, that is $\Gamma^I + \Gamma^J \leqslant \Gamma^{I \cap J}$. In particular, $\rk(\Gamma^I + \Gamma^J) \leq \rk(I \cap J)$, with equality if and only if the inclusion is of a finite index subgroup, so the result follows.
\end{proof}

The following lemma shows that, for a hyperplane spanning scheme satisfying {\bf C}, the lattice $\Gamma$, up to finite index, splits as a sum of subgroups which project to $n$ complementary lines lying along the intersections of $n-1$ supporting subspaces of the window. The result below will also be useful later, so we note here that the hyperplane spanning condition may be dropped (since it follows from {\bf C}) once Theorem \ref{thm: C => hyperplane spanning} has been proven.

\begin{lemma} \label{lem: lattice spanned by lines}
Suppose that $\mathcal{S}$ is hyperplane spanning and satisfies {\bf C}. Choose any flag $f = \{V_1, V_2, \ldots, V_n\} \subseteq \sH_0$ and consider the subsets $\widehat{i} \subset f$ given by $\widehat{i} = f \setminus \{V_i\}$. Then $\Gamma^{\widehat{1}} + \Gamma^{\widehat{2}} + \cdots + \Gamma^{\widehat{n}}$ is a finite index subgroup of $\Gamma$. Each $\Gamma^{\widehat{i}}_<$ is contained in the $1$-dimensional subspace $L(i)$, given as the intersection of the $V \in \widehat{i}$; these lines give a direct sum decomposition of $\intl$.
\end{lemma}

\begin{proof}
By Lemma \ref{lem: higher stabilisers}, we have that $\Gamma^{\widehat{i}}_< = \Gamma_< \cap L(i)$, where $L(i)$ is the intersection of subspaces in $\widehat{i}$. The intersection of any subspace $X$ with a codimension $1$ subspace has dimension $\dim(X)$ or $\dim(X)-1$, so we see that, as an intersection of $n-1$ such subspaces, $\dim(L(i)) \geq 1$. We must have that $\dim(L(i)) \leq 1$, or else $\dim(L(i) \cap V_i) > 0$, contradicting $f$ being a flag, so each $L(i)$ is $1$-dimensional. The lines $L(i)$ are complementary. Indeed, suppose that $v_1 + \cdots + v_n = 0$, where each $v_i \in L(i)$. Then any given $v_i$ may be expressed as a linear sum of the other $v_j$. But each such $v_j \in V_i \supseteq L(j)$, which implies that $v_i \in V_i$. Hence, $v_i \in V_i \cap L(i) = \{0\}$, since $f$ is a flag, so each $v_i = 0$, as required.

To see that the $\Gamma^{\widehat{i}}$ have sum which is finite rank in $\Gamma$, we calculate their ranks using the previous Lemma. We let $r(i)$ denote the rank of $\Gamma^{V_i}$ and $r(123 \cdots i)$ denote the rank of $\Gamma^I$, with $I = \{V_1,V_2,\ldots,V_i\}$. We calculate
\begin{align*}
r(1)+r(2)+\cdots+r(n) & \leq r(12)  + r(3) + r(4) + \cdots + r(n) + k      \\
                      & \leq r(123) + r(4)        + \cdots + r(n) + 2k     \\
                      & \cdots                                      \\
                      & \leq r(123\cdots (n-1))            + r(n) + (n-2)k.
\end{align*}
In each step of the above, we use the fact that
\[
r(123 \cdots i) + r(i+1) \leq r(123 \cdots (i+1)) + \rk(\emptyset) = r(123 \cdots (i+1)) + k
\]
by taking $I = \{V_1, V_2, \ldots, V_i\}$ and $J = \{V_{i+1}\}$ in Lemma \ref{lem: lattice ranks}. By definition, $r(123 \cdots (n-1)) = \rk(\Gamma^{\widehat{n}})$ and we see that
\[
\sum_{V \in f} \rk(V) \leq \rk(\Gamma^{\widehat{n}}) + \rk(V_n) + (n-2)k.
\]
Analogously, by leaving out any other $i$ rather than $n$ we see that
\begin{equation} \label{eq: line ranks}
\sum_{V \in f} \rk(V) \leq \rk(\Gamma^{\widehat{i}}) + \rk(V_i) + (n-2)k.
\end{equation}

By Theorem \ref{thm: generalised complexity}, property {\bf C}, Corollary \ref{cor: minimal complexity} and hyperplane spanning (each $\beta_H = n-1$),
\begin{equation} \label{eq: low complexity}
\sum_{V \in f} (d - \rk(V) + (n-1)) = d,
\end{equation}
that is,
\[
\sum_{V \in f} \rk(V) = n \cdot (d+n-1) - d = n \cdot k - n - d = (n-1) \cdot k.
\]
Applying this to Equation \ref{eq: line ranks}, we have
\[
(n-1) \cdot k \leq \rk(\Gamma^{\widehat{i}}) + \rk(V_i) + (n-2)\cdot k, \text{ that is,} \ \rk(\Gamma^{\widehat{i}}) \geq k - \rk(V_i).
\]
Summing over each $i$ and applying Equation \ref{eq: low complexity} we obtain
\[
\sum_{i=1}^n \rk(\Gamma^{\widehat{i}}) \geq \sum_{V \in f} (k - \rk(V)) = \left(\sum_{V \in f} d - \rk(V) + (n-1)\right) + n = d+n=k.
\]
The subgroups $\Gamma^{\widehat{i}}$ must be linearly independent in $\tot$ (since they map to complementary lines in $\intl$). Since their ranks sum to  at least $k = \rk(\Gamma)$, their sum must in fact have rank $k$ and be finite index in $\Gamma$.
\end{proof}

We conclude this section with its main theorem:

\begin{theorem} \label{thm: C => hyperplane spanning}
If $\mathcal{S}$ satisfies {\bf C} then it is hyperplane spanning.
\end{theorem}

\begin{proof}
Suppose that $\mathcal{S}$ is not hyperplane spanning but satisfies {\bf C}. By Corollary \ref{cor: minimal complexity}, each $\alpha_f = d$. Iteratively applying Lemma \ref{lem: beta increase}, we may increase each $\beta_H$ whilst preserving the property that each $\alpha_f = d$. We do this for all except one hyperplane $K \in \sH$, which we adjust until $\beta_K = n-2$. In other words, this scheme (which we now refer to as $\mathcal{S}$) is `one move' from a hyperplane spanning scheme; applying Lemma \ref{lem: beta increase} one more time, we may replace $K$ for some hyperplane $K'$ to obtain a hyperplane spanning scheme satisfying {\bf C}, denoted $\mathcal{S}'$.

Take a flag $f$ for the scheme $\mathcal{S}'$ which contains $K'$. By {\bf C} and Corollary \ref{cor: minimal complexity},
\begin{equation} \label{eq: l complexity}
\alpha_f = \sum_{H \in f} d - \rk(H) + \beta_H = d,
\end{equation}
where each $\beta_H = n-1$ since $\mathcal{S}'$ is hyperplane spanning.

By the previous Lemma \ref{lem: lattice spanned by lines}, $\Gamma$ is the sum (up to finite index) of subgroups $\Gamma^{\widehat{i}}$, which project to complementary lines defined by $f$ in the internal space. Consider the subspaces $V$ and $V'$, given by translating $K$ and $K'$ over the origin, respectively. Their intersection $X = V \cap V'$ is an $(n-2)$-dimensional linear subspace of $V$ and contains all points of $\Gamma^K_<$. At least one line $L(i) \subseteq V'$, containing some $\Gamma^{\widehat{i}}$, is not contained in $X$ (since otherwise $K$ and $K'$ would be parallel and have the same stabilisers). It follows that, when replacing $K$ with $K'$, a subgroup of rank at least $\rk(\Gamma^{\widehat{i}})$ is removed, which is linearly independent from $\Gamma^K$.

So the claim follows if each $\rk(\Gamma^{\widehat{i}}) \geq 2$. Indeed, in that case, $\rk(K') \geq \rk(K) + 2$, whereas $\beta_{K'} = \beta_K + 1$, so the complexity exponent $\alpha_f$ in Equation \ref{eq: l complexity} increases by at least $1$ when replacing $K'$ with $K$, hence the complexity exponent for $\mathcal{S}$ is strictly greater than $d$, contradicting it satisfying {\bf C}.

To see that each $\Gamma^{\widehat{i}}$ has rank at least $2$, note first that $\Gamma_<$ is dense in $\intl$. If $\Gamma^{\widehat{i}}$ was only rank $0$ or $1$, then it would be a discrete subgroup of $\intl$. The sum of the remaining $\Gamma^{\widehat{j}}$, for $j \neq i$, all lie inside the hyperplane $V_i$ which is complementary to the line $L(i) \supset \Gamma^{\widehat{i}}$. Hence, the sum group $S = \Gamma^{\widehat{1}} + \cdots + \Gamma^{\widehat{n}}$ is contained in a discrete union of parallel codimension $1$ hyperplanes of $\intl$, so is not dense. Since $S$ is finite index in $\Gamma$, we have that $\Gamma \leqslant (1/N) \cdot S$ for some $N \in \N$. Since $(1/N) \cdot S$ is still not dense in $\intl$, this contradicts density of $\Gamma$.
\end{proof}

\section{Decomposable cut and project schemes}\label{sec: decompositions}

One difficulty in treating the very general class of cut and project schemes considered here is that some schemes may have internal systems which `split' into summands. In this case, the different components can behave like independent schemes with varying dimensions, meaning that the relevant (and optimal) Diophantine exponents can vary between factors:

\begin{example}
Consider two codimension $1$, canonical cut and project schemes which are {\bf LR}, with total dimensions $k_1 < k_2$. This means that each window is an interval, with length equal to the length of a projected lattice point, and the projection of the lattice points to the internal space are each Diophantine, meaning that non-trivial lattice points of norm at most $r$ in the total space project to points which are at least distance $cr^{-(k_i-1)}$ from the origin in the internal space. By taking the direct sums of the physical, total and internal spaces, products of the lattices and of the windows, we obtain a codimension $2$ cut and project scheme, which can be shown to be {\bf LR} by applying our main theorem; indeed, the resulting patterns are products of two {\bf LR} patterns. However, the projected lattice $\Gamma_<$ (of rank $k_1 + k_2$) is a sum $\Gamma_< = T+S$, where $\rk(S) = k_1$, $\rk(T) = k_2$ and $S$, $T$ are contained in complementary $1$-dimensional subspaces. Since $k_1 < k_2$, the acceptance domains will be typically rectangular and wide in the first coordinate, but relatively short in the second coordinate. Whilst the projected lattice as a whole does not satisfy a natural Diophantine condition, it splits into two factors which are each Diophantine.
\end{example}

To address the irregularities of examples such as the above, we develop a theory of decompositions of polytopal cut and project schemes. We start by defining what it means to decompose the window; we turn to the problem of compatibly splitting the lattice in Section \ref{sec: splitting lattice}.

\subsection{Decomposing windows}

Our window will be \emph{decomposable} if it can be expressed as a Minkowski sum. The following more abstract definition, however, will be easier to work with in what is to follow. It will be seen to be equivalent to a geometric product decomposition of the window in Theorem \ref{thm: sum decomposition}. It is not hard to see that the following definition is equivalent to the definition of decomposability of Construction 3.5 in \cite{FHK02} when applied to the normals of supporting hyperplanes of a window:

\begin{definition} \label{def: decomp}
We call the window $W$ {\bf decomposable} if there exists a {\bf decomposition}, a partition $\sH_0 = A \cup B$, with $A$, $B \neq \emptyset$, for which $\intl = X_A + X_B$, where $X_A = \bigcap_{V \in A} V$ and $X_B = \bigcap_{V \in B} V$. Call $W$ {\bf indecomposable} if it does not permit a decomposition.
\end{definition}

Since $W$ is a polytope (it is compact) it has vertices, which implies that $\bigcap_{V \in \sH_0} V = \{0\}$. So if $W$ is decomposable we have that $\intl = X_A + X_B$ gives a direct sum decomposition of $\intl$ into complementary subspaces $X_A$ and $X_B$, with each $V \in \sH_0$ containing either $X_A$ or $X_B$. One way to visualise the above definition is as follows: we can split $\intl$ non-trivially as a sum $X_A + X_B$, so that every supporting subspace contains either $X_A$ or $X_B$.

\begin{example}
In codimension $n=1$ there is only one possible supporting subspace for $W$, the trivial subspace $\{0\}$. Hence $\# \sH_0 = 1$, so the scheme is automatically indecomposable.
\end{example}

\begin{example} \label{ex: decomp codim 2}
For codimension $n=2$, we have that $\# \sH_0 \geq 2$ (there are at least $2$ non-parallel supporting edges). If $\sH_0 = \{V_1,V_2\}$, then $A=\{V_1\}$, $B=\{V_2\}$ gives a decomposition of $W$. Note that in this case, $W$ is a parallelogram, which is a product of intervals. If $\# \sH_0 > 2$, then $W$ is indecomposable. Indeed, take any decomposition $\sH_0 = A \cup B$ and assume (without loss of generality) that $\# A \geq 2$. Two distinct $1$-dimensional subspaces intersect to the origin, so that $\dim(X_A) = 0$. Since $\dim(X_B) \leq 1$ (being contained in any $V \in B$), we cannot have that $\intl = X_A + X_B$.
\end{example}

\begin{example}
In the codimension $n=3$ case, for any valid decomposition $\sH_0 = A \cup B$, we have that $\dim(X_A) = 2$ and $\dim(X_B) = 1$ (or vice versa). Indeed, as noted, $X_A + X_B$ gives a direct sum decomposition of $\intl$, and $\dim(X_A) \leq 2$ (since $X_A \leqslant V$ for any $V \in A$); similarly, $\dim(X_B) \leq 2$. Since (without loss of generality) $\dim(X_A) = 2$, we must have that $A = \{V\}$ for a single $V \in \sH_0$ (since if it contained any other supporting hyperplane, the intersection would have strictly smaller dimension). Each element of $B$ contains the line $X_B$. It follows that $W$ has a `top' and `bottom' face, parallel with $V$, and that all other faces are rotates of each other around the line $X_B$, intersecting in edges parallel to $X_B$. So $W$ is a product of a $2$-dimensional polytope and an interval. If $W$ can not be expressed as such a product, then it is indecomposable.
\end{example}

We now give an equivalent definition of indecomposability by constructing a finite, undirected graph using the structure of flags on $\sH_0$:

\begin{definition} \label{def: flag graph}
Call $f \subset \sH_0$ a {\bf pre-flag} if there exists some $V \in \sH_0$ so that $f \cup \{V\}$ is a flag. We define the undirected graph $G(W)$ by setting $\sH_0$ as the nodes, and connect $V_1$ to $V_2$ with an edge if there is a pre-flag $f$ so that $f \cup \{V_i\}$ is a flag for both $i=1$ and $i=2$, in which case we write $V_1 \sim V_2$.
\end{definition}

The above is distinct but related to the graphs defined in Construction 3.5 in \cite{FHK02}, which always have only $n$ vertices. The graph $G(W)$ has the advantage of not depending on a choice of basis (or inner product on the internal space). It also allows for a more elementary proof of Theorem \ref{thm: connected => equal ranks} later.

\begin{theorem} \label{thm: disconnected decomposition graph}
A partition $\sH_0 = A \cup B$ provides a decomposition of $W$ if and only if $A$ and $B$ are disconnected in $G(W)$.
\end{theorem}

\begin{proof}
Suppose that $\sH_0 = A \cup B$ is a decomposition for $W$. We claim that for any flag $f = \{V_1,V_2,\ldots,V_n\}$, the set $f_A \subseteq f$ of hyperplanes in $A$ has $\dim(V_B)$ elements, and analogously for $B$. Indeed, since $f$ is a flag the sequence $\dim(V_1)$, $\dim(V_1 \cap V_2)$, \ldots, $\dim(V_1 \cap V_2 \cap \cdots, \cap V_n)$ reduces by $1$ at each step. Arrange $f$ so that the elements of $f_A$ appear first in the list. Then we see that $\dim(X_A) \leq n - \# f_A$, since $X_A$ is contained in the intersection of elements of $f_A$, and analogously $\dim(X_B) \leq n - \# f_B$. Hence $\# f_A = n - \# f_B \geq \dim(X_B)$, and similarly $\# f_B \geq \dim(X_A)$. Neither of these inequalities can be strict, since otherwise $n = \# f_A + \# f_B > \dim(V_B) + \dim(X_A) = n$, a contradiction, so the claim follows.

It is now clear that $A$ and $B$ are disconnected in $G(W)$. Indeed, suppose that $V_a \in A$, $V_b \in B$ and $V_a \sim V_b$, so that there is a pre-flag $f$ which is completed to a flag by adding either $V_a$ or $V_b$. But adjoining $V_a$ gives a flag with one more element in $A$ than by adjoining $V_b$, contradicting that the numbers $\#f_A$ and $\#f_B$ only depend on $\dim(X_A)$ and $\dim(X_B)$.

For the converse direction, suppose that $A$ and $B$ are disconnected in $G(W)$. We need to show that $\intl = X_A + X_B$. Take any flag $f = \{V_{a_1}, V_{a_2}, \ldots, V_{a_s}, V_{b_1}, \ldots, V_{b_t}\}$, where each $V_{a_i} \in A$ and $V_{b_i} \in B$. Let $X'_A = \bigcap V_{a_i}$ and $X'_B = \bigcap V_{b_i}$ (that is, we only intersect the elements of the flag here, rather than all elements of $A$ for $X_A$ and all elements of $B$ for $X_B$). Clearly $X_A' \supseteq X_A$, $X_B' \supseteq X_B$ and $\intl = X_A' + X_B'$ (since each intersect trivially, and their dimensions sum to $n$, as $f$ is a flag). We claim that $X'_A \subseteq V_a$ for all $V_a \in A$, and similarly for $B$, from which the result follows since this implies the reverse inequalities $X_A \subseteq X_A'$ and $X_B \subseteq X_B'$.

To see this, let $V_a \in A$. If $V_a \in f$ then clearly $X'_A \subseteq V_a$. So suppose not, and consider the set $f \cup \{V_a\}$, where we order the elements so that the elements of $A$ appear first, then $V_a$, then the elements of $B$:
\[
f' = \{V_{a_1},V_{a_2},\ldots,V_{a_s},V_a,V_{b_1},V_{b_2},\ldots,V_{b_t}\}.
\]
Since each element of the above is a codimension 1 subspace, taking intersections
\[
V_{a_1}, \ V_{a_1} \cap V_{a_2}, \ V_{a_1} \cap V_{a_2} \cap V_{a_3}, \ \ldots
\]
from left to right drops the dimension by either $0$ or $1$ at each step. If $X'_A \nsubseteq V_a$ then the dimension drops by $1$ for each new intersection with an element of $A$, including $V_a$. Since there are $n+1$ elements, the dimension must not drop when intersecting with some $V_b \in B \cap f$, that is, some $V_b$ contains the intersection of the elements before it in $f'$, so omitting it still results in a set with trivial intersection, which is thus a flag. So the pre-flag $f' \setminus \{V_a,V_b\}$ can be completed to a flag by adjoining either $V_a$ or $V_b$, that is, $V_a \sim V_b$, contradicting that $A$ and $B$ are disconnected in $G(W)$. Hence, $X'_A \subseteq V_a$ for all $V_a \in A$ and analogously for $B$, as required.
\end{proof}

The advantage of the definition of a decomposition in terms of the graph of flags is it allows a simpler proof that minimal complexity schemes with indecomposable windows have hyperplanes of constant stabiliser rank:

\begin{theorem} \label{thm: connected => equal ranks}
Suppose that $\mathcal{S}$ satisfies {\bf C} and that $V_1$ and $V_2$ are connected by an edge in $G(W)$. Then $\rk(V_1) = \rk(V_2)$.
\end{theorem}

\begin{proof}
By Theorem \ref{thm: generalised complexity} and Corollary \ref{cor: minimal complexity}, for any flag $f \subseteq \sH_0$ we have
\[
\sum_{V \in f} (k-\rk(H)) = d+n = k,
\]
since each $\beta_H = n-1$ by hyperplane spanning, which follows from {\bf C} by Theorem \ref{thm: C => hyperplane spanning}. It follows that $\sum \rk(H) = k(n-1)$ for each flag.

Suppose then that $V_1 \sim V_2$ in $G(W)$. Take a pre-flag $f$ which is completed to a flag by adding either $V_1$ or $V_2$. Then $\rk(V_1) = k(n-1) - \left(\sum_{V \in f} \rk(V)\right) = \rk(V_2)$, as required.
\end{proof}

\begin{corollary} \label{cor: equal ranks}
Suppose that $\mathcal{S}$ satisfies {\bf C}. If $\rk(H_1) \neq \rk(H_2)$ for some $H_1$, $H_2 \in \sH$, then $W$ is decomposable.
\end{corollary}

\begin{proof}
By the previous theorem, $H_1$ and $H_2$ belong to different path-components in $G(W)$, so $G(W)$ is disconnected. Hence $W$ is decomposable by Theorem \ref{thm: disconnected decomposition graph}.
\end{proof}

The above result allows us to decompose a minimal complexity cut and project scheme whenever it has hyperplanes with stabilisers of different ranks. This can be repeated until an indecomposable setup is reached. To make this concrete, we now give an alternative, more direct description of a decomposition into more than two pieces. It will turn out that this is equivalent to iteratively breaking summands into two, as in the above definitions, and will eventually describe the window as a product of lower dimensional ones.

\begin{definition} \label{def: multi decomp}
A {\bf decomposition} of $W$ is a set $\{X_1,X_2,\ldots,X_m\}$ of $m \geq 2$ subspaces of $\intl$ for which
\begin{enumerate}
	\item each $X_i$ is non-trivial (i.e., not equal to $\{0\}$ or $\intl$);
	\item the $X_i$ are complementary and give a direct sum decomposition of $\intl$;
	\item for each $V \in \sH_0$ we have that $V \supseteq X_i$ for all but one $X_i$.
\end{enumerate}
\end{definition}

We give an alternative version of the above definition, which for $m=2$ clearly extends our original Definition \ref{def: decomp}.

\begin{definition} \label{def: copartition decomp}
For subsets $A_i \subseteq \sH_0$, call $\{A_i\}_{i=1}^m$ a {\bf copartition} of $\sH_0$ if $\{\sH_0 \setminus A_i\}_{i=1}^m$ is a partition of $\sH_0$. Call a copartition {\bf non-trivial} if each $A_i \neq \emptyset$, $\sH_0$. Then we define a {\bf decomposition} of the window $W$ to be a non-trivial copartition $\{A_i\}_{i=1}^m$ of $\sH_0$ such that $m \geq 2$ and $\intl = X_1 + \cdots + X_m$, where
\begin{equation} \label{eq: copartition decomp}
X_i \coloneqq \bigcap_{V \in A_i} V.
\end{equation}
\end{definition}

Of course, $\{A_i\}_{i=1}^m$ is a co-partition precisely when each $V \in \sH_0$ appears in all but one $A_i$. The above definition agrees with Definition \ref{def: decomp} in the case that $m=2$. We prove below that both versions, Definition \ref{def: multi decomp} and \ref{def: copartition decomp}, are equivalent:

\begin{proposition} \label{prop: decomps equiv}
Suppose that $\{X_i\}_{i=1}^m$ is a decomposition of $W$, in the sense of Definition \ref{def: multi decomp}. Define $A_i$ to be the set of $V \in \sH_0$ for which $V \supseteq X_i$. Then $\{A_i\}_{i=1}^m$ is a decomposition of $W$, in the sense of Definition \ref{def: copartition decomp}.

Conversely, suppose that $\{A_i\}_{i=1}^m$ is a decomposition, in the sense of Definition \ref{def: copartition decomp}. Then the subspaces $\{X_i\}_{i=1}^m$, as defined in Equation \ref{eq: copartition decomp}, are a decomposition of $W$ in the sense of Definition \ref{def: multi decomp}.
\end{proposition}

\begin{proof}
Suppose that $\{X_i\}_{i=1}^m$ is a decomposition and let $A_i$ be as defined in the statement of the proposition. Clearly $\{A_i\}_{i=1}^m$ is a copartition, since by Definition \ref{def: multi decomp} (3) each $V \in \sH_0$ belongs to all but precisely one $X_i$. No $A_i = \emptyset$ or $\sH_0$ or else $X_i = \intl$ or $\{0\}$, respectively, violating (1). Moreover, the original subspaces $X_i$ are at least contained in those of the intersections in Equation \ref{eq: copartition decomp}, so these intersections sum to $\intl$ by (2), as required.

For the converse direction, suppose that $\{A_i\}_{i=1}^m$ gives a decomposition in the sense of Definition \ref{def: copartition decomp}. We must show that properties (1--3) of Definition \ref{def: multi decomp} hold for the $X_i$ given by Equation \ref{eq: copartition decomp}.

For condition (1), clearly each $X_i \neq \intl$, since $X_i \subseteq V$ for any $V \in A_i$ (and each $A_i \neq \emptyset$ by assumption). To see that each $X_i \neq \{0\}$, firstly note that each $X_i \cap X_j = \{0\}$ for distinct $i$, $j$. Indeed, $A_i \cup A_j = \sH_0$, since any given $V \in \sH_0$ belongs to at least one of $A_i$ or $A_j$. Hence, \[
X_i \cap X_j = \bigcap_{V \in A_i \cup A_j} V = \bigcap_{V \in \sH_0} V = \{0\},
\]
so the $X_i$ intersect trivially. By assumption, $X_i + \hat{X}_i = \intl$, where $\widehat{X}_i$ is the direct sum of all $X_j$ with $j \neq i$. Then $X_i \neq \{0\}$ follows from $\widehat{X}_i \neq \intl$, which holds since each $X_j \subseteq V$, for $j \neq i$, where $V \in \sH_0$ is any element not contained in $A_i$ (and thus is contained in every other $A_j$); such an element exists since by assumption each $A_i \neq \sH_0$.

The above argument shows that the $X_i$ intersect trivially, and by assumption their sum spans $\intl$, so (2) is satisfied. Finally, by Equation \ref{eq: copartition decomp}, $V \supseteq X_i$ if $V \in A_i$ (since $X_i$ is the intersection of such $V$), and $V \in A_i$ for all but one $i$, since $\{A_i\}_{i=1}^m$ is a copartition. For the unique value of $i$ with $V \notin X_i$, we cannot have $V \supseteq X_i$, since otherwise the sum of the $X_i$ would be contained in $V$, contradicting the sum spanning $\intl$, so (3) holds.
\end{proof}

By the above, we may use either of the two definitions of a decomposition (in terms of the $X_i$ or $A_i$) interchangeably. The following corollary of the above will be useful when analysing how the lattice decomposes with respect to a decomposition, and also in showing that any two decompositions have a common refinement:

\begin{corollary} \label{cor: decomps as intersections}
If $\{X_i\}_{i=1}^m$ is a decomposition of $W$, then each $X_i$ is equal to the intersection of hyperplanes of $\sH_0$ containing it. \end{corollary}

\begin{proof}
By the above proposition, the subspaces $X_i'$ given by intersecting the $V \in \sH_0$ containing some $X_i$ give a decomposition in the sense of Definition \ref{def: multi decomp}. Clearly each $X_i \subseteq X_i'$, but none of these can be a strict inclusion or else the $X_i'$ would not be complementary, violating Definition \ref{def: multi decomp} (2).
\end{proof}

\begin{proposition} \label{prop: decomp reduce}
Suppose that $\{X_i\}_{i=1}^m$ is a decomposition, with $m \geq 3$. Then replacing the terms $X_i$, $X_j$ ($i \neq j$) with $X_i+X_j$ gives another decomposition of $m-1$ elements. In terms of Definition \ref{def: copartition decomp}, for the copartition $\{A_i\}_{i=1}^m$ this corresponds to replacing $A_i$ and $A_j$ with $A_i \cap A_j$.
\end{proposition}

\begin{proof}
Conditions (1--2) of Definition \ref{def: multi decomp} clearly still hold after replacing $X_i$ and $X_j$ with their sum. For each $V \in \sH_0$, if $V \supseteq X_i$, $X_j$, then $V \supseteq X_i + X_j$ (since $V$ is a vector subspace). Hence, if $X_\ell$ with $\ell \neq i$, $j$ is the unique subspace which is not contained in $V$, this property still holds in the new decomposition. Otherwise, if $\ell = i$, then $X_i + X_j \nsubseteq V$, or else $X_i \subseteq V - X_j \subseteq V-V = V$, a contradiction. Similarly, if $\ell = j$ then $X_i + X_j \nsubseteq V$, so (3) is satisfied and we have a decomposition.

In terms of coparitions, we first define the copartition $\{A_i\}_{i=1}^m$ as in Proposition \ref{prop: decomps equiv}. Then for $\ell \neq i$, $j$, $A_i$ is the set of $V \in \sH_0$ for which $V \supseteq X_i$. On the other hand, $A_i \cap A_j$ is precisely the subset of hyperplanes for which $V \supseteq X_i + X_j$. Indeed, if $V \in A_i \cap A_j$, then by definition $V \supseteq X_i$ and $V \supseteq X_j$, so $V \supseteq X_i + X_j$, since $V$ is a subspace. On the other hand, if $V \notin A_i$ or $V \notin A_j$, then $V \nsupseteq X_i + X_j$, as above. So replacing $A_i$ and $A_j$ with $A_i \cap A_j$ corresponds to replacing $X_i$ and $X_j$ with $X_i + X_j$, according to Proposition \ref{prop: decomps equiv}.
\end{proof}

Recall that for $m = 2$, Definition \ref{def: copartition decomp} of a decomposition in terms of copartitions coincides with Definition \ref{def: decomp}. The above shows that we may define a window to be {\bf decomposable} if it has a decomposition, using any of the three previous definitions.

The following result generalises an argument from the proof of Theorem \ref{thm: disconnected decomposition graph}. The first part of it will be used later in showing that decompositions have common refinements:

\begin{proposition} \label{prop: decomp flag sizes}
Let $\{X_i\}_{i=1}^m$ be a decomposition of $W$ and $f$ be a flag in $\sH_0$. Define $X_i^f$ to be the intersection of hyperplanes $V \in f$ which contain $X_i$. Then $X_i^f = X_i$ for each $i$. Moreover, the number of elements in any flag $f$ which contain $X_i$ is given by $n - \dim(X_i)$. That is, $\#(f \cap A_i) = n - \dim(X_i)$ for each $i$, where $\{A_i\}_{i=1}^m$ is the equivalent copartition description of the decomposition coming from Proposition \ref{prop: decomps equiv}.
\end{proposition}

\begin{proof}
We claim that $X_i^f \cap X_j^f = \{0\}$ for distinct $i$ and $j$. Indeed, given any $V \in f$ we have that $V \supseteq X_i$ or $V \supseteq X_j$, by (3) of Definition \ref{def: multi decomp}. Hence $V \supseteq X_i^f \cap X_j^f$. Since this holds for all elements of $f$, and the elements of a flag intersect to $\{0\}$, we must have that $X_i^f \cap X_j^f = \{0\}$.

Clearly $X_i^f \supseteq X_i$ for each $i$. Since the $X_i^f$ intersect trivially, and the sum of $X_i$ span $\intl$, for dimensional reasons we must have that $X_i^f = X_i$ for each $i$.

By definition,
\[
X_i^f = \bigcap_{V \in f \cap A_i} V,
\]
where $A_i$ is defined as the set of $V \in \sH_0$ containing $X_i$, as in Proposition \ref{prop: decomps equiv}. Since $f$ is a flag, the dimension of the above intersection is given by $\dim(X_i^f) = n - \#(f \cap A_i)$. Since $X_i = X_i^f$, the result follows.
\end{proof}

We now turn to the problem of relating this abstract and combinatorial definition of a decomposition to a more geometric one: expressing the window as a Minkowski sum. We first introduce some important subspaces associated to a decomposition which define the factors of the window.

\begin{notation} \label{not: decomposition}
Given a decomposition $\{X_1,\ldots,X_m\}$ of $W$ and $i \in \{1,\ldots,m\}$:
\begin{enumerate}
	\item $\widehat{X}_i$ is defined as the sum of all of the $X_j$ for $j \neq i$;
	\item for $H \in \sH$ we let $X_H$ be the unique $X_i$ for which $V(H)$ does not contain $X_i$, and write $\widehat{X}_H$ for $\widehat{X}_i$
	\item $H_X \coloneqq H \cap X_H$;
	\item $H^+_X \coloneqq H^+ \cap X_H$;
	\item $V_X = V \cap X_H$ for $V = V(H)$.
\end{enumerate}
We note that $H_X$, $H_X^+$ and $V_X$ are an affine hyperplane, half space and codimension $1$ subspace in $X_H$, respectively (since $X_H \nsubseteq V(H)$), which aids in remembering this notation (in line with how we use $H$, $H^+$ and $V$). Terms $X$ and $\widehat{X}$ (with subscripts) are complementary subspaces.
\end{notation}

\begin{lemma} \label{lem: decomposed hyperplanes}
If $\{X_1,\ldots,X_m\}$ is a decomposition of $W$, then for $H \in \sH$ and $V = V(H)$ we have:
\begin{enumerate}
	\item $H = H_X + \widehat{X}_H$;
	\item $H^+ = H_X^+ + \widehat{X}_H$;
	\item $V = V_X + \widehat{X}_H$.
\end{enumerate}
\end{lemma}

\begin{proof}
We start with the third item. Since $V$ is codimension $1$ we have that $V_X$ has dimension $\dim(X_H) - 1$, indeed, $X_H$ is not contained in $V$ by the definition of $X_H$. So $V_X + \widehat{X}_H$ has the correct dimension $n-1$. Clearly if $v \in \widehat{X}_H$ then $v \in V$, and if $v \in V_X$ then $v \in V$, so it follows that $V = V_X + \widehat{X}_V$.

Let $v \in \intl$ be such that $H+v = V$. Since $X_H + \widehat{X}_H$ is a direct sum decomposition of $\intl$, we may write $v = v_H + \widehat{v}_H$, where $v_H \in X_H$ and $\widehat{v}_H = \widehat{X}_H$. Since $\widehat{X}_H \subseteq V$ we may assume that $v = v_H$ and $\widehat{v}_H = 0$, since $H = V + v_H + \widehat{v}_H = (V + \widehat{v}_H) + v_H = V + v_H$. Hence $H = V + v = (V_X + v) + \widehat{X}_H$, so we need to check that $V_X + v = H_X$. Indeed, $V_X + v = (V \cap X_H) + v = (V+v) \cap (X_H + v) = (V+v) \cap X_H = H \cap X_H = H_X$, as required. Item 2 is proved analogously.
\end{proof}

Given a decomposition $\{X_i\}_{i=1}^m$ of $W$ and $i \in \{1,\ldots,m\}$, let $\sH_i$ denote the set of $H_X \subseteq X_i$, where $X_H = X_i$ (that is, $X_i \nsubseteq V(H)$), and $\sH_i^+$ the corresponding half spaces $H_X^+$. We define the polytope $W_i \subseteq X_i$ as:
\[
W_i = \bigcap_{H_X^+ \in \sH_i} H_X^+.
\]

\begin{theorem} \label{thm: sum decomposition}
Let $\{X_i\}_{i=1}^m$ be a set of subspaces giving a direct sum decomposition of $\intl$. Then $\{X_i\}_{i=1}^m$ is a decomposition of $W$ if and only if there are polytopes $W_i \subseteq X_i$ for which $W = W_1 + \cdots + W_m$.
\end{theorem}

\begin{proof}
Suppose that $\{X_i\}_{i=1}^m$ is a decomposition of $W$. By Lemma \ref{lem: decomposed hyperplanes}, we have that
\[
W = \bigcap_{i=1}^m \bigcap_{H^+ \in \sH_i^+} (H^+_X + \widehat{X}_H) = \bigcap_{i=1}^m \left(\left(\bigcap_{H^+ \in \sH_i^+} H^+_X \right) + \widehat{X}_i\right) = \bigcap_{i=1}^m (W_i + \widehat{X}_i).
\]
The second equality above follows from observing that the $\widehat{X}_H$ coordinates of the $i$th term of the left-hand intersection must be taken equal, since each $H_X^+$ is contained in $X_H$, which is complementary to $\widehat{X}_H$.

Given $w = w_1 + w_2 + \cdots + w_m \in W_1 + W_2 + \cdots W_m$, clearly $w$ is in right-hand intersection above, since $w = w_i + (w - w_i)$, and $w - w_i = w_1 + \cdots + w_{i-1} + w_{i+1} + \cdots + w_m \in \widehat{X}_i$ for each $i$. Conversely, given an element $w$ of the above intersection, for each $i$ we may write $w = w_i + \widehat{w}_i$ for $w_i \in W_i$ and $\widehat{w}_i \in \widehat{X}_i$. Then $w = w_1 + w_2 + \cdots w_m$, since this has the same $X_i$ components and the $X_i$ give a direct sum decomposition. So $W = W_1 + \cdots W_m$. Since $W$ is compact, the polyhedra $W_i$ must be compact and hence, as intersections of half spaces, are polytopes.

So now suppose that $W = W_1 + W_2 + \cdots + W_m$ for polytopes $W_i \subseteq X_i$. Let $\sH^+_i$ be the (irredundant) set of half-spaces in $X_i$ with intersection $W_i$. We claim that $\sH^+$ is the set of half-spaces $H^+ + \widehat{X}_i$ over all $H^+ \in \sH^+_i$ and $i \in \{1,\ldots,m\}$. Indeed, as above, $W$ is the intersection of these half-spaces. The set $\sH^+$ is irredundant since omitting any one element, say $H_i + \widehat{X}_i$, leads to an intersection equal to $W_1 + \cdots + W_{i-1} + W_i' + W_{i+1} + \cdots W_m$, where $W_i'$ is strictly larger than $W_i$ by the fact that $\sH^+_i$ is irredundant. Finally, we have that $V(H+\hat{X}_i) = (V(H) + \widehat{X}_i) \supseteq \widehat{X}_i \supseteq X_j$ for $H \in \sH_i$ and $j \neq i$, so each hyperplane of $\sH_0$ contains all but one of the $X_i$, as required.
\end{proof}

The result below shows that we may always decompose uniquely into indecomposables:

\begin{proposition} \label{prop: common refinement}
Any two decompositions $X = \{X_i\}_{i=1}^q$ and $Y = \{Y_i\}_{i=1}^r$ of $W$ have a common refinement, that is, a decomposition $Z = \{Z_i\}_{i=1}^s$ such that each $Z_i$ is contained in an element of $X$ and an element of $Y$.
\end{proposition}

\begin{proof}
Define the decomposition $Z$ to be
\[
\{X_i \cap Y_j \mid X_i \in X, Y_j \in Y\},
\]
omitting elements with trivial intersection, so Definition \ref{def: multi decomp} (1) is satisfied by construction. Clearly $Z$ refines both $X$ and $Y$, and distinct elements intersect trivially (since distinct elements of both $X$ and $Y$ intersect trivially). Any $V \in \sH_0$ wholly contains every $X_i \cap Y_j$ except for $X_H \cap Y_H$. It must be that $V \nsupseteq X_H \cap V_H$, and hence that (3) holds, once we have shown (2) (that the sum of $Z_i \in Z$ span $\intl$), since otherwise each element of $Z$ is contained in $V$, contradicting that they span $\intl$.

To prove (2), take any flag $f \subseteq \sH_0$ and consider the $1$-dimensional subspace $L(V)$ given by intersecting all elements of $f$ except for $V \in f$. Just as in the proof of Lemma \ref{lem: lattice spanned by lines}, these are $1$-dimensional subspaces which give a direct sum decomposition of $\intl$, so it will suffice to show that some element of $X$ and some element of $Y$ contains $L(V)$.

Recall the $X_i^f$ from Proposition \ref{prop: decomp flag sizes}. Given $V \in f$, take the unique $i$ with $V \nsupseteq X_i$. By definition, $X_i^f$ is an intersection of elements of $f$ not including $V$, so $L(V)$, being the intersection of all elements of $f$ except $V$, must be a subset of $X_i^f$. Then $L(V) \subseteq X_i$ since $X_i^f = X_i$ by Proposition \ref{prop: decomp flag sizes}. By the same argument, there is some $Y_j \in Y$ containing $L(V)$, so $L(V) \subseteq X_i \cap Y_j$, as required.
\end{proof}

\begin{corollary}
Every window $W$ has a unique and minimal decomposition $\{X_i\}_{i=1}^m$ so that $W = W_1 + \cdots + W_m$, where each $W_i \subseteq X_i$ is indecomposable.
\end{corollary}

In summary, there always exists a decomposition of a window into a sum of indecomposables. Furthermore, when assuming minimal complexity {\bf C}, the supporting hyperplanes in each summand have equal ranks by Corollary \ref{cor: equal ranks}.

\subsection{Decomposing schemes: Splitting both $\Gamma$ and $\intl$} \label{sec: splitting lattice}
The acceptance domains (or cut regions) are defined not just by the window but also by how it is aligned relative to the projected lattice. We may not effectively apply the above techniques without a corresponding decomposition of $\Gamma$. Fortunately, for schemes satisfying {\bf C}, the geometric decomposition of the window gives an associated decomposition of the lattice, up to finite index:

\begin{definition}
Let $\{X_i\}_{i=1}^m$ be a decomposition of $W$, with associated copartition decomposition $\{A_i\}_{i=1}^m$. We define $\Gamma_i = \Gamma^{A_i}$, the subgroup which stabilises the hyperplanes of $A_i$. Equivalently, $\Gamma_i$ is determined by $(\Gamma_i)_< = \Gamma_< \cap X_i$.
\end{definition}

\begin{remark}
Note that $(\Gamma^{A_i})_< = \Gamma_< \cap X_i$ by Lemma \ref{lem: higher stabilisers} and the definition of $X_i$ in terms of $A_i$ by Definition \ref{def: copartition decomp}, which agrees with the original decomposition $\{X_i\}_{i=1}^m$ by Proposition \ref{prop: decomps equiv}. This uniquely defines $\Gamma_i \leqslant \Gamma$ since $\pi_<$ is injective.
\end{remark}

\begin{theorem}\label{thm: finite index splitting}
Suppose that $\mathcal{S}$ satisfies {\bf C}. Then $\Gamma_1 + \cdots + \Gamma_m$ is finite index in $\Gamma$. 
\end{theorem}

\begin{proof}
We shall prove the result by induction on the number of elements of the decomposition. Recall from Proposition \ref{prop: decomp reduce} that for a decomposition $\{A_i\}_{i=1}^m$, $m \geq 3$, the decomposition given by replacing the two elements $A_i$ and $A_j$ ($i \neq j$) with $A_i \cap A_j$ is also a decomposition. So, for $\ell = 2$, \ldots, $m$, we have decompositions of $\ell$ elements given by
\[
D_\ell \coloneqq \{A_1, A_2, \ldots, A_{\ell-1}, B_\ell\} \ , \text{ where } B_\ell \coloneqq \bigcap_{j=\ell}^m A_j.
\]
Note that $D_m = \{A_1,\ldots,A_m\}$ is our target decomposition. Although not strictly a decomposition, we can also define the `one element decomposition' $D_1 = \{\emptyset\}$, for which the result holds trivially, as $\Gamma^\emptyset = \Gamma$. So suppose that the result holds for the decompositions $D_1$, $D_2$, \ldots, $D_\ell$. We shall show that it holds for $D_{\ell+1}$.

We have that $D_\ell$ is obtained from $D_{\ell+1}$ by replacing $A_\ell$ and $B_{\ell+1}$ by their intersection $B \coloneqq B_\ell$. Note that $A_\ell \cup B_{\ell+1} = \sH_0$ (since every element of $\sH_0$ is in at least one of these subsets, by Definition \ref{def: copartition decomp}) so we may write
\[
A_\ell = B \cup I \ , \ B_{\ell+1} = B \cup J,
\]
for the last two elements of $D_{\ell+1}$, where $I$ and $J \subseteq \sH_0$ are chosen so that $\sH_0 = B \cup I \cup J$ is a partition. Now, by our induction assumption,
\[
\Gamma_1 + \Gamma_2 + \cdots + \Gamma_{\ell-1} + \Gamma^B \leqslant \Gamma
\]
is a finite index inclusion. So to obtain the corresponding result for $D_{\ell+1}$, we need that
\[
\Gamma^{B \cup I} + \Gamma^{B \cup J} \leqslant \Gamma^B
\]
is a finite index inclusion.

Choose any flag $f$ and let $f_B \coloneqq f \cap B$, $f_I \coloneqq f \cap I$ and $f_J \coloneqq f \cap J$. Let $\alpha = \# f_I$ and $\beta = \# f_J$. By {\bf C}, Theorem \ref{thm: generalised complexity} and Corollary \ref{cor: minimal complexity},
\begin{equation}\label{eq: rks 1}
\sum_{V \in f} \rk(V) = \left(\sum_{V \in f_B} \rk(V)\right) + \left(\sum_{V \in f_I} \rk(V)\right) + \left(\sum_{V \in f_J} \rk(V)\right) = (n-1) \cdot k,
\end{equation}
since $\mathcal{S}$ is hyperplane spanning by Theorem \ref{thm: C => hyperplane spanning}. By Lemma \ref{lem: lattice ranks}, for any $S \subseteq \sH_0$ and $V \in S$,
\begin{equation} \label{eq: remove singles}
\rk(S) \geq \rk({S - \{V\}}) + \rk(V) - \rk(\emptyset) = \rk(S - \{V\}) + \rk(V) - k.
\end{equation}
Recall from Proposition \ref{prop: decomp flag sizes} that for any flag $f$, we have that $X_i = X_i^f$ for a decomposition $\{X_i\}_{i=1}^m$. That is, one may restrict intersections to any given flag in determining the subspace intersections $X_i$; moreover, the stabiliser subgroups $\Gamma^{A_i}$ are the intersections with these $X_i$, once projected to $\intl$. In particular, $\rk(B \cup I) = \rk(B \cup f_I)$ (once the elements of the flag $f_I$ are stabilised, so too are all elements of $I$). So by inductively applying Equation \ref{eq: remove singles} to successively remove elements of $f_I$ from $B \cup f_I$, we see that
\[
\rk(B \cup I) = \rk(B \cup f_I) \geq \rk(B) + \left(\sum_{V \in f_I} \rk(V)\right) - \alpha \cdot k,
\]
and similarly for $B \cup J$; summing these together, we obtain:
\begin{equation} \label{eq: rks 2}
\rk(B \cup I) + \rk(B \cup J) \geq 2 \cdot \rk(B) + \left(\sum_{V \in f_I \cup f_J} \rk(V) \right) - (\alpha+\beta) \cdot k.
\end{equation}
Similarly for $B = B_\ell$,
\begin{equation} \label{eq: rks 3}
\rk(B) = \rk(f_B) \geq \left(\sum_{V \in f_B} \rk(V)\right) - ((n-(\alpha+\beta)) - 1) \cdot k,
\end{equation}
where in this case $f_B$ has $n - (\alpha + \beta)$ elements and we need to apply \eqref{eq: remove singles} $n - (\alpha + \beta) - 1$ times to break it into singletons.

Combining \eqref{eq: rks 1}, \eqref{eq: rks 2} and \eqref{eq: rks 3},
\[
\rk(B \cup I) + \rk(B \cup J) \geq \rk(B) + \left(\sum_{V \in f} \rk(V)\right) - (n-1) \cdot k = \rk(B).
\]
On the other hand, $\Gamma^{B \cup I}$, $\Gamma^{B \cup J}$ have trivial intersection (since $\sH_0 = B \cup I \cup J$, and no non-trivial element of $\Gamma$ can stabilise all elements of $\sH_0$), so the sum of their ranks is the rank of their sum. Hence, having equal rank, $\Gamma^{B \cup I} + \Gamma^{B \cup J}$ is finite index in $\Gamma^B$, as required.
\end{proof}

Since $\Gamma_1 + \cdots + \Gamma_m$ is finite rank in $\Gamma$, each $\Gamma_i$ projects densely to $X_i$ (or else the sum does not project densely to $\intl$, contradicting that $\Gamma$ does). The following corollary summarises the main conclusions of this section and establishes some further notation. 

\begin{corollary}\label{cor: decompose}
If $\mathcal{S}$ is aperiodic and satisfies {\bf C} then it may be split into components, in the following sense:
\begin{enumerate}
\item $\intl = X_1 + \cdots + X_m$ with the subspaces $X_i$ complementary;
\item $W = W_1 + \cdots + W_m$ where each $W_i \subset X_i$ is a polytope;
\item $\Gamma_1 + \cdots + \Gamma_m$ is a finite index subgroup of $\Gamma$, and each $\Gamma_i$ projects densely into $X_i$. 
\end{enumerate}
For each $H \in \sH$ there is exactly one $X_H \in \{X_i\}_{i=1}^m$ such that $V(H)\not \supseteq X_H$. For each $i = 1,\dots,m$, let 
\[
\sH_i = \{H \cap X_i \mid H \in \sH, V(H) \nsupseteq X_i\} \ , \ \ \sH_i^+ = \{H^+ \cap X_i \mid H^+ \in \sH^+, V(H) \nsupseteq X_i\}.
\]
Then $H \in \sH_i$ are affine hyperplanes in $X_i$ and give the sides of the polytope
\[
W_i = \bigcap_{H^+ \in \sH_i^+} H^+.
\]
For $H \in \sH_i$, given by $\tilde {H} \cap X_i$ for some $\tilde{H} \in \sH$, we have the stabiliser of $H$ in $\Gamma_i$:
\[
\Gamma_i^{H} = \Gamma_i \cap \Gamma^{\tilde{H}} = \{\gamma \in \Gamma_i \mid \gamma_< \in V(H)\}.
\]
The rank $\rk(H) = \rk(\Gamma_i^H)$ will sometimes be denoted by $\rk_i(H)$ if we wish to emphasise these as coming from subgroups of $\Gamma_i$. We may always choose a sufficiently fine decomposition (for example, by passing to indecomposables) so that for a fixed $i = 1, \dots, m$, the ranks of $\Gamma_i^H$ for $H \in \sH_i$ are all equal. We denote this common rank by $r_i$. This rank is determined by $k_i$ and $n_i$ when property {\bf C} is assumed (see Proposition \ref{prop: rank formula}).
\end{corollary}

\begin{definition} \label{def: subsystems}
We call each $(\spintl, \Gamma_i, W_i)$ a {\bf subsystem}. We denote $k_i \coloneqq \rk \Gamma^i$, $n_i \coloneqq \dim X_i$ and $d_i \coloneqq k_i - n_i$. By Corollary \ref{cor: decompose} there is a number $N$ such that 
\[
\Gamma_1 + \cdots + \Gamma_m \leqslant \Gamma \leqslant \tfrac 1N \Gamma_1 + \cdots +\tfrac 1N\Gamma_m. 
\]
When context makes it clear, we shall sometimes also refer to $(X_i,\frac{1}{N}\Gamma_i,W_i)$ as a subsystem.
\end{definition}

\subsubsection{Complexity and repetitivity of subsystems} \label{sec: subsystems}
For $r > 0$, we let $\Gamma_i(r) \coloneqq B_r(0)\cap \Gamma_i$, where $B_r(0)$ is the closed ball of radius $r$ at the origin in the total space $\tot$. Most of the constructions and definitions we have seen need only the data of the subsystems. For example, the stabiliser subgroups (as discussed above), acceptance domains and cut regions may be given the same definitions, as well as the hyperplane spanning and (weakly) homogeneous conditions.

Even though the subsystems do not have canonically associated physical spaces (although one may somewhat arbitrarily assign a physical dimension and projection), the associated `lifted' cut and project sets still exist:
\[
\cps_i^\wedge \coloneqq \{\gamma \in \Gamma_i \mid \gamma_< \in W_i\}.
\]
We may define the complexity and repetitivity function in terms of the acceptance domains (see Notation \ref{not: subsystem acceptance domains}). So the complexity function $p(r)$ is given by the number of distinct $r$-acceptance domains. The repetitivity function $\rho(r)$ is the smallest value of $R$ so that for every $\gamma \in \cps_i^\wedge$, and every $A \in \sA_i(r)$, there is some $\gamma' \in B_R(\gamma) \cap \Gamma_i$ with $\gamma_<' \in A$. We may thus define a subsystem to be linearly repetitive if it has repetitivity $\rho(r) \leq Cr$ for some $C > 0$. Alternatively, by choosing a suitable projection of $\cps_i^\wedge$ onto a $d_i$-dimensional subspace, one obtains a $k_i$-to-$d_i$ cut and project scheme with the corresponding complexity and repetitivity function, up to linear constants.

\section{Acceptance domains and cut regions in the subsystems} \label{sec: accs and cuts in subsystems}

In this section we shall see how properties of the subsystems relate to those of the original cut and project scheme. We note that, in this section, it will not be necessary to assume that the subsystems have constant stabiliser rank (although we will still always assume that $\Gamma_1 + \cdots + \Gamma_m$ is finite index in $\Gamma$, as in Definition \ref{def: subsystems}). Firstly, we observe that the properties of hyperplane spanning and weak homogeneity are inherited:

\begin{lemma}
A scheme $\mathcal{S}$ is hyperplane spanning if and only if each subsystem is hyperplane spanning.
\end{lemma}

\begin{proof}
Let $H_X = H \cap X_H \in \sH_i$, where $H \in \sH$ and $X_H = X_i$ for some $i$. We let $V_X$ be the translate of $H_X$ over the origin (see Notation \ref{not: decomposition}). Suppose that $\mathcal{S}$ is hyperplane spanning. Then $\Gamma_<^H = (\Gamma_<) \cap V$ has linear span all of $V \coloneqq V(H)$. By assumption, $\Gamma' \coloneqq (\Gamma_1)_< + (\Gamma_2)_< + \cdots + (\Gamma_m)_<$ is finite index in $\Gamma$, so $\Gamma' \cap \Gamma^H$ still spans $V$. By Lemma \ref{lem: decomposed hyperplanes}, we may write $V = V_X + \widehat{X}_i$. Write $\widehat{\Gamma}_i$ for the sum of all $\Gamma_j$, with $j \neq i$. By definition, $(\widehat{\Gamma}_i)_< \leqslant \widehat{X}_i$, which is complementary with $V_X$. So $\Gamma' \cap V = (\Gamma_i)_< + (\widehat{\Gamma}_i)_<$ can only have linear span all of $V$ if $(\Gamma_i)_<$ has linear span all of $V_X$, since $V_X$ is complementary with $\widehat{X}_i$ in $V$. Since $H$ was arbitrary, we see that each subsystem is hyperplane spanning, as required.

Suppose, on the other hand, that each subsystem is hyperplane spanning. Then for any $H_X \in \sH_i$, the stabiliser $(\Gamma_i^{H_i})_<$ has linear span all of $V_X$. Since each $(\Gamma_j)_<$ projects densely into $X_j$, we have that $(\widehat{\Gamma}_i)_<$ spans $\widehat{X}_i$. It follows that $\Gamma_i^{H_i} + \widehat{\Gamma}_i \leqslant \Gamma$ has projection spanning all of $V = V_X + \widehat{X}_i$, so that $\mathcal{S}$ is hyperplane spanning, as required.
\end{proof}

\begin{lemma}
The scheme $\mathcal{S}$ is weakly homogeneous if and only if the subsystems $(X_i,\Gamma_i,W_i)$ are weakly homogeneous. 
\end{lemma}

\begin{proof}
First, assume that $\mathcal{S}$ is weakly homogeneous. We take $o$ as the origin in $\intl$ (which we can do without loss of generality, since homogeneity is independent of translation of $W$). Hence, for each $H \in \sH$, we have that $V = V(H) = H + (1/n_H) \cdot (\gamma_H)_<$ for some $\gamma_H \in \Gamma$. By assumption, we have that $\Gamma \leqslant \frac{1}{N} \Gamma_1 + \cdots + \frac{1}{N}\Gamma_m$ for some $N \in \N$, so that $\gamma_H = \frac{\gamma_1}{N} + \cdots + \frac{\gamma_m}{N}$ for $\gamma_i \in \Gamma_i$. Given $X_H = X_i$, by Lemma \ref{lem: decomposed hyperplanes},
\[
V_X + \widehat{X}_i = H_X + \widehat{X}_i + \left(\frac{\gamma_1}{N\cdot n_H}\right)_< + \cdots + \left(\frac{\gamma_m}{N \cdot n_H}\right)_< = H_X + \widehat{X}_i + \left(\frac{\gamma_i}{N \cdot n_H}\right)_<.
\]
The final equality above follows from the fact that $(\gamma_j)_< \in \widehat{X}_i$, for $j \neq i$. It follows that $H_X + (\gamma_i/(N \cdot n_H))_<$ contains the origin, since $\widehat{X}_i$ is complementary to $X_i \supseteq H_X$. Since $H \in \sH$ was arbitrary, we see that all $H_X \in \sH_i$ may be translated over the origin, by some element $(\gamma_i/N \cdot n_H)$, so each subsystem is weakly homogeneous.

Conversely, suppose that each subsystem is weakly homogeneous and let $H \in \sH$, with $H_X \in \sH_i$ for some $i$. Again, we assume that the origin $o$ is taken as the origin for each $W_i$ (and we can translate each $W_i$ independently, since whichever translates we choose their sum will still be a translate of $W$). There exists some $\gamma \in \Gamma_i$ and $n \in \N$ with $H_X + (\gamma/n)_<$ containing the origin. Then $H + (\gamma/n)_< = H_X + \widehat{X}_i + (\gamma/n)_<$ contains the origin. Since $H \in \sH$ was arbitrary, we see that $\mathcal{S}$ is weakly homogeneous.
\end{proof}

If one assumes that each subsystem $(X_i,\Gamma_i,W_i)$ is strictly homogeneous, then it is easy to see in the above proof that $\mathcal{S}$ is homogeneous. Conversely, if $\mathcal{S}$ is homogeneous (so we may take each $n_H = 1$), then we use the lattice elements $\gamma_i/N \in \frac{1}{N}\Gamma_i$ in the above proof. Hence, the subsystems $(X_i,\frac{1}{N}\Gamma_i,W_i)$ are homogeneous. There is little utility in saying more since, even assuming {\bf C}, we cannot assume in general that $N=1$ i.e., that $\Gamma_1 + \cdots + \Gamma_m = \Gamma$.

Recall from Section \ref{sec: subsystems} that we still have notions of cut regions for subsystems. We set their notation here:

\begin{notation} \label{not: subsystem cut regions}
For each $i = 1$, \ldots, $m$ and $r > 0$ we define the $r$-cut regions $\sC_i(r)$ of the systems $(X_i,\frac{1}{N}\Gamma_i,W_i)$ analogously to the usual cut regions (note here that we use the lattices $\frac{1}{N}\Gamma_i$, whose sum contains $\Gamma$). In more detail, the $r$-cut regions $\sC_i(r)$ are defined as the connected components of $W_i$ after removing all translates $H + \frac{\gamma_<}{N}$, where $H \in \sH_i$ and $\gamma \in \Gamma_i(r)$ (recall that $\Gamma_i(r) \coloneqq B_r(0) \cap \Gamma_i$).
\end{notation}

If it were the case that $\Gamma_1 + \cdots + \Gamma_m = \Gamma$ (rather than just being a finite index subgroup) and $N=1$, then hyperplanes $H \in \sH$ would be aligned to all but one factor $X_i$ of the decomposition. So it is easy to see that in that case the $r$-cut regions for $W$ would correspond to products of $r$-cut regions in each $W_i$ for the subsystems. In general, we can not assume that $N = 1$; however, in the next lemma we see that products of cut regions refine those of the whole window, allowing us to work with the finer lattices $\frac{1}{N}\Gamma_i$ for the cut regions. 

\begin{lemma}\label{lem: cut region products}
There is some $\lambda > 0$ satisfying the following. For each $C \in \sC(r)$, there are $C_i \in \sC_i(\lambda r)$ with $C_1 + C_2 + \cdots + C_m \subseteq C$.
\end{lemma}

\begin{proof}
Since $\Gamma_1 + \cdots + \Gamma_m$ is finite index in $\Gamma$, there exists some $\nu > 0$ so that whenever $\|\gamma_1 + \cdots + \gamma_m\| \leq r$, for $\gamma_i \in \Gamma_i$, then each $\|\gamma_i\| \leq \nu r$ (indeed, note that the linear spans of the $\Gamma_i$ give a direct sum decomposition of $\tot$).

Let $\sH^\pm$ denote the set of open half spaces $(H^+)^\circ$ and their opposites $(H^+)^c$ for $H^+ \in \sH^\pm$, and similarly for $\sH_i^\pm$. We may alternatively express $C$ as an intersection of half spaces
\[
C = \bigcap_{Z \in \sH^\pm} Z + (\gamma_Z)_<,
\]
where each $\gamma_Z \in \Gamma(r)$. By Lemma \ref{lem: decomposed hyperplanes}, we may write each $Z \in \sH_i^\pm$ as $Z = Z_X + \widehat{X}_i$, where $Z_X \in \sH_i^\pm$. Moreover, we may write each $\gamma_Z \in \Gamma(r)$ as
\[
\gamma_Z = \frac{\gamma_1}{N} + \frac{\gamma_2}{N} + \cdots + \frac{\gamma_m}{N},
\]
where $\gamma_i \in \Gamma_i$. In fact, each $\gamma_i \in \Gamma_i(N \nu r)$, since $\|\gamma_1 + \cdots + \gamma_m\| = \|N \cdot \gamma_Z\| \leq Nr$. Hence, we may rewrite the above intersection defining $C$ as
\[
C = \bigcap_{i=1}^m \bigcap_{Z_X \in \sH_i^\pm} (Z_X + \widehat{X}_i) + (\gamma_Z)_< = \bigcap_{i=1}^m \bigcap_{Z_X \in \sH_i^\pm} Z_X + \widehat{X}_i + \left(\frac{\gamma_i}{N}\right)_<,
\]
since each $(\gamma_j)_< \in \widehat{X}_i$, for $j \neq i$. Since each $X_i$ is complementary to $\widehat{X}_i$ this is equivalent to
\[
C = C_1 + C_2 + \cdots + C_m, \ \text{ where } \ C_i = \bigcap_{Z_X \in \sH_i^\pm} Z_X + \left(\frac{\gamma_i}{N}\right)_<.
\]
Each $C_i$ is an intersection of half spaces, or their opposites, from $\sH_i^\pm$, and each $\gamma_i \in \Gamma_i(N \nu r)$. It follows that each $C_i$ is either a cut region of $\sC_i(N \nu r)$, or contains such a cut region. By taking such a cut region for each $C_i$, and letting $\lambda = N \nu$, the result follows.
\end{proof}

The above shows that if each $r$-cut region in every subsystem is large, then all cut regions in the whole system are large. We will show later that this happens if we assume {\bf C} and {\bf D} for our subsystems. Low complexity is inherited from the subsystems, using the generalised complexity result of Theorem \ref{thm: generalised complexity}:

\begin{proposition} \label{prop: complexity of subsystems}
For each $i = 1$, \ldots, $m$, let $\alpha_i$ be the complexity exponent for each subsystem, calculated as stated in Theorem \ref{thm: generalised complexity}. Then the complexity exponent of the whole system is given by $\alpha = \sum_{i=1}^m \alpha_i$. Moreover, $\mathcal{S}$ satisfies {\bf C} if and only if each subsystem satisfies $\alpha_i=d_i$. 
\end{proposition}

\begin{proof}
Fix some $i = 1$, \ldots, $m$. We wish to compute $\rk(H_X)$ and $\beta_{H_X}$ for each $H_X \in \sH_i$. Firstly, we have that $\Gamma^{H_X} + \widehat{\Gamma}_i$ is a finite index subgroup of $\Gamma^H$, where $\widehat{\Gamma}_i$ is the sum of $\Gamma_j$ with $j \neq i$. To see this, first note that by assumption $\Gamma_i + \widehat{\Gamma}_i$ is finite index in $\Gamma$, so that $\rk(H) = \rk(\Gamma \cap V) = \rk((\Gamma_i+\widehat{\Gamma}_i))_< \cap V)$. By Lemma \ref{lem: decomposed hyperplanes}, $V = V(H) = V_X + \widehat{X}_i$. Since $V_X \supseteq (\Gamma_i)_<$ and $\widehat{X}_i \supseteq (\widehat{\Gamma}_i)_<$ give a sum decomposition of $V$, we have that
\[
(\Gamma_i + \widehat{\Gamma}_i)_< \cap V = ((\Gamma_i)_< \cap V_X) + ((\widehat{\Gamma}_i)_< \cap \widehat{X}_i) = (\Gamma_i^{H_X})_< + (\widehat{\Gamma}_i)_<.
\]
The last sum is of two complementary subgroups, so we see that
\begin{equation} \label{eq: rk}
\rk(H) = \rk_i(H_X) + \rk(\widehat{\Gamma}_i) = \rk_i(H_X) + (k-k_i), \ \text{ that is } \ \rk_i(H_X) = \rk(H) - (k-k_i).
\end{equation}
To determine $\dim(\langle\Gamma^{H_X}\rangle_\R)$, we note by the same reasoning as above,
\[
\langle \Gamma^H_< \rangle_\R = \langle \Gamma^{H_X}_< \rangle_\R + \widehat{X}_i,
\]
since $(\widehat{\Gamma}_i)_<$ is dense in $\widehat{X}_i$. This is a direct sum decomposition, so
\begin{equation} \label{eq: beta}
\beta_H = \beta_{H_X} + (n-n_i), \ \text{ that is } \ \beta_{H_X} = \beta_H - (n-n_i).
\end{equation}

Using Equations \ref{eq: rk} and \ref{eq: beta}, for a flag $f_i \subseteq \sH_i$ we calculate
\begin{align} \label{eq: complexity subsystem}
\alpha_{f_i} = & \sum_{H_X \in f_i} d_i - \rk(H_X) + \beta_{H_X} \\
&= \sum_{H_X \in f_i} (k_i-n_i) - (\rk(H)-(k-k_i)) + (\beta_{H}-(n-n_i))  \\
 &= \sum_{H_X \in f_i} d - \rk(H) + \beta_{H}. \nonumber
\end{align}
In other words, we get the same sum for the flag $f_i$ as one does by considering it as a subset of a flag in the whole system, replacing each $H_X \in \sH_i$ with $H \in \sH$. It is not hard to see that choosing a flag $f \subseteq \sH$ is equivalent to choosing flags $f_i \subseteq \sH_i$, one for each $i=1$, \ldots, $m$. Indeed, after repositioning the hyperplanes over the origin,
\[
\bigcap_{V \in f} V = \bigcap_{i=1}^m \bigcap_{V_X \in f_i} (V_X + \widehat{X}_i) = \bigcap_{i=1}^m Z_i + \widehat{X}_i, \ \ \text{ where } Z_i = \bigcap_{V_X \in f_i} V_X.
\]
Each $Z_i$ is trivial if and only if each $f_i$ is a flag. In this case, the above intersection is of the $\widehat{X}_i$, which is trivial, so $f$ is a flag. Conversely, if some $f_i$ is not a flag, then $Z_i$ is non-trivial. Since $Z_i \subseteq \widehat{X}_j$ for all $j \neq i$, we see that the intersection of hyperplanes is non-trivial, so $f$ is not a flag.

It follows that
\[
\alpha = \max_{f \in \sF} \alpha_f = \max_{f \in \sF} \left(\sum_{i=1}^m \alpha_{f_i}\right) = \sum_{i=1}^m \max_{f_i \in \sF_i} \alpha_{f_i} = \sum_{i=1}^m \alpha_i,
\]
where $\sF$ is the set of flags of $\sH$, $\sF_i$ is the set of flags of $\sH_i$ and, for the second equality above, given a flag $f$ we let $f_i$ be the associated flag in $\sH_i$. Since each subsystem can be regarded as an individual cut and project scheme, Corollary \ref{cor: minimal complexity} still applies and each $\alpha_{f_i} \geq d_i$. It follows that $\alpha = d$ if and only if each $\alpha_i = d_i$, since $d = d_1 + \cdots + d_m$.
\end{proof}

For minimal complexity schemes, when the subsystems have constant stabiliser ranks $r_i$ then we may determine these ranks from only the ranks of the $\Gamma_i$ and dimensions of the $X_i$:

\begin{proposition} \label{prop: rank formula}
Suppose that $\mathcal{S}$ has property {\bf C} and is equipped with a decomposition with constant stabiliser ranks. Then each such rank is given by
\[
r_i = k_i - \delta_i - 1, \ \text{ where } \delta_i \coloneqq \frac{d_i}{n_i}.
\]
\end{proposition}

\begin{proof}
By the above Proposition \ref{prop: complexity of subsystems}, each subsystem has property {\bf C}, and is also hyperplane spanning by Theorem \ref{thm: C => hyperplane spanning}. By Theorem \ref{thm: generalised complexity}, we see that
\[
\sum_{V \in f_i} (d_i - r_i +(n_i-1)) = \sum_{V \in f_i} (k_i - r_i - 1) = d_i,
\]
for each flag $f_i \subseteq \sH_i$. Since each $r_i$ is constant, and there are $n_i$ elements in a flag, we deduce
\[
n_i(k_i-r_i-1) = d_i \ \text{ that is, } \ r_i = k_i-\delta_i-1.
\]
\end{proof}

Notice that if $\mathcal{S}$ has property {\bf C} and already has constant stabiliser ranks, then by the above this rank is given by $k-\delta-1 \in \N$, and hence $n$ must divide $d$ (see also \cite[Theorem 6.7]{FHK02}). This is the case, for example, when $\mathcal S$ is indecomposable. 

We saw in Lemma \ref{lem: cut region products} that products of cut regions for the subsystems $(X_i,\frac{1}{N} \Gamma_i,W_i)$ refine the cut regions of the original scheme. In particular, if each such subsystem has large cut regions, then so does the whole system. We now consider the acceptance domains. For proving that {\bf LR} implies {\bf C} and {\bf D}, we will show that if some subsystem has a small acceptance domain (in terms of its volume) then so does the whole scheme. For this reason, for defining the acceptance domains of each subsystem, we use the subsystems $(X_i,\Gamma_i,W_i)$ (that is, using $\Gamma_i$ rather than $\frac{1}{N}\Gamma_i$), whose products of acceptance domains contain those of the whole system:

\begin{notation} \label{not: subsystem acceptance domains}
We define the $r$-acceptance domains for a subsystem $(X_i,\Gamma_i,W_i)$ in the usual way. That is, (for sufficiently large $r$) the $r$-acceptance domain $A \subseteq W_i$ containing $x \in W_i$ (for $x \notin \partial W_i + (\Gamma_i)_<$) is the intersection of $(\Gamma_i(r))_<$ translates of $W_i^\circ$ and $W_i^c$ which contain $x$ (see \cite[Corollary 3.2]{I}). The finite set of possible $r$-acceptance domains is denoted $\sA_i(r)$.
\end{notation}

\begin{lemma} \label{lem: acceptance domain products}
There is some $c > 0$ so that, for sufficiently large $r$, given $A_i \in \sA_i(r)$ for each $i=1$, \ldots, $m$, there exists $A \in \sA(r+c)$ for which $A \subseteq A_1 + A_2 + \cdots + A_m$.
\end{lemma}

\begin{proof}
Write each $A_i$ as an intersection
\[
A_i = \left(\bigcap_{\gamma \in P_i} (W_i^\circ + \gamma_<)\right) \cap \left(\bigcap_{\gamma \in Q_i} (W_i^c+\gamma_<) \right)
\]
for subsets $P_i, Q_i \subseteq \Gamma_i(r)$, where $A_i \subseteq W_i^\circ$ (as $0 \in P_i$). We now simply translate $W^\circ$ and $W^c$ by the same lattice elements and take the corresponding intersection. Since each $\gamma \in \Gamma_i \subseteq X_i$, it is not hard to see that
\begin{equation}\label{eq: product acceptance domains}
(A_i + \widehat{X}_i) \cap W^\circ = \left(\bigcap_{\gamma \in P_i} (W^\circ + \gamma_<)\right) \cap \left(\bigcap_{\gamma \in Q_i} (W^c+\gamma_<) \right) \cap W^\circ.
\end{equation}
Indeed, an element of $W^\circ$ belongs to either intersection precisely when its projection to $X_i$ is in $A_i$, since each $\gamma_< \in X_i$.

We have that
\[
A_1 + \cdots + A_m = \bigcap_{i=1}^m (A_i + \widehat{X}_i) = \bigcap_{i=1}^m ((A_i + \widehat{X}_i) \cap W^\circ),
\]
with the last equality coming from the fact that each $A_i \subseteq W_i^\circ$, so that any element in $A_1 + \cdots + A_m$ must be in $W^\circ$. By Equation \ref{eq: product acceptance domains}, the right-hand intersection is an intersection of translates of $W^\circ$ or $W^c$ by elements $\gamma_<$ for $\gamma \in \Gamma_i(r) \leqslant \Gamma(r)$. Each $A \in \sA(r)$ is a minimal intersection of such elements, possibly more, so there exists some $A \subseteq A_1 + \cdots + A_m$, as required.
\end{proof}

If we have that $\Gamma = \Gamma_1 + \cdots + \Gamma_m$, then something may be said about the the converse result of the above lemma, of finding a product of acceptance domains $A_1 + \cdots + A_m$ contained in some given $A \in \sA(r)$. However, the precise statement is slightly technical and will not be needed, so we omit it here. The implication of the above lemma that will be used in the proof of our main theorem is the following:

\begin{corollary} \label{cor: small volume}
Suppose that for some subsystem $(X_i,\Gamma_i,W_i)$, and any given $\epsilon > 0$, we may find some $r > 0$ and some acceptance domain $A_i(\epsilon) \in \sA_i(r)$ with volume $|A_i(\epsilon)| < \epsilon/r^{\alpha_i}$, where $\alpha_i$ is the complexity exponent for this subsystem. Then, for every $\epsilon > 0$, we may find an acceptance domain $A(\epsilon) \in \sA(r)$ with volume $|A(\epsilon)| < \epsilon/r^\alpha$, where $\alpha$ is the complexity exponent of $\mathcal{S}$.
\end{corollary}

\begin{proof}
For each $j \neq i$, we may find some acceptance domain $A_j \in \sA(r)$ with volume $A_j \ll 1/r^{\alpha_j}$, since $\# \sA_j(r) \asymp r^{\alpha_j}$, and the volumes of the acceptance domains $\sA_j(r)$ sum to the volume of $W_j$. Since $\{X_i\}_{i=1}^m$ is a direct sum decomposition of $\intl$, we may find some $C > 0$ so that for any $\epsilon > 0$ there are acceptance domains with
\[
|A_1 + \cdots + A_i(\epsilon) + \cdots + A_m| \leq C \epsilon \cdot \frac{1}{r^{\alpha_1}} \cdot \frac{1}{r^{\alpha_2}} \cdot \cdots \cdot \frac{1}{r^{\alpha_m} }= \frac{\epsilon}{r^\alpha}.
\]
In the above, we use the fact that $\alpha = \sum \alpha_i$ by Proposition \ref{prop: complexity of subsystems}. By Lemma \ref{lem: acceptance domain products}, there exists some $A \in \sA(r)$ with $A \subseteq A_1 + \cdots + A_m$ so that $|A| \leq C \epsilon / r^\alpha$. Since $C$ does not depend on $\epsilon > 0$, which can be taken arbitrarily small, the result follows.
\end{proof}
\begin{remark}
The notion of a system satisfying \emph{positivity of weights} ({\bf PW}) will be introduced in Section \ref{sec: LR => C and D}. With this terminology, the above theorem implies that if one subsystem of a decomposition does not satisfy {\bf PW}, then $\mathcal{S}$ does not satisfy {\bf PW}.
\end{remark}

\section{The Diophantine condition {\bf D}} \label{sec: diophantine}

Now that we have defined what it means to decompose a cut and project scheme, we can define our Diophantine condition {\bf D}. This will say that the lattice points do not project close to the origin, relative to their norm in the total space, in any subsystem. Equivalently, lattice points which lie close to the physical space $\phy$ are necessarily distant from the origin. Before specialising to cut and project sets, we first establish some basic properties for the abstract notion of a Diophantine lattice.

\subsection{Diophantine lattices}

Let $G$ be a free Abelian group of rank $k$. Call $\eta \colon G \to \R$ a {\bf lattice norm} if there exists an embedding of $G$ as a lattice into a normed vector space (of dimension $k$) so that $\eta$ is the restriction of this norm to $G$. Since norms on finite dimensional vector spaces are equivalent, and any two given lattices are linearly isomorphic, any two lattice norms $\eta_1$ and $\eta_2$ are linearly equivalent, in the sense that there exist constants $A$ and $B$ so that $\eta_1(g) \leq A \eta_2(g)$ and $\eta_2(g) \leq B \eta_1(g)$ for all $g \in G$. It is not too hard to show that a lattice norm is the same as a $\Z$-norm (satisfying the usual properties of a norm, but with $\eta(\lambda \cdot g) = |\lambda| \cdot \eta(g)$ for $g \in G$ and only $\lambda \in \Z$) which gives $G$ the discrete topology or, equivalently, is such that all $g \in G \setminus \{0\}$ have norm $\eta(g) > c$ for some $c > 0$. A standard choice of lattice norm is given by choosing a $\Z$-basis $\{g_i\}_{i=1}^k$ for $G$, and defining $\eta(\sum_{i=1}^k \lambda_i g_i) \coloneqq \max |\lambda_i|$ for $\lambda_i \in \Z$.

\begin{definition} \label{def: diophantine lattice}
Let $G$ be a free Abelian group of rank $k$, given as a dense subgroup $G \leqslant X$ of some $n$-dimensional vector space $X$. Choose any norm on $X$ and any lattice norm $\eta$ on $G$. Then we call $G$ {\bf Diophantine} if there exists some $c > 0$ so that, for all $g \in G \setminus \{0\}$, we have
\[
\|g\| \geq c \cdot \eta(g)^{-\delta}, \ \text{ where } \delta = \frac{k-n}{n}.
\]
\end{definition}

We call $\delta = (k-n)/n = d/n$, for $d \coloneqq k-n$, the {\bf Diophantine exponent}. The Diophantine property depends only on the embedding $G \leqslant X$ up to linear isomorphism, since all norms and all lattice norms are linearly equivalent.

\begin{example} \label{exp: badly approximable}
Let $\alpha \in \R$ be any irrational number. Consider the dense, rank $2$, free Abelian group $G(\alpha) \coloneqq \Z + \alpha \Z \leqslant \R$, where $\R$ has the usual norm $\|x\| \coloneqq |x|$ and $G(\alpha)$ has lattice norm $\eta(m + n\alpha) \coloneqq |m|+|n|$. Then $G(\alpha)$ is Diophantine if for all $(m,n) \in \Z^2 \setminus \{(0,0)\}$ we have
\[
|n\alpha - m| \geq \frac{c}{|m|+|n|}.
\]
We may as well only consider the above for $n\alpha - m$ small (relative to $n$), so taking $m = [n \alpha]$, the nearest integer to $n \alpha$. Since $[n \alpha] \asymp n$ we equivalently need some constant $\tau > 0$ with
\[
d(n \alpha,\Z) \geq \frac{\tau}{n},
\]
for all $n \in \N$, where $d(n \alpha,\Z)$ denotes the distance from $n \alpha$ to the nearest integer. Alternatively,
\[
\left|\alpha-\frac{m}{n}\right| \geq \frac{C}{n^2}
\]
for some constant $C > 0$ and all $m \in \Z$, $n \in \N$. This is the standard notion of a number $\alpha$ being {\bf badly approximable}, which by a classical result of Dirichlet is equivalent to $\alpha$ having continued fraction expansion with bounded entries.

More generally, given vectors $\{v_i\}_{i=1}^d$, for $v_i \in \R^n$, we may consider the group
\[
G(\alpha_1,\ldots,\alpha_d) \coloneqq \langle v_1, v_2, \ldots, v_d \rangle_\Z + \Z^n,
\]
assumed to be rank $k$ and dense in $\R^n$. This is Diophantine if and only if the system of vectors is badly approximable in the usual sense \cite{Cas65}. 
\end{example}

\

In a previous work \cite{HaynKoivWalt2015a} for `cubical' cut and projects, there was a natural choice for an integer reference lattice, so that Diophantine notions such as badly approximable could be understood in terms of that frame of reference. In our current more general setting of polytopal cut and project sets there is no canonical candidate. Instead, given $G \leqslant X$, we may choose subgroups $G'$, $Z \leqslant G$ of ranks $k-n$ and $n$, respectively, so that the `reference lattice' $Z$ spans $X$ and $G = G'+Z$. There is a linear map $A \colon X \to \R^n$, taking $Z$ to $\Z^n$, and $G'$ to some rank $k-n$ subgroup $A(G') = G(\alpha_1,\ldots,\alpha_d)$ for some $\alpha_i \in \R^n$. Since $(G \leqslant X)$ and $(G(\alpha_1,\ldots,\alpha_d) + Z \leqslant \R^n)$ are linearly isomorphic, one is Diophantine if and only if the other is. We see that the above example is essentially the general case. However, the choice of the reference lattice $Z$ is inconsequential to the Diophantine property, so the most elegant solution is to not choose one. 

We note that the Diophantine property is stable under taking finite index super- or subgroups:

\begin{lemma} \label{lem: Diophantine finite index}
Let $G \leqslant K$, where $G$ and $K$ are free Abelian groups of equal rank $k$, densely embedded in some Euclidean space $X$. Then $G$ is Diophantine if and only if $K$ is.
\end{lemma}

\begin{proof}
Take any lattice norm $\eta$ on $K$. We may take the restriction of $\eta$ as the lattice norm on $G$. Clearly if the larger subgroup $K$ is Diophantine, then so is the smaller group $G$, taking the same Diophantine constant $c$.

For the converse, suppose that $G$ is Diophantine. Since $G$ and $K$ are equal rank, $G$ is finite index in $K$, so there exists some $N \in \N$ with $N \cdot K \leqslant G$. So, given any $g \in K$, we have that
\[
\|g\| = \frac{1}{N}\|N \cdot g\| \geq \frac{1}{N} c \cdot \eta(N \cdot g)^{-\delta} = \left(\frac{c}{N^{1+\delta}}\right) \eta(g),
\]
so $K$ is Diophantine with constant $c/N^{1+\delta}$.
\end{proof}

We have the following simple result, similar to the classical Dirichlet Approximation Theorem, although weaker in the sense that the constant depends on $G$ (a more closely related statement could involve choosing a reference lattice as outlined above serving as the integer lattice $\Z^n \leqslant \R^n$, but we prefer to give the following, simpler statement). It shows that the Diophantine exponent $\delta$ is optimal:

\begin{lemma} \label{lem: Dirichlet}
For $G \leqslant X$ and $\delta$ as in Definition \ref{def: diophantine lattice}, there is a constant $c > 0$ so that
\[
\|g\| \leq c\cdot \eta(g)^{-\delta} \text{ for infinitely many } g \in G.
\]
\end{lemma}

\begin{proof}
Choose any two complementary subgroups $G'$, $Z \leqslant G$ such that $G' + Z = G$ and $\rk (Z) = n$, with $Z$ spanning $X$. Choose any basis $\{z_i\}_{i=1}^n$ for $Z$, with associated fundamental domain $F$. Given any $g \in G'$, there is some $z_g \in Z$ with $g-z_g \in F$. We have that $\eta(z_g) \ll \|g\| \ll \eta(g)$, since $Z$ has a basis which is also a basis for $X$. Since there are $\gg r^d$ (where $d = k-n$) elements $g \in G'$ with $\eta(g) \leq r$ it follows that there are $\gg r^d$ elements $g-z_g \in G \cap F$ with $\eta(g-z_g) \ll r$. So we may find some $\kappa > 0$ for which there are at least $\kappa r^d$ elements $g \in G \cap F$ with $\eta(g) \leq \frac{r}{2}$ for sufficiently large $r$.

Since $F$ is a fixed, $n$-dimensional fundamental domain, we may cover it by strictly less than $\kappa r^d$ balls of radius $(\frac{c}{2}) \cdot r^{-\delta} = (\frac{c}{2}) \cdot r^{-d/n}$ for some $c > 0$. By the Pigeonhole Principle, there is some ball which contains two distinct elements $g$, $h \in G \cap F$ with $\eta(g)$, $\eta(h) \leq \frac{r}{2}$. So $g-h \in G$, $\|g-h\| \leq cr^{-\delta}$ and $\eta(g-h) \leq r$. Since $r$ can be made arbitrarily large, the result follows.
\end{proof}

Sums of Diophantine lattices are still Diophantine when they have the same exponent:

\begin{lemma} \label{lem: Diophantine sums} 
Take two finite rank, free Abelian dense subgroups $G_1 \leqslant X_1$ and $G_2 \leqslant X_2$ of vector spaces $X_1$ and $X_2$ with Diophantine exponents $\delta_1$ and $\delta_2$, respectively. If $\delta_1 \neq \delta_2$, then $G_1 + G_2 \leqslant X_1 + X_2$ is not Diophantine. If $\delta_1 = \delta_2$, then $G_1 + G_2 \leqslant X_1 + X_2$ is Diophantine if and only if both $G_1 \leqslant X_1$ and $G_2 \leqslant X_2$ are Diophantine.
\end{lemma}

\begin{proof}
Suppose that $\delta_1 \neq \delta_2$; without loss of generality, $\delta_1 > \delta_2$. Denote $\rk(G_1) = k_1$, $\rk(G_2) = k_2$, $\dim(X_1) = n_1$, $\dim(X_2) = n_2$ and $d_i = k_i-n_i$, so that $\delta_i = \frac{d_i}{n_i}$. The required Diophantine exponent $\delta$ for $G_1+G_2$ is $\frac{d_1+d_2}{n_1+n_2} < \delta_1$. By the above Lemma \ref{lem: Dirichlet}, there exist infinitely many elements $g \in G_1 \setminus \{0\}$ with $\|g\| \leq c \eta(g)^{-\delta_1}$ for some $c > 0$. Since $c r^{-\delta_1} < \epsilon r^{-\delta}$ for arbitrarily small $\epsilon$, for sufficiently large $r$, these elements also rule out $G_1 + G_2$ from being Diophantine.

Suppose then that $\delta_1 = \delta_2$. Hence the Diophantine exponent $\delta$ for $G_1 + G_2 \leqslant X_1 + X_2$ is also $\delta_1 = \delta_2 = \frac{d_1+d_2}{n_1+n_2}$. Suppose that one of $G_1$ or $G_2$ is not Diophantine; without loss of generality, $G_1$ is not Diophantine. Hence, for all $\epsilon > 0$, there exists some $g \in G_1 \setminus \{0\}$ with $\|g\| < \epsilon \eta(g)^{-\delta}$. The corresponding elements of $G_1 + G_2$ show that this group is also not Diophantine. Conversely, suppose that both $G_1$ and $G_2$ are Diophantine. Since $X_1$ and $X_2$ are complementary in the sum, we have that $\|g_1 + g_2\| \geq \kappa_1 \cdot \max\{\|g_1\|,\|g_2\|\}$ and $\max\{\eta(g_1),\eta(g_2)\} \leq \kappa_2 \cdot \eta(g_1+g_2)$ for some $\kappa_1$, $\kappa_2 > 0$, for all $g_1 \in G_1$, $g_2 \in G_2$.

Take any non-trivial $g_1 + g_2 \in G_1 + G_2$. If $g_2 = 0$, then using that $G_1$ is Diophantine:
\[
\|g_1 + g_2\| = \|g_1\| \geq c_1 \eta(g_1)^{-\delta} = c_1 \cdot \eta(g_1+g_2)^{-\delta},
\]
since $\delta_1 = \delta_2 = \delta$. Similarly, if $g_1 = 0$ then $\|g_1+g_2\| \geq c_2 \cdot \eta(g_1+g_2)^{-\delta}$ for some $c_2 > 0$. If $g_1$ and $g_2 \neq 0$ then
\begin{align*}
\|g_1+g_2\| \geq \kappa_1 \max\{\|g_1\|,\|g_2\| \} \geq \kappa_1 \max\{c_1 \eta(g_1)^{-\delta_1}, c_2 \eta(g_2)^{-\delta_2}\} \geq \\
\kappa_1 \min\{c_1,c_2\} \min\{\eta(g_1)^{-\delta},\eta(g_2)^{-\delta}\} = \kappa_1 c \cdot \max\{\eta(g_1),\eta(g_2)\}^{-\delta} \geq (\kappa_1 \kappa_2 c) \cdot \eta(g_1+g_2)^{-\delta},
\end{align*}
where $c = \min\{c_1,c_2\}$ and we use $\delta_1 = \delta_2 = \delta > 0$. So $G_1+G_2$ is Diophantine with constant $\kappa_1 \kappa_2 c$, as required.
\end{proof}

\subsection{Transference}
The Diophantine property ensures that elements with small lattice norm are reasonably distant from the origin. This implies a dual property: all group elements with norm less than $r$ do not leave large gaps between them on a bounded subset, relative to $r$. In the classical setting, of families of vectors in $\R^n$ relative to the integer lattice $\Z^n$, this is sometimes known as \emph{transference}; see, for example \cite{Cas65} (although we note that the term more usually refers to Diophantine properties of a system of vectors and the transpose system). Given our alternative formulation, in terms of Diophantine lattices, we will restate this property and recall the standard proof of it, using Minkowski's Second Theorem. Throughout, we assume that $G \leqslant X$ and $\eta$ are as in Definition \ref{def: diophantine lattice}.

\begin{definition}
Call $G \leqslant X$ {\bf densely distributed} if there is some $c > 0$ so that, for all $r > 0$, every point in the unit ball of $X$ is within distance $c \cdot r^{-\delta}$ of some $g \in G$ with $\eta(g) \leq r$.
\end{definition}

In the above, $\delta$ is the Diophantine exponent $\delta = \frac{k-n}{n}$, as in Definition \ref{def: diophantine lattice}. As usual, the choice of norm on $X$ does not make a difference, up to a change of the constant $c$. The choice of the unit ball above is also arbitrary. Given any other bounded region $B$, since $G$ is dense in $X$, we can use finitely many $G$-translates $g_i \in G$ applied to the unit ball to cover $B$. So all points of $B$ are within distance $c \cdot r^{-\delta}$ of some element $g + g_i$ with $\eta(g) \leq r$, hence $\eta(g+g_i) \leq \eta(g) + \eta(g_i) \leq r + \tau$ for some constant $\tau$. After reparametrising, we thus have:

\begin{lemma} \label{lem: transference general regions}
If $G \leqslant X$ is densely distributed, then there is some $c > 0$ satisfying the following: for any given bounded region $B$ we have that every $b \in B$ is within distance $c \cdot r^{-\delta}$ of some $g \in G$, with $\eta(g) \leq r$.
\end{lemma}

There are of order $r^{k-n}$ elements of $G$ with lattice norm bounded above by $r$ in the unit ball of $X$. On the other hand, one may also find of order $r^{k-n}$ disjoint balls of radius $r^{-\delta}$ contained in the unit ball. So being densely distributed is a special property, implying that $G$ fills $X$ efficiently with respect to the lattice norm. The following transference theorem says that this property follows from the Diophantine property:

\begin{theorem} \label{thm: transference}
If $G \leqslant X$ is Diophantine then it is densely distributed.
\end{theorem}

We prove the above through two lemmas, where we temporarily make use of the viewpoint of $G \leqslant X$ being the projection of a lattice to a lower dimensional subspace. This may always be done, as follows. Choose some basis $\{g_i\}_{i=1}^k$ of $G$ and consider the projection of $\R^k$ to $X$ defined by $e_i \mapsto g_i$, where $e_i$ is the $i$th standard basis vector. This maps the integer lattice $\Z^k$ isomorphically to $G$. We will use analogous notation for this projection as the projection to internal space for cut and project schemes.

\begin{lemma} \label{lem: dense lattice}
Let $L$ be a lattice in a normed vector space $E$. Suppose that $\|g\| \geq s > 0$ for all $g \in L \setminus \{0\}$. Then for all $z \in E$ there is some $g \in L$ such that
\[
\|z - g\| \leq c \cdot s^{1- \dim E} \covol(L)
\]
where the constant $c$ only depends on $\dim E$.
\end{lemma}

\begin{proof}
The $i$th successive minimum in Minkowski's Second Theorem is defined as
\[
\lambda_i = \inf\{\lambda > 0 \mid \text{there exist linearly independent } (g_1,\ldots,g_i) \text{ with each } \|g_i\| \leq \lambda\}.
\]
Then clearly $s \leq \lambda_1 \leq \lambda_2 \leq \cdots \leq \lambda_k$, where $k = \dim E$. Minkowski's Second Theorem provides constants $0<A<B$, depending only on $\dim E$, such that
\[
A \covol(L) \leq \lambda_1 \lambda_2 \cdots \lambda_k \leq B \covol(L).
\]
It follows that
\[
\lambda_k \leq \frac{B \covol(L)}{\lambda_1 \lambda_2 \cdots \lambda_{k-1}} \leq B \cdot s^{1-k} \covol(L).
\]
In other words, there is a sublattice of $L$ with basis $\{b_i\}_{i=1}^k$ with each $\|b_i\| \leq B \cdot s^{1-k} \covol(L)$. This basis defines a fundamental domain $F$ whose translates by $\Z$-sums of the $b_i$ cover $E$. Moreover, $F$ has diameter bounded above by $\|b_1\| + \cdots + \|b_k\| \leq k \cdot B \cdot s^{1-k} \covol(L)$. Hence, every element of $E$ is within this distance of a lattice element, as required.
\end{proof}

The above result shows that, for lattices of fixed covolume, if there are no short lattice elements then the lattice is reasonably dense. The following standard argument then shows that if the projection of the lattice to some `internal space' $X$ keeps points distant from the origin, then those projecting to some reasonably large ball in a complementary `physical space' project densely to the `internal space':

\begin{lemma}\label{lem: general transference}
Fix a normed vector space $E$ with complementary subspaces $P$, $X \leqslant E$. Given $z = (p+x) \in E$, write $z_\vee = p$ and $z_< = x$. We denote $k = \dim E$, $n = \dim X$ and $d = k-n = \dim P$. Finally, fix a lattice $L \leqslant E$. Then there are constants $A$, $B$ satisfying the following: Suppose that $\|g_<\| \geq s$ for all $g \in L \setminus \{0\}$, with $\|g\| \leq r$. Then for all $p \in P$ and $x \in X$, there is some $g \in L$ with
\begin{itemize}
	\item $\|p-g_\vee\| \leq A \cdot r^{-(d-1)} \cdot s^{-n}$, and
	\item $\|x-g_<\| \leq B \cdot r^{-d} \cdot s^{-(n-1)}$.
\end{itemize}
\end{lemma}

\begin{proof}
Define a linear map $M \colon E \to E$ by setting $M(p+x) \coloneqq \frac{p}{r} + \frac{2x}{s}$ for $p \in P$ and $x \in X$. Since all norms are linearly equivalent, we may as well choose norm on $E$ coming from an inner product making $P$ and $X$ orthogonal, so that
\[
\det M = r^{-d}\cdot 2^n \cdot s^{-n}, \ \text{ hence } \ \covol M(L) = (r^{-d} \cdot 2^n \cdot s^{-n}) \covol L.
\]
Let $g \in L \setminus \{0\}$ with $\|g\| \leq r$. Then
\[
\|M(g)\| = \|M(g)_\vee + M(g)_<\| = \left\|\frac{g_\vee}{r} + \frac{2 g_<}{s}\right\| \geq \frac{2\|g_<\|}{s} - \frac{\|g_\vee\|}{r} \geq 1,
\]
since by assumption $\|g_<\| \geq s$ and $\|g_\vee\| \leq \|g\| \leq r$. Take any $p \in P$ and $x \in X$ and consider the point $M(p+x) = \frac{p}{r}+\frac{2x}{s}$. By the above, may apply Lemma \ref{lem: dense lattice} to the lattice $M(L)$ with $s=1$, so there is a constant $c > 0$ (depending only on $k$) and some $g \in L$ for which
\[
\left\|\left(\frac{p}{r}+\frac{2x}{s}\right) - M(g)\right\| \leq c \cdot \covol M(L).
\]
We have that $M(g) = \frac{g_\vee}{r} + \frac{2 g_<}{s}$, so by orthogonality of $P$ and $X$:
\[
\left\|\left(\frac{p}{r}+\frac{2x}{s}\right) - M(g)\right\| = \left\|\frac{p}{r} - \frac{g_\vee}{r}\right\| + \left\|\frac{2x}{s} - \frac{2 g_<}{s}\right\| = \frac{1}{r} \|p-g_<\| + \frac{2}{s}\|x-g_<\|.
\]
Since each summand is bounded above by $c \cdot \covol M(L)$, we obtain:
\[
\frac{1}{r}\|p-g_<\|, \ \frac{2}{s}\|x-g_<\| \leq c \cdot (r^{-d} \cdot 2^n \cdot s^{-n}) \covol L
\]
The result follows by multiplying the first term by $r$, and the second by $\frac{s}{2}$.
\end{proof}

Applying the above to the case where the projected lattice is Diophantine proves the transference result:

\begin{proof}[Proof of Theorem \ref{thm: transference}]
Let $G \leqslant X$ be Diophantine. As remarked following the statement of Theorem \ref{thm: transference}, up to linear isomorphism $G \leqslant X$ may be given as the projection of a lattice $L \leqslant E$, for a $k$-dimensional vector space $E$, onto an $n$-dimensional vector space $X$. We may regard $X$ as a subspace of $E$ (up to linear isomorphism), by setting $P$ as the kernel of the projection and identifying $X$ with any subspace complementary to $P$. Choose such an extension to a `total space' and any norm on $E$, whose restriction to $L$ gives a lattice norm on $G = L_<$, setting $\eta(g_<) \coloneqq \|g\|$ for $g \in L$.

Choose any $b$ in the unit ball of $X$ and $r > 0$. Since $G \leqslant X$ is Diophantine, there is some constant $c > 0$ so that $\|g_<\| \geq c \cdot \eta(g)^{-\delta} = c \cdot \|g\|^{-\delta}$ for all $g \in L \setminus \{0\}$, where $\delta = \frac{d}{n}=\frac{k-n}{n}$. Let $s \coloneqq c \cdot r^{-\delta}$ and apply Lemma \ref{lem: general transference} on $p \coloneqq 0$, $x \coloneqq b$ so that
\begin{itemize}
	\item $\|g_\vee\| \leq A \cdot r^{-(d-1)}\cdot (c \cdot r^{-\delta})^{-n} = (Ac^{-n})\cdot r$;
	\item $\|b-g_<\| \leq B \cdot r^{-d} \cdot (c \cdot r^{-\delta})^{-(n-1)} = (Bc^{-(n-1)}) \cdot r^{-\delta}$,
\end{itemize}
for some $g \in L$. Since $b$ belongs to the unit ball, we have that $\|b-g_<\| \leq \tau$ for some constant $\tau > 0$. So we see from the bound on $\|g_\vee\|$ that $\|g\| \leq Cr$ for some $C > 0$. That is, we may find $g \in L$ with $\|g\| \leq Cr$ and $\|b-g_<\| \leq D \cdot r^{-\delta}$ for some constant $D$. Rescaling $r$ and phrasing in terms of $G = L_<$, we have that for all $r > 0$, there exists $\gamma \in G$ with $\eta(\gamma) \leq C(\frac{r}{C}) = r$ and $\|b-\gamma\| \leq D \cdot (\frac{r}{C})^{-\delta} = (D C^\delta)r^{-\delta}$. It follows that $G$ is densely distributed.
\end{proof}

\subsection{The property {\bf D} for cut and project schemes}

We now define what it means for a cut and project scheme to have the Diophantine property {\bf D}:

\begin{definition} \label{def: diophantine scheme}
A subsystem $(X_i,\Gamma_i,W_i)$ is defined to be {\bf Diophantine} if $(\Gamma_i)_< \leqslant X_i$ is Diophantine. We say that a cut and project scheme $\mathcal{S}$ is {\bf Diophantine}, or satisfies {\bf D}, if it has a decomposition whose subsystems are each Diophantine.
\end{definition}

We emphasise that in the above, each subsystem is required to have constant hyperplane stabiliser rank (see Corollary \ref{cor: decompose}). Without this assumption, the Diophantine condition on the whole scheme would not be appropriate. If $\mathcal{S}$ is decomposable but already has constant hyperplane stabiliser rank, then we still allow ourselves to call it Diophantine if $\Gamma_<$ is Diophantine in the internal space, without further decomposition.

There may be several decompositions for a cut and project scheme which have constant hyperplane stabiliser ranks. As one would hope, which decomposition is chosen is inconsequential:

\begin{lemma}
Suppose that $\mathcal{S}$ satisfies {\bf C} and has two decompositions $Y$ and $Z$ with constant hyperplane stabiliser rank. Then each subsystem of $Y$ is Diophantine if and only if each subsystem of $Z$ is Diophantine.
\end{lemma}

\begin{proof}
Denote the minimal decomposition into indecomposables by $\{X_i\}_{i=1}^m$, with associated lattice factors $\{\Gamma_i\}_{i=1}^m$. It follows from Proposition \ref{prop: common refinement} that such a decomposition exists, with all other decompositions given by taking direct sums of the $X_i$. For example, suppose that $Y \in \{Y_i\}_{i=1}^\ell$ is a decomposition. Without loss of generality, suppose that $(Y_i,G,W^Y_i)$ is a subsystem, where
\[
Y_i = X_1 + \cdots + X_p.
\]
Each $X_i$ has an associated Diophantine exponent $\delta_i \coloneqq \frac{d_i}{n_i}$. In fact, these exponents must all be equal. Indeed, by Proposition \ref{prop: rank formula}, we have that $r_i = k_i - \delta_i - 1$. An arbitrary hyperplane in the subsystem for $Y_i$ is given by
\[
H = X_1 + \cdots + X_{\ell-1} + H_\ell + X_{\ell+1} + \cdots + X_p
\]
for some $H_\ell \in \sH_\ell$. Up to finite index $G$ is the sum $\Gamma_1 + \cdots + \Gamma_m$. The subgroup of this sum stabilising $H$ is given by
\[
\Gamma_1 + \cdots + \Gamma_{\ell-1} + \Gamma^{H_\ell}_\ell + \Gamma_{\ell+1} + \cdots + \Gamma_p,
\]
which has rank
\[
k_1 + k_2 + \cdots + k_{\ell-1} + r_\ell + k_{\ell+1} + \cdots + k_p = (k_1 + \cdots + k_p) - \delta_\ell - 1.
\]
Since these numbers are all required to be equal, we see that each $\delta_\ell$ must be equal, or else the subsystem corresponding to $Y_i$ does not have constant hyperplane stabiliser rank. It then follows from Lemma \ref{lem: Diophantine sums} that the subsystems $(G,Y_i,W^Y_i)$ are Diophantine if and only if each for indecomposable subsystem, $(\Gamma_j \leqslant X_j)$ is Diophantine.
\end{proof}

\section{Proof that {\bf LR} implies {\bf C} and {\bf D}} \label{sec: LR => C and D}
We are now ready to prove the first direction of our main theorem. Given an $r$-patch $P$, we let $\xi_P \in (0,1]$ be its associated frequency. This is defined as follows. Firstly, we let
\[
\xi_P' = \lim_{R \to \infty} \frac{\{z \in \Lambda \cap B_R(0) \mid P(z,r) = P \text{, up to translation} \}}{|B_R(0)|},
\]
where $|\cdot|$ denotes Lebesgue measure (so the denominator above is $cR^d$, for the constant $c = |B_1(0)|$). That is, to approximate $\xi_P'$, one counts occurrences of $P$ in a large $R$-ball (choosing its centre at the origin is arbitrary and inconsequential), dividing by the volume of the $R$-ball. It is a consequence of unique ergodicity that the limit defining $\xi'_P$ always exists. Finally, we let $\xi_P \coloneqq \xi_P' / \lambda$, with $\lambda$ the density of $\Lambda$, chosen to normalise frequencies so that $\sum_{P \in \{r\text{-patches}\}} \xi_P = 1$ for any $r > 0$.

For cut and project sets, we have an important connection with the acceptance domains:
\[
\xi_P = |A_P|/|W|,
\]
which follows from strict ergodicity, see for example \cite[Section 3]{HaynKoivSaduWalt2015}. It follows that if one wishes to show that there is some $r$-patch of low frequency (relative to $r$), it is equivalent to construct an acceptance domain of small volume. The existence of patches of small frequencies rules out linear repetitivity, as we shall now explain.

We say that $\Lambda$ (or the associated scheme $\mathcal{S}$) satisfies {\bf positivity of weights} ({\bf PW}) if there exists some $C > 0$ for which
\[
\xi_P \geq \frac{C}{r^d}
\]
for all $r$-patches $P$, for sufficiently large $r$. This terminology is used in \cite{BBL}, where it is shown that for FLC patterns, {\bf LR} is equivalent to {\bf PW} and another property {\bf U} (\emph{uniformity of return words}) both holding. It is unknown whether {\bf PW} is sufficient for {\bf LR} without assuming {\bf U} for general FLC patterns of dimension $d > 1$, although we will see later that {\bf LR} and {\bf PW} turn out to be equivalent for the class of cut and project sets dealt with by our main theorem. In any case, we have the following for general Delone sets or tilings, and will sketch the elementary proof:

\begin{lemma} \label{lem: LR => PW}
{\bf LR} implies {\bf PW}.
\end{lemma}

\begin{proof}
Suppose that $\rho(r) \leq Cr$ (that is, every $r$-patch appears at least once with centre in any given $Cr$-ball). We may place disjoint $Cr$-balls (for example, by placing them near the centres of an array of cubes), centred at points of $\Lambda$ (which is relatively dense), and densely enough so that the number of such balls contained in any $R$-ball is $\geq c (R/r)^d$, for some $c > 0$ (not depending on $r$ or $R$) and $R$ sufficiently large relative to $r$. Since any given $r$-patch appears at least once for each $Cr$-ball, the number of appearances of $P$ in an $R$-ball is $\geq c (R/r)^d$. Since $|B_R(0)| \asymp R^d$, it follows that $\xi_P' \gg r^{-d}$, so $\xi_P \geq \frac{\kappa}{r^d}$ for some $\kappa > 0$ and all $r$-patches $P$ for sufficiently large $r$.
\end{proof}

\begin{lemma} \label{lem: PW => C}
{\bf PW} implies {\bf C}.
\end{lemma}

\begin{proof}
Since $\xi_P \geq C/r^d$ for sufficiently large $r$ and some $C$, we have
\[
\frac{C \cdot p(r)}{r^d} = \sum_{r\text{-patches } P} \frac{C}{r^d} \leq \sum_{r\text{-patches } P} \xi_P = 1,
\]
since the sum of frequencies of $r$-patches is always equal to $1$. Hence $p(r) \leq r^d/C$.
\end{proof}

Combining the above two lemmas, we see that {\bf LR} implies {\bf C}. Of course, regardless of consideration of frequencies, this is clear anyway. Indeed, suppose that $\rho(r) \leq Cr$, so each $r$-patch appears at least once with centre in every $Cr$-ball. The number of possible centres in a $Cr$-ball is $\asymp r^d$ (since $\Lambda$ is a Delone set), so that $p(r) \ll r^d$. This is somewhat superfluous, since it turns out \cite{LP03} that {\bf LR} implies uniform patch frequencies, but we record it for convenience:

\begin{corollary} \label{cor: LR => C and PW}
{\bf LR} implies {\bf C} and {\bf PW}.
\end{corollary}

Hence, to show that {\bf LR} implies {\bf C} and {\bf D}, it is sufficient to show that {\bf PW} implies {\bf D}. We prove the contrapositive of this in the proof below:

\begin{proposition} \label{prop: PW => D}
Suppose that $\mathcal{S}$ satisfies {\bf PW} and has constant hyperplane stabiliser rank. Then $\mathcal{S}$ also satisfies {\bf D}.
\end{proposition}

\begin{proof}
Suppose, to the contrary, that $\mathcal{S}$ does not satisfy {\bf D}. We will use the failure of {\bf D} to construct acceptance domains of small volume.

Let $f \subseteq \sH$ be any flag and $\epsilon > 0$ be arbitrary. Since {\bf D} does not hold, we may find some $\gamma \in \Gamma(r)$ with
\[
\|\gamma_<\| \leq \frac{\epsilon}{r^\delta},
\]
where $\delta = d/n$. For some $K \in f$, we have that $\gamma \notin \Gamma^K$ (since $f$ is a flag), so $K \neq K + \gamma_<$. Exchanging $\gamma$ with $-\gamma$ if necessary, we have that the `strip'
\[
S_K \coloneqq (K^+)^\circ \cap ((K^+)^c + \gamma_<)
\]
is non-empty. The set $S_K$ is a thickening of $K$, with thickness bounded above by $\|\gamma_< \|$. In fact, since $\gamma_<$ can be made arbitrarily small (by picking $\epsilon$ sufficiently small), the set
\[
A_K \coloneqq W^\circ \cap (W^c+\gamma_<)
\]
agrees with $S_K$ to some reasonable fixed radius $\lambda$ (depending only on $W$) about the centre of the face $\delta_K$ of $W$ in $K$. Indeed, $W$ agrees with $K^+$ near to the centre of $\delta_K$.

Our strategy is to now construct similar sets $A_H$ for each $H \in \sH$ which are sufficiently thin. Their intersection will be a small subset of the window containing an acceptance domain.

We may construct a small parallelepiped $P$ (but chosen independently of $\epsilon$), with faces aligned with those of $f$, using translates of (the interiors of) the half-spaces $H^+$, for $H \in f$, and the complements of these half spaces, translated by elements of $\Gamma_<(c)$ for some constant $c$. Translating by an element of $\Gamma_<$ if necessary, we may also assume that $P$ is close to the centre of $W$.

Without loss of generality, we assume that {\bf C} is satisfied (since if it is not, then {\bf PW} does not hold by Lemma \ref{lem: PW => C}). By assumption, each $\rk(H)$ has a common value $\tau$, and in fact $\tau = k-\delta-1$ by Proposition \ref{prop: rank formula}, where $\delta = \frac{d}{n}$.

By Theorem \ref{thm: C => hyperplane spanning}, $\mathcal{S}$ is hyperplane spanning. We then recall from \cite[Section 7.3]{I} that there are $\asymp r^{k-\tau-1} = r^\delta$ `nice cuts' of $H$ with $P$. More precisely, there are $\gg r^\delta$ translates $H + \gamma_<$, with $\gamma \in \Gamma(r)$, so that $W+\gamma_<$ and $H^+ + \gamma_<$ intersect $P$ identically (so the translated face of $W$ cuts the whole way through $P$, but no other such faces intersect $P$). It follows that, for each $H \in f$ (with $H \neq K$), we may find distinct translates $W^\circ + \alpha_<$, $W^c + \beta_<$, with $\alpha$, $\beta \in \Gamma(r)$, with
\[
A_H \coloneqq P \cap (W^\circ + \alpha_<) \cap (W^c + \beta_<) = P \cap (H^\circ + \alpha_<) \cap (H^c + \beta_<) \neq \emptyset,
\]
which as a subset of $P$ is a strip of width $\ll r^{-\delta}$, since we have $\gg r^\delta$ distinct hyperplanes cutting $P$ to choose from to bound $A_H$. Translate the original thin slice $A_K$, if necessary, so as to pass through $P$, which only requires an element of $\Gamma_<$ of lattice norm bounded by a constant not depending on $\epsilon$. Then the intersection
\[
A \coloneqq \bigcap_{H \in f} A_H
\]
is a parallelepiped, with faces parallel to the hyperplanes in $f$, and given as an intersection of translates of $W^\circ$ and $W^c$ by elements $\gamma_< \in \Gamma_<(r+c)$, for $r$ sufficiently large and some $c$ not depending on $\epsilon$, introduced when translating $A_K$ to intersect $P$. Therefore, $A$ contains some acceptance domain $A' \in \sA(r+c)$. By construction, $A$ has widths $\ll r^{-\delta}$ in each direction, and width $\ll \epsilon r^{-\delta}$ in the direction normal to $K$, with constants not depending on $\epsilon$. It follows that there is some $C > 0$, not depending on $\epsilon$, with $|A'| \leq |A| \leq C \epsilon r^{-\delta}(r^{-\delta})^{n-1} = C\epsilon r^{-d}$. Since $\xi_P = |A'|/|W| \leq \epsilon (\frac{C}{|W|}) r^{-d}$, for some $(r+c)$-patch $P$ and $\epsilon > 0$ arbitrary, we see that $\mathcal{S}$ does not satisfy {\bf PW}, as required.
\end{proof}

\begin{corollary} \label{cor: PW => C+D}
{\bf PW} implies {\bf C} and {\bf D}.
\end{corollary}

\begin{proof}
Suppose that $\mathcal{S}$ satisfies {\bf PW}, so it satisfies {\bf C} by Lemma \ref{lem: PW => C}. By Corollary \ref{cor: decompose}, there is a decomposition of $\mathcal{S}$ where each subsystem has hyperplane stabiliser subgroups of constant rank. Suppose that some subsystem does not satisfy {\bf D}. The proof of the above proposition applies identically to each subsystem, so we see that some subsystem does not satisfy {\bf PW}. But then, by Corollary \ref{cor: small volume} (and see the following remark), the original scheme $\mathcal{S}$ does not satisfy {\bf PW}.
\end{proof}

Since {\bf LR} implies {\bf PW}, by Corollary \ref{cor: LR => C and PW}, the above establishes the first direction of our main theorem:

\begin{corollary}
{\bf LR} implies {\bf C} and {\bf D}.
\end{corollary}

Finally, we note that the above theorem assumes only that $\mathcal{S}$ is aperiodic and does not require the (weakly) homogeneous condition. However, the converse is not true in general without assuming weak homogeneity.

\section{Vertices of cut regions} \label{sec: vertices}

We wish to show that {\bf C} and {\bf D} together imply {\bf LR}. It will be sufficient to work with the cut regions instead of acceptance domains. In this section, we shall establish results on the vertices of the cut regions by relating them to the projected lattice $\Gamma_<$. It will turn out that, assuming minimal complexity {\bf C} and (weak) homogeneity, displacements between vertices of cut regions are contained in $\frac{1}{M}\Gamma_<$, for some $M \in \N$. The Diophantine condition {\bf D} then ensures that the vertices of the cut regions do not get too close to each other, which makes each cut region reasonably large. Finally, the transference Theorem \ref{thm: transference} guarantees that the projected lattice is sufficiently dense for orbits to visit each cut region.

\subsection{Vertex sets}
We first establish a link between the vertex sets of cut regions and the projected lattice for minimal complexity schemes. We introduce the following {\bf vertex sets}, which we shall explain in more detail following the definition.

\begin{definition} \label{def: vertex sets}
Let $\mathcal{S}$ be a polytopal cut and project scheme. Denote the set of flags in $\sH$ by $\sF$. Given a subset $G \subseteq \Gamma$ and a flag $f \in \sF$, we denote
\[
\cV(G,f) \coloneqq \bigcap_{H \in f} (H + G_<), \ \ \cV(G) \coloneqq \bigcup_{f' \in \sF} \bigcap_{H \in f'} (H + G_<) = \bigcup_{f' \in \sF} \cV(G,f').
\]
If $\mathcal{S}$ has a decomposition $\{X_i\}_{i=1}^m$, we define analogous subsets for the subsystems: for $i = 1$, \ldots, $m$, $G \subseteq \Gamma_i$ and a flag $f \subseteq \sF_i$ (where $\sF_i$ is the set of flags of $\sH_i$). We write
\[
\cV_i(G,f) \coloneqq \bigcap_{H \in f} (H + G_<), \ \ \cV_i(G) \coloneqq \bigcup_{f' \in \sF_i} \bigcap_{H \in f'} (H + G_<) = \bigcup_{f' \in \sF_i} \cV_i(G,f').
\]
\end{definition}

Let us explain the relevance of these sets. Call a point $v$ in the internal space a {\bf vertex} if it may be given as the unique intersection point of some $\Gamma_<$-translates of hyperplanes in $\sH$. Then we may as well restrict these hyperplanes to a flag, so $v$ is a vertex if and only if $v \in \cV(\Gamma)$. Vertices of cut regions are examples of vertices, but we also count intersection points not lying inside the window. It will be of interest to look at only those vertices coming from translates of a flag, since (after an appropriate rescaling) they will form groups for weakly homogeneous schemes, after an appropriate translation. We allow $G \subseteq \Gamma$ to be only a subset, since it will be useful to allow $G = \Gamma(r)$, the set of lattice points within radius $r$, which contain the set of vertices of $r$-cut regions.

\begin{proposition} \label{prop: vertex flag groups}
Suppose that $\mathcal{S}$ is homogeneous. Choose a translate of the window so that each $H \in \sH$ contains a point of $\Gamma_<$. Then $\cV(\Gamma,f)$ is a subgroup of $\Gamma_<$ for each flag $f \in \sF$, and we have
\[
\rk(\cV(\Gamma,f)) = \sum_{H \in f} (k - \rk(H)).
\]
The analogous result is also true for the subsystems of a decomposition of $\mathcal{S}$.
\end{proposition}

\begin{proof}
By the homogeneous condition, we have that $H + \Gamma_< = V(H) + \Gamma_<$ for any $H \in \sH$. Each $V(H)$ is a subspace, in particular a subgroup of $\intl$. Therefore, each $V(H) + \Gamma_<$ is a group, and so is
\[
\cV(\Gamma,f) \coloneqq \bigcap_{H \in f} (H + \Gamma_<) = \bigcap_{V \in V(f)} (V + \Gamma_<).
\]
Since $0 \in V$ for each $V \in V(f)$, it is clear that $\Gamma_< \leqslant \cV(\Gamma,f)$.

We have an isomorphism
\begin{equation} \label{eq: flag isomorphism}
\cV(\Gamma,f) = \bigcap_{H \in f} (V(H) + \Gamma_<) \cong \bigoplus_{H \in f} \frac{\Gamma}{\Gamma^H},
\end{equation}
given by sending the vertex of $\bigcap (V(H) + (\gamma_H))$ to $([\gamma_H])_H$. To see that this is an isomorphism, firstly we note it is well defined: two translates $V(H) + (\gamma_1)_<$ and $V(H) + (\gamma_2)_<$ intersect only when they are precisely equal, in which case $\gamma_1-\gamma_2 \in \Gamma^H$. The map is a homomorphism, since if $v \in V(H) + (\gamma_H)_<$ and $w \in V(H) + (\tau_H)_<$ for each $H \in f$, then $v+w \in (V(H) + (\gamma_H)_<) + (V(H) + (\tau_H)_<) = V(H) + (\gamma_H+\tau_H)_<$ for each $H \in f$. It is clearly surjective, since we have free choice for each $\gamma_H$. Finally, it is injective, since if each $\gamma_H \in \Gamma^H$ then each translate $V(H) + \gamma_H = V(H)$, by definition of $\Gamma^H$, so the intersection vertex is the origin.

It follows that
\[
\rk(\cV(\Gamma,f)) = \rk\left(\bigoplus_{H \in f} \frac{\Gamma}{\Gamma^H}\right) = \sum_{H \in f} (k-\rk(H)),
\]
as required. The proof for subsystems is identical, except for added indices.
\end{proof}

Suppose that $\mathcal{S}$ (or a subsystem) satisfies {\bf C}. By Theorem \ref{thm: C => hyperplane spanning}, the system is hyperplane spanning. By the general formula for the complexity exponent given by Theorem \ref{thm: generalised complexity} and the lower bound for this exponent by Corollary \ref{cor: minimal complexity}, we have that
\[
d = \sum_{H \in f} (k - \rk(H) - 1), \text{ that is, } \ \rk(\Gamma_<) = k = \sum_{H \in f} (k-\rk(H))
\]
for any flag $f$. Conversely, suppose that $k = \sum_{H \in f}(k-\rk(H))$ for any flag $f$. It follows that each $\alpha_f = \sum (d - \rk(H) + \beta_H) \leq \sum (k - \rk(H) - 1) = k-n = d$, so the scheme satisfies {\bf C}. Hence, we have established the following:

\begin{corollary} \label{cor: G finite index in vertices}
Suppose that $\mathcal{S}$ is homogeneous. Then $\mathcal{S}$ satisfies {\bf C} if and only if $\Gamma_<$ is a finite index subgroup of $\cV(\Gamma,f)$ for any flag $f \in \sF$, after taking a suitable translate of the window so that $0 \in \cV(\Gamma)$. The analogous result holds for subsystems.
\end{corollary}

Recall that the full vertex set $\cV(\Gamma)$ is the union of groups $\cV(\Gamma,f)$ over all possible flags (every vertex is an intersection of hyperplanes from some flag). Since there are only finitely many flags, we have the following:

\begin{corollary} \label{cor: finite index vertices}
Suppose that $\mathcal{S}$ is homogeneous. Then $\mathcal{S}$ satisfies {\bf C} if and only if there exists some $M \in \N$ so that $\cV(\Gamma) \subseteq \frac{1}{M} \Gamma_<$ (after taking a suitable translate of the window). The analogous result holds for subsystems.
\end{corollary}

The above results apply equally well to weakly homogeneous schemes. Indeed, recall that a scheme is weakly homogeneous if it is homogeneous after replacing $\Gamma$ with some $\frac{1}{N}\Gamma$ (see Lemma \ref{lem: hom vs weak hom}); this replacement does not affect property {\bf C}, by the formula in Theorem \ref{thm: generalised complexity}. So the above results also apply to weakly homogeneous schemes, by replacing each term $\Gamma$ by some $\frac{1}{N}\Gamma$. Then the above corollary has the following important implication for the sizes of cut regions: if $\Gamma_<$ is Diophantine, and the scheme is weakly homogeneous and satisfies {\bf C}, then vertices of cut regions stay distant. Indeed, the vertices of cut regions are precisely the subset $W \cap \cV(\Gamma)$, and we have
\[
\cV(\Gamma) \subseteq \cV(\frac{1}{N}\Gamma) \subseteq \frac{1}{N}\frac{1}{M}\Gamma_<,
\]
for some $M \in \N$. The latter is just a rescaling of $\Gamma_<$, which is still Diophantine.

To effectively apply the above, we need a quantitative version of the previous corollary:

\begin{proposition} \label{prop: vertex inclusions}
Suppose that $\mathcal{S}$ is homogeneous, with window translated so that $o$ may be taken to be the origin in Definition \ref{def: homogeneous}. Then there is some $c > 0$ so that
\[
\Gamma_<(r) \subseteq \cV(\Gamma(r + c)).
\]
Conversely, suppose that $\mathcal{S}$ also satisfies {\bf C}. Then there is some $M \in \N$ and $C > 0$ so that, for sufficiently large $r$,
\[
M \cdot \cV(\Gamma(r)) \subseteq \Gamma_<(Cr).
\]
The analogous result holds for subsystems.
\end{proposition}

\begin{proof}
After a suitable translation we have that $0 \in H + \Gamma_<$ for each $H \in \sH$. Hence, for each $H \in \sH$, there exists some $\gamma_H \in \Gamma$ with $0 \in H + (\gamma_H)_<$. Take any $x \in \Gamma(r)$. Then $x_< \in V(H) + x_< = H + (\gamma_H+x)_<$ for any $H \in \sH$. We see that $x_< \in \cV(\Gamma(r+c))$, where $c = \max_{H \in \sH} \|\gamma_H\|$.

So now suppose that $\mathcal{S}$ also satisfies {\bf C}. By Corollary \ref{cor: finite index vertices}, we have an inclusion $M \cdot \cV(\Gamma) \leqslant \Gamma_<$.

Take any flag $f \subseteq \sH$ and recall that we have a canonical identification $\cV(\Gamma,f) \cong \bigoplus_{H \in f} \frac{\Gamma}{\Gamma^H}$, see Equation \ref{eq: flag isomorphism}; we sometimes speak of elements in either group interchangeably. Choose a new norm $\eta$ on $\tot$ by letting
\[
\eta(e) \coloneqq \max_{H \in f} \|e-\langle \Gamma^H \rangle_\R \|,
\]
where by $\|e - \langle \Gamma^H \rangle_\R \|$, we mean the distance from $e$ to the subspace $\langle \Gamma^H \rangle_\R$. This is a norm on $\tot$. Indeed, invariance under scalar multiplication and the triangle inequality follow from the same for each $\|e - \langle \Gamma^H \rangle_\R\|$. Clearly $\eta(0) = 0$. Finally, suppose that $\eta(e) = 0$. We show next that this implies that $e=0$. 

First, notice that for each $H \in f$ we have $\Gamma^H_< \subset V(H)$ so that by linearity,
\[
\langle \Gamma^H \rangle_\R\subset \langle (V(H)+\intl) \cap \Gamma \rangle_\R. 
\]
But $\displaystyle \bigcap_{H\in f}V(H)=\{0\}$ since $f$ is a flag, so that 
\[
\bigcap_{H\in f}\langle \Gamma^H \rangle_\R\subset \bigcap _{H\in f}\langle (V(H)+\intl) \cap \Gamma \rangle_\R=\{0\}, 
\]
which also uses the fact that by injectivity of $\pi_<$ on $\Gamma$, $\intl$ only contains the one lattice point, $0$. Hence, the subspaces $\langle \Gamma^H \rangle_\R$ have trivial intersection. So if $\eta(e) = 0$ then $e \in \langle \Gamma^H \rangle_\R$ for each $H \in f$, so $e = 0$. This proves that $\eta$ is a norm. Since norms are linearly equivalent, there is some constant $A > 0$ so that $\|e\| \leq A \eta(e)$ for all $e \in \tot$.

Put a norm on the group $\bigoplus_{H \in f} \frac{\Gamma}{\Gamma^H} \cong \cV(\Gamma,f)$ by setting
\[
\eta' ([\gamma_H])_H \coloneqq \max_{H \in f} \|\gamma_H - \langle \Gamma^H \rangle_\R\|,
\]
where we may take any representative $\gamma_H$ for a class $[\gamma_H] \in \frac{\Gamma}{\Gamma^H}$, since translating by an element of $\Gamma^H$ does not change the distance to $\langle \Gamma^H \rangle_\R$. Notice that if $\gamma \in \Gamma$ then $\gamma_< \in \cV(\Gamma,f)$ is represented by the element $(\gamma,\gamma_,\ldots,\gamma)$ and so we have
\[
\eta'(\gamma_<) = \max_{H \in f} \|\gamma - \langle \Gamma^H \rangle_\R\| = \eta(\gamma).
\]

Let $v \in \cV(\Gamma,f,r)$, so that
\[
\{v\} = \bigcap_{H \in f} \left(V(H) + (\gamma_H)_<\right)
\]
for $\gamma_H \in \Gamma(r+c)$ (where $c$, as above, is taken so that we may use translates of the $V(H) \in V(f)$ rather than $H \in f$). Since we may choose each $\|\gamma_H\| \leq r+c$ it follows that each $\|\gamma_H - \langle \Gamma^H \rangle_\R\|  \leq r+c$. Since $Mv = \gamma_<$ for some $\gamma \in \Gamma$, we see that
\[
\|\gamma\| \leq A \cdot \eta(\gamma) = A \cdot \eta'(\gamma_<) = A \cdot \eta'(Mv) \leq AM(r+c),
\]
so $Mv = \gamma \in \Gamma_<(AM(r+c))$. Hence, for $C > AM$, we have that $\cV(\Gamma,f,r) \subseteq \Gamma_<(Cr)$ for sufficiently large $r$. Repeating for each flag $f$, we obtain the required bound. The proof for subsystems is identical.
\end{proof}

The above applies to weakly homogeneous schemes since, again, we can make a scheme homogeneous by replacing $\Gamma$ by some $\frac{1}{N} \Gamma$. Since we ultimately only care about displacements between vertices, this is how we present the result below (which thus does not depend on which translate of the window is chosen):

\begin{corollary} \label{cor: vertex displacements}
Suppose that $\mathcal{S}$ is weakly homogeneous. Then there is some $N \in \N$ and $c > 0$ so that
\[
\Gamma_<(r) \subseteq \cV\left(\frac{1}{N}\Gamma,r+c\right) - \cV\left(\frac{1}{N}\Gamma,r+c\right).
\]
Conversely, suppose that $\mathcal{S}$ also satisfies {\bf C}. Then there is some $M \in \N$ and $C > 0$ so that, for sufficiently large $r$,
\[
M \cdot \left(\cV(\Gamma,r) - \cV(\Gamma,r) \right) \subseteq \left(\frac{1}{N}\Gamma_<\right)(Cr).
\]
The analogous result holds for subsystems.
\end{corollary}

\begin{proof}
By replacing $\Gamma$ with $\frac{1}{N}\Gamma$, we may obtain a strictly homogeneous system. Applying Proposition \ref{prop: vertex inclusions}, we have that
\[
\left(\frac{1}{N} \Gamma_<\right)(r) - \left(\frac{1}{N}\Gamma_<\right)(r) \subseteq \cV\left(\frac{1}{N}\Gamma,r\right) - \cV\left(\frac{1}{N}\Gamma,r\right).
\]
Since $\Gamma \subseteq \frac{1}{N}\Gamma$, so that $\Gamma_<(r) \subseteq (\frac{1}{N}\Gamma_<)(r)$, the first inclusion follows.

For the second, we proceed as above to obtain
\[
M \cdot \cV\left(\frac{1}{N}\Gamma,r\right) - M \cdot \cV\left(\frac{1}{N}\Gamma,r\right) \subseteq \left(\frac{1}{N}\Gamma_<\right)(Cr) - \left(\frac{1}{N}\Gamma_<\right)(Cr) \subseteq \left(\frac{1}{N} \Gamma_<\right)(2Cr)
\]
for sufficiently large $r$. We have that $\cV(\Gamma,r) \subseteq \cV(\frac{1}{N}\Gamma,r)$, so the result follows after replacing $C$ with $2C$.
\end{proof}

\section{Proof that {\bf C} and {\bf D} imply {\bf LR}} \label{sec: C and D => LR}

The previous result shows that a weakly homogeneous scheme which satisfies {\bf C} has cut regions whose vertices belong to some (rescaling of) the lattice $\Gamma_<$. Hence, if $\Gamma_<$ is Diophantine, then these vertices never get too close to each other. There are restrictions on the geometries of these cut regions: their faces are aligned with translates of the hyperplanes. Hence, these convex regions contain reasonably large balls when their vertices are distant:

\begin{lemma} \label{lem: geometric lemma} Let $P$ be an $n$-dimensional polytope for which $d(v_1,v_2) \geq \epsilon$ for any distinct pair of vertices $v_1$, $v_2$ of $P$. Then $P$ contains a ball of radius $C \epsilon$, with the constant $C$ depending only on the collection $\sH(P)$ of supporting hyperplanes defining $P$. 
\end{lemma}

\begin{proof} For this proof call a collection $\{e_i\}_{i=1}^n$ of edges an \emph{edge span at $v$}, for a vertex $v$, if each edge $e_i$ contains $v$ and additionally the directions that the $e_i$ point in are linearly independent (that is, the collection $\{p_1-v,p_2-v,\ldots,p_n-v\}$ are linearly independent, where $p_i \neq v$ is a point of $e_i$). Geometrically, it is intuitively clear that every vertex has an edge span, although we provide here a proof. In dimension $1$ it is trivial, since the two vertices are incident with the edge $P$ itself, whose affine span is one-dimensional. Supposing it is true for $(n-1)$-polytopes, take any vertex $v$ of a given $n$-polytope. Clearly there are at least two facets containing $v$, since $v$ is the intersection point of $n$ distinct hyperplanes of $\sH(P)$, each intersecting $P$ in a facet. Taking two of these facets $F_1 = P \cap H_1$ and $F_2 = P \cap H_2$, for $H_1$, $H_2 \in \sH(P)$, by the induction hypothesis we can find edge spans $\{e_1,\ldots,e_{n-1}\}$ and $\{e_1',\ldots,e_{n-1}'\}$ at $v$ with respect to $F_1$ and $F_2$. Some $e_\ell'$ points out of the affine hyperplane spanned by the $e_i$, since otherwise $H_1 = H_2$; indeed, each $H_i$ is determined as the affine span of the $e_i$ with respect to $v$. So $\{e_1,\ldots,e_{n-1},e_\ell'\}$ gives an edge span at $v$.

So now let $P$ be given, and take any edge span $\{e_1,\ldots,e_n\}$ at a vertex $v$. Each $e_i$ has two endpoints as vertices, $v$ and some $v_i$, and since the $e_i$ are distinct so are the vertices $v$, $v_1$, \ldots, $v_n$. By assumption we have that these vertices are at least distance $\epsilon$ from each other, and by convexity we have that the simplex given by their convex hull is contained in $P$.

So it is sufficient to show that the convex hull $S = \conv\{0,v_1-v,v_2-v,\ldots,v_n-v\}$ contains a ball of radius $C \epsilon$, with $C$ depending only on the $V(H)$. Consider the linear map $M$ which takes the $i$th standard basis vector to $(v_i - v)/\|v_i-v\|$. Let $D$ be the linear map represented by the diagonal matrix with $i$th entry $\|v_i-v\|$. Then $S = (D \circ M)(\Delta)$, where $\Delta$ is the simplex spanned by $0$ and the standard basis vectors. It contains a ball of radius $R$ depending only on the dimension $n$. Since $D$ is a diagonal matrix, $d(D(x),D(y)) \geq \alpha \min \|v_i-v\|\cdot d(x,y) \geq \alpha \epsilon \cdot d(x,y)$ for some constant $\alpha > 0$ depending only on the choice of norm. Similarly, $M$ is a non-singular (hence bi-Lipschitz) linear map which is determined only by the directions of the $e_i$. So there is some bound $C'$ for which $d(M(x),M(y)) \geq C' d(x,y)$, where a suitable bound $C'$ can be determined just by the collection of hyperplanes $\sH(P)$ (which determines the finite set of possible edges). Hence, $S$ contains a ball of radius $(C' \alpha R) \cdot \epsilon$, as required. \end{proof}

We apply the above to find large balls in cut regions for Diophantine schemes:

\begin{lemma} \label{lem: large cut regions}
Suppose that $\mathcal{S}$ is weakly homogeneous, satisfies {\bf C} and {\bf D} and has constant hyperplane stabiliser rank. Then there is some $c > 0$ so that, for sufficiently large $r$, every cut region $C \in \sC(r)$ contains a ball of radius $cr^{-\delta}$, where $\delta = \frac{d}{n}$. An analogous result holds for the subsystems $(X_i,\frac{1}{N}\Gamma_i,W_i)$.
\end{lemma}

\begin{proof}
Since $\mathcal{S}$ has constant hyperplane stabiliser rank, we do not need to take a decomposition of $\mathcal{S}$ in applying {\bf D}; see Definition \ref{def: diophantine scheme} and the proceeding comments. Hence, $\Gamma_<$ is Diophantine in $\intl$, so there exists some $\nu > 0$ so that
\[
\|\gamma_<\| \geq \nu \cdot \|\gamma\|^{-\delta}
\]
for all $\gamma \in \Gamma$, where $\delta = \frac{d}{n}$.

By definition, every vertex of an $r$-cut region is an element of $\cV(\Gamma,r)$. By Corollary \ref{cor: vertex displacements}, given vertices $v_1$, $v_2$ of an $r$-cut region, we have that
\[
v_1 - v_2 = \frac{1}{M}\frac{\gamma_<}{N} \text{ for some } \gamma \in \Gamma \cap B_{NCr}(0)
\]
for sufficiently large $r$, where $M$, $N$ and $C > 0$ are fixed constants depending only on $\mathcal{S}$. Hence, if $v_1 \neq v_2$, we have that
\begin{equation} \label{eq: distant vertices}
\|v_1 - v_2\| = \frac{1}{MN}\|\gamma_<\| \geq \frac{\nu}{MN}(NCr)^{-\delta} = Ar^{-\delta},
\end{equation}
for a fixed constant $A > 0$. By the above geometric Lemma \ref{lem: geometric lemma}, for sufficiently large $r$ every $r$-cut region contains a ball of radius $c r^{-\delta}$ for some constant $c > 0$.

An identical proof for the subsystem $(X_i,\frac{1}{N}\Gamma_i,W_i)$ holds; we remark that $\frac{1}{N}\Gamma_i$ is still Diophantine, since it is just a rescaling of $\Gamma_i$.
\end{proof}

The above shows that cut regions (in the subsystems) are reasonably large when {\bf C} and {\bf D} hold. This implies that patches appear with high frequency. They also appear relatively uniformly, without large gaps, since orbits are dense by transference. This establishes the second direction of our main theorem:

\begin{theorem}
If $\mathcal{S}$ is weakly homogeneous and satisfies both {\bf C} and {\bf D}, then it is {\bf LR}.
\end{theorem}

\begin{proof}
By the previous Lemma \ref{lem: large cut regions}, for each subsystem $(X_i,\frac{1}{N}\Gamma_i,W_i)$, the $r$-cut regions contain balls of radius $c r^{-\delta_i}$ for sufficiently large $r$ and some fixed $c > 0$. By definition, each $(\Gamma_i)_<$ is Diophantine. By transference (Theorem \ref{thm: transference}), each $\Gamma_i$ is densely distributed. Hence, by Lemma \ref{lem: transference general regions}, there exists some $\tau > 0$ so that, for any for $w \in W_i-W_i$, there is some $\gamma_i \in \Gamma_i(r)$ with
\begin{equation} \label{eq: close projection}
\|w-(\gamma_i)_<\| \leq \tau \cdot r^{-\delta_i};
\end{equation}
by taking $\tau$ large enough, we may assume this holds for each $i$ and sufficiently large $r > 0$.

Take any $r$-patch $P$, with $r$ sufficiently large to apply the above bounds. It has an associated $r$-acceptance domain $A_P \in \sA(r)$. By Lemma \ref{lem: cuts refine acc}, for sufficiently large $r$ we have that $C \subseteq A_P$ for some cut region $C \in \sC(r)$. By Lemma \ref{lem: cut region products}, there are cut regions $C_i \in \sC_i(\lambda r)$ for the subsystems $(X_i,\frac{1}{N}\Gamma_i,W_i)$ so that $C_1 + \cdots + C_m \subseteq C \subseteq A_P$.

Take any $y \in \cps$, and write $y = \gamma_\vee$ for some $\gamma \in \Gamma$. By the above, each $C_i$ contains some ball $B_i$ of radius $c(\lambda r)^{-\delta_i} = (c \lambda^{-\delta_i}) r^{-\delta_i}$, with centre $b_i \in B_i$.

We may write $\gamma_< = x_1 + \cdots + x_m$, where each $x_i \in W_i$ (since $W = W_1 + \cdots + W_m$ and $y^* = \gamma_< \in W$). Then we have $b_i-x_i \in B_i-x_i \subseteq W_i-W_i$, so by Equation \ref{eq: close projection} there exist $\gamma_i \in \Gamma_i(\kappa r)$ with
\[
\|(b_i-x_i) - (\gamma_i)_<\| \leq \tau \cdot (\kappa r)^{-\delta_i} \leq (c \lambda^{-\delta_i}) r^{-\delta_i},
\]
for some suitably large $\kappa > 0$ (with respect to $c$, $\tau$, $\lambda$ and the $\delta_i$, and hence only depending on $\mathcal{S}$). It follows that each $(\gamma_i)_< \in B_i-x_i$, so that $(\gamma_i)_<+x_i \in B_i \subseteq C_i$.

Now, consider the point $g \coloneqq (\gamma_1 + \cdots + \gamma_m) + \gamma \in \Gamma$ and let $z = g_\vee$. Each $\|\gamma_i\| \leq \kappa r$, so that
\[
\|z-y\| = \|g_\vee - \gamma_\vee\| \leq = \|(g-\gamma)_\vee\| = \|(\gamma_1 + \cdots + \gamma_m)_\vee\| \leq \alpha m (\kappa r),
\]
where $\alpha$ depends only on the projection $\pi_\vee$. On the other hand, the $i$th component of $g_<$ with respect to $(X_i)_{i=1}^m$, which is $(\gamma_i)_< + x_i$, belongs to $C_i$ by the above. Hence, $z^* = g_< \in C_1 + \cdots + C_m \subseteq C \subseteq A_P$, so the $r$-patch at $z$ is $P$. Since $P$ was arbitrary, we see that every $r$-patch occurs within radius $(\alpha m \kappa) r$ of any point of $\cps$ for sufficiently large $r$, so $\mathcal{S}$ is {\bf LR}.
\end{proof}

\section{Further results and examples} \label{sec: further results and examples}

\subsection{Implied equivalences}
Recall that, by Corollary \ref{cor: PW => C+D}, {\bf PW} alone implies {\bf C} and {\bf D}, even without assuming the weakly homogeneous condition. Since {\bf PW} implies {\bf LR} (Lemma \ref{lem: LR => PW}), we have that {\bf PW} is also equivalent to {\bf C} and {\bf D} for weakly homogeneous schemes.

It is shown in \cite{BBL} that {\bf LR} is equivalent to {\bf PW} and another condition {\bf U}, but it is not known if the {\bf U} may be dropped in this characterisation. The above shows that it can for the class of cut and project sets considered here. It is also shown in \cite{BBL} that an FLC Delone set satisfies a \emph{subadditive ergodic theorem} ({\bf SET}) if and only it satisfies \emph{positivity of quasiweights} ({\bf PQ}). We refer the reader to \cite{BBL} for the notations of {\bf PQ} and {\bf SET}.

\begin{corollary}
The following are equivalent for a weakly homogeneous scheme:
\begin{enumerate}
	\item {\bf LR};
	\item {\bf PW};
	\item {\bf PQ};
	\item {\bf SET};
	\item {\bf C} and {\bf D}.
\end{enumerate}
\end{corollary}

We note a useful fact about changing the lattice in a cut and project scheme. It follows from the general theorem on complexity functions, Theorem \ref{thm: generalised complexity}, that {\bf C} is not affected by changing the lattice $\Gamma$ to a finite index super- or sub-lattice. The same is true for property {\bf D}, as follows from Lemma \ref{lem: Diophantine finite index}. So we have shown the following:

\begin{corollary}
Let $\mathcal{S}$ and $\mathcal{S}'$ be polytopal schemes which have the same data, except that $\mathcal{S}$ has lattice $\Gamma$, and $\mathcal{S}'$ has lattice $\Gamma' \leqslant \Gamma$. Then $\mathcal{S}$ is {\bf LR} if and only if $\mathcal{S}'$ is {\bf LR}.
\end{corollary}

\subsection{Example applications of Theorem \ref{thm: main}}

We conclude by demonstrating our main theorem on some individual examples, as well as certain classes of cut and project schemes. In particular, we shall see that Theorem \ref{thm: main} generalises the main result of \cite{HaynKoivWalt2015a}, and give a simplified characterisation of {\bf LR} for canonical cut and project sets.

\subsubsection{Codimension 1 Schemes} \label{sec: codimension 1}

If $n=1$ then the window $W$ is indecomposable, given by a closed interval. Let us translate $W$ so that one end-point lies over the origin, and denote the second end-point by $\beta \in \intl$. It is easy to see that the scheme $\mathcal{S}$ is homogeneous if and only if $\beta \in \Gamma_<$, and is weakly homogeneous if $N \beta \in \Gamma_<$ for some $N \in \N$. For example, if $\mathcal{S}$ is canonical (so that $W$ is the projection of a fundamental domain of $\Gamma \leqslant \tot$), then $\mathcal{S}$ is homogeneous.

The collection of supporting hyperplanes of $W$ is $\sH = \{\{0\},\{\beta\}\}$. In particular, $\sH_0 = \{\{0\}\}$, so that $\Gamma^H$ is trivial for each of the two supporting hyperplanes. So $\rk(H) = \beta_H = 0$ in Theorem \ref{thm: generalised complexity} and we see that the cut and project sets generated by $\mathcal{S}$ have complexity $r^\alpha$ where $\alpha = d$.

Hence, $\mathcal{S}$ satisfies {\bf C}, so is {\bf LR} if and only if {\bf D} is satisfied, that is, if and only if $\Gamma_<$ is Diophantine. By choosing a basis for $\Gamma$, we may write
\[
\Gamma_< = \langle x_1, x_2, \ldots, x_d, y \rangle_\Z,
\]
where we implicitly identify $\intl \cong \R$. Write $x \coloneqq (x_1,\ldots,x_d)^T$. Then $\Gamma_<$ is Diophantine if and only if there is a constant $c > 0$ such that
\[
|(n \cdot x) + my | \geq \frac{c}{\max\{|n_1|,\ldots,|n_d|,|m|\}^d},
\]
where $n \in \Z^d$ is non-zero, with coordinates $n = (n_1,\ldots,n_d)^T$, and $n \cdot x$ is the dot product $\sum_{i=1}^d n_i x_i$. Equivalently, we can rescale the internal space by a factor of $1/y$, so that we now use the vector $z \coloneqq  (x_1/y,x_2/y,\ldots,x_d/y)^T$. The projected lattice is Diophantine if and only if there is some constant $c > 0$ such that
\[
d(n \cdot z,\Z) \geq \frac{c}{\|n\|},
\]
for all non-zero $n \in \Z^d$, where $d(-,\Z)$ denotes the distance to the nearest integer and $\|n\|$ is any lattice norm on $\Z^d$, such as $\|n\| = \max_i |n_i|$. That is, the linear form $L \colon \Z^d \to \R$, given by $L(n) = n \cdot z$ is \emph{badly approximable}. We summarise:

\begin{theorem}
Suppose that $\mathcal{S}$ is a codimension $1$ aperiodic polytopal cut and project set. Identify $\intl \cong \R$. Without loss of generality, we may write $\Gamma_< = G + \Z$, where $G = \langle x_1,\dots,x_d\rangle_\Z$, for irrational $x_i \in \R$. Let the window $W$ be an interval of length $\beta$. Then $\mathcal{S}$ is weakly homogeneous if and only if $N \cdot \beta \in G + \Z$ for some $N \in \N$, that is,
\[
\beta = \frac{1}{N}\left(n_1 x_1 + n_2 x_2 + \cdots n_d x_d + m\right),
\]
for $n_i$, $m \in \Z$. In this case, $\mathcal{S}$ produces linearly repetitive cut and project sets if and only if $(x_1,\ldots,x_d)^T$ is badly approximable.
\end{theorem}

Constructing examples of such schemes is simple. Firstly, one may start with a badly approximable system $(x_1,\ldots,x_d)^T$, with $\{x_1,\ldots,x_d\}$ a set of $\Q$-linearly independent numbers. One may construct a scheme whose projected lattice is given by $\langle x_1,\ldots,x_d,1\rangle_\Z$. For example, we may take $\tot = \R^{d+1}$, $\phy = \R^d \times \{0\}$, $\intl = \{0\}^d \times \R$ and
\[
\Gamma \coloneqq \langle (e_1,x_1), (e_2,x_2),\ldots, (e_d,x_d), (0,1)\rangle_\Z,
\]
where the $e_i$ are the standard basis vectors of $\R^d$. Finally, we choose the window $W \subseteq \intl$ with length $\beta \in \Q[1,x_1,x_2,\ldots,x_d]$. Then $\mathcal{S}$ produces {\bf LR} cut and project sets.

One may ask which {\bf LR} cut and project sets may be produced if the weakly homogeneous property is dropped, that is, if the $\beta \notin \Q \Gamma_<$. In the codimension $1$ setting, the determination of acceptance domains is simple. In particular, the acceptance domains have endpoints in the set $\Gamma_< \cup (\Gamma_< + \beta)$, and it is not too hard to show that the cut and project set will be {\bf LR} if and only if both $\Gamma_<$ is badly approximable and also the following \emph{inhomogeneous} Diophantine condition holds:
\begin{equation}\label{eq:codim1LR}
\|\gamma_< + \beta\| \geq \frac{c}{\|\gamma\|^d} \text{ for all non-zero } \gamma \in \Gamma.
\end{equation}
That is, the points $\gamma_<$ stay distant from both the origin and $\beta$, relative to $\|\gamma\|$.

It turns out that weak homogeneity is not a necessary condition for a cut and project scheme to be linearly repetitive, since the condition \eqref{eq:codim1LR} can be satisfied even when the window is not weakly homogeneous. On the other hand, Theorem \ref{thm: main} does not hold without the weak homogeneity condition, since it is possible for a cut and project scheme to satisfy both {\bf C} and {\bf D} without being {\bf LR}. 

To demonstrate this, we present here the details in the simple case when $k=2$, $d=1$, and $\Gamma=\mathbb Z^2$. Letting $\phy$ be a subspace of slope $\alpha$ and $\intl$ the $y$-axis with the window $[0, \beta]$, we see that in order for the scheme to be weakly homogeneous, we need a number $N\in \N$ such that $N\beta$ is of the form $n\alpha \mod 1$ for some $n\in \Z$. 

As explained above, for linear repetitivity, we need $\alpha$ to be badly approximable, and the additional condition \eqref{eq:codim1LR} which in this case has the form
\[
\|\alpha n + \beta\|\ge \tfrac c n, 
\]
where $c$ is a constant that depends on $\beta$. It is known (see \cite{Kleinbock}) that for any (irrational) $\alpha$ the set of choices for $\beta$ such that this inequality is satisfied is of full Hausdorff dimension. In particular, taking $\alpha$ to be badly approximable, there are uncountably many choices of $\beta$ which are not weakly homogeneous even though they are {\bf LR}. This shows that our results are sharp. 

On the other hand, by \cite{Kim}, for any irrational $\alpha$ the set of choices of $\beta$ such that this inequality fails is of full measure. In particular, even when $\alpha$ is badly approximable, almost every length of interval window gives a scheme which is not {\bf LR}, even though both {\bf C} and {\bf D} are satisfied.

\subsubsection{Ammann--Beenker and golden octagonal tilings} \label{sec: octagonals} \label{exp: AB}

We consider two $4$-to-$2$ canonical cut and project tilings, in each case $W = ([0,1]^4)_<$. Firstly, we take the well-known Ammann--Beenker tilings. They may be generated by a primitive substitution rule, so it is already known that they are {\bf LR}. They are MLD to their vertex sets, generated by a canonical cut and project scheme (which automatically means that they are homogeneous). With respect to one parametrisation (and after rescaling to remove a factor of $1/2$), the basis vectors of the lattice $\Gamma \leqslant \R^4$ are sent to the vectors
\[
f_1 =
\begin{pmatrix}
1 \\
0
\end{pmatrix},  \
f_2 =
\begin{pmatrix}
-\sqrt{2} \\
\sqrt{2}
\end{pmatrix}, \
f_3 =
\begin{pmatrix}
0 \\
-1
\end{pmatrix}, \
f_4 =
\begin{pmatrix}
\sqrt{2} \\
\sqrt{2}
\end{pmatrix}
\]
in the internal space $\intl \cong \R^2$. These vectors are indicated inside the window, on the left of Figure \ref{fig: octagonal}. The window has supporting hyperplanes $\sH_0 = \{H_i\}_{i=1}^4$, where $H_i$ is the subspace containing $f_i$. We have that $\Gamma_< = \langle f_1, f_2, f_3, f_4 \rangle_\Z$ and
\[
\Gamma^{H_1} = \langle e_1, e_2-e_4 \rangle_\Z, \ \Gamma^{H_2} = \langle e_2, e_1+e_3 \rangle_\Z, \ \Gamma^{H_3} = \langle e_3,e_2+e_4 \rangle_\Z, \ \Gamma^{H_4} = \langle e_4,e_1-e_3 \rangle_\Z.
\]
Hence, each $\rk(H_i) = 2$, and the scheme is hyperplane spanning, so the Ammann--Beenker tilings have complexity $p(r) \asymp r^2$ by Theorem \ref{thm: generalised complexity}. That is, the scheme satisfies {\bf C}.

We now consider the Diophantine condition. For any $\gamma = (n_1,n_2,n_3,n_4) \in \Z^4$, we have that
\[
\gamma_< = \sum_{i=1}^4 n_i f_i =
\begin{pmatrix}
(n_4-n_2) \sqrt{2} + n_1 \\
(n_2 +n_4)\sqrt{2} - n_3
\end{pmatrix}.
\]
Since $\sqrt{2}$ is a quadratic irrational, it is badly approximable so, for some $c \leq 1$,
\[
|m \sqrt{2} - n| \geq \frac{c}{n}
\]
for all non-zero $n \in \Z$. It follows that $\|\gamma_<\|$ has either first or second component with norm at least $c/n_1$ or $c/n_3$, if $n_1$ or $n_3 \neq 0$. If $n_1=n_3=0$, and one of $n_2$, $n_4 \neq 0$, then the corresponding component has norm at least $\sqrt{2}$. Hence, for non-zero $\gamma \in \Gamma$,
\[
\|\gamma_<\| \geq \frac{c}{\max\{n_i\}_{i=1}^4} = \frac{c}{\|\gamma\|},
\]
where we choose, without loss of generality, the max-norm on $\R^4$. So the projected lattice $\Gamma_<$ is Diophantine (Definition \ref{def: diophantine lattice}). As calculated above, the scheme has constant stabiliser rank (which also follows from {\bf C} and Corollary \ref{cor: equal ranks} by indecomposability), so the scheme $\mathcal{S}$ satisfies {\bf D} (Definition \ref{def: diophantine scheme}). Hence, by Theorem \ref{thm: main}, $\mathcal{S}$ satisfies {\bf LR}.

We shall more briefly review a similar example with octagonal window. The \emph{golden octagonal tilings} \cite{BedFer} are generated by a canonical cut and project scheme with
\[
\phy =
\langle (-1,0,\varphi,\varphi)^T, (0,1,\varphi,1)^T \rangle_\R, \Gamma = \Z^4,
\]
where $\varphi = (1 + \sqrt{5})/2$ is the golden ratio. Let $e_i$ denote the standard basis vectors of $\R^4$. We choose $\intl = \langle e_3, e_4 \rangle_\R$ (although only for convenience, and of course one should choose internal space making $\pi_\vee$ injective on $\Z^4$, but this does not affect calculations). The standard basis vectors are thus mapped to the vectors
\[
f_1 =
\begin{pmatrix}
\varphi \\
\varphi
\end{pmatrix}, \
f_2 =
\begin{pmatrix}
-\varphi \\
-1
\end{pmatrix}, \
f_3 =
\begin{pmatrix}
1 \\
0
\end{pmatrix}, \
f_4 =
\begin{pmatrix}
0 \\
1
\end{pmatrix}
\]
whose $\Z$-span generates $\Gamma_<$. As for the Ammann--Beenker window, we take $\sH_0 = \{H_i\}_{i=1}^4$ with $H_i = \langle f_i \rangle_\R$. The window and these vectors are illustrated in Figure \ref{fig: octagonal}, where we apply the linear map $(x,y)^T \mapsto (x-y,y/\varphi)^T$ to make the shape more regular (one could equally have used a different choice of internal space).

We find that
\[
\Gamma^{H_1} = \langle e_1,e_3+e_4 \rangle_\Z, \ \Gamma^{H_2} = \langle e_2,e_1+e_3 \rangle_\Z, \ \Gamma^{H_2} = \langle e_3,e_2+e_4 \rangle_\Z, \ \Gamma^{H_4} = \langle e_4, e_1+e_2 \rangle_\Z,
\]
using that $\varphi^2 = \varphi + 1$. Again, each $\rk(H_i) = 2$, so that {\bf C} is satisfied by Theorem \ref{thm: generalised complexity}. To establish {\bf D}, consider $\gamma = (n_1,n_2,n_3,n_4) \in \Z^4$. Then
\[
\gamma_< =
\begin{pmatrix}
(n_1 - n_2)\varphi + n_3 \\
n_1\varphi + (n_4 - n_2)
\end{pmatrix}.
\]
The golden ratio $\varphi$ is badly approximable. By similar calculations as for the Ammann--Beenker, it easily follows that $\Gamma_<$ is Diophantine. Since the scheme is indecomposable, again {\bf D} holds, so by Theorem \ref{thm: main} the golden octagonal tilings satisfy {\bf LR}.

\subsubsection{$4$-to-$2$ with regular polygon windows}

Let $n \in \N$ and $H_i = \langle \exp(2 \pi \ell / n) \rangle_\R$ be the line in $\C$ pointing along the $i$th root of unity. Let $V(n)$ be the set of these lines. Consider a $4$-to-$2$ weakly homogeneous cut and project scheme with $\sH_0 = V(n)$. One may take $W$ to be a regular polygon, for example. The property of {\bf LR} depends only on $n$ and $\Gamma_<$.

For property {\bf C}, by Theorem \ref{thm: generalised complexity} it is necessary and sufficient that each $\rk(H) = 2$. This includes the decomposable $n=4$ case, since by Lemma \ref{lem: lattice spanned by lines} we have that $\Gamma^{H_i} + \Gamma^{H_j}$ is finite rank in $\Gamma$ for any two distinct $H_i$, $H_j \in \sH_0$ (and each $\Gamma^{H_i}$ must be at least rank $2$ to ensure that $\Gamma_<$ is dense in $\intl$). Hence the scheme has constant stabiliser rank, so satisfies {\bf D} if and only if both of $\Gamma^{H_i}$ and $\Gamma^{H_j}$ are Diophantine lattices in their corresponding subspaces. So we have shown the following:

\begin{corollary}
A weakly homogeneous $4$-to-$2$ scheme with $\sH_0 = V(n)$ satisfies {\bf LR} if and only if $\rk(H) = 2$ for each $H \in \sH$, and for any pair of generators (up to finite index) of $\Gamma^H$ are of the form $x$, $\alpha x$ for $\alpha \in \R$ badly approximable. The latter Diophantine condition need only be checked for two distinct $H \in \sH$.
\end{corollary}

For example, if $n=3$ (equivalently $n=6$) then the window may be taken as a triangle or hexagon. To obtain an {\bf LR} cut and project scheme, we may take
\[
\Gamma = T + \alpha T,
\]
where $\alpha$ is badly approximable and $T$ is the triangular lattice $T \coloneqq \langle 1, (-1/2 + \sqrt{3}i/2) \rangle_\Z$.

For $n=4$ the window is a square (or rectangle) and we have {\bf LR} if and only if $\Gamma = \Gamma_\alpha + \Gamma_\beta$, where $\Gamma_\alpha = A \cdot \langle 1, \alpha \rangle_\Z$, $\Gamma_\beta = B \cdot \langle i, \beta i \rangle_\Z$, where both $\alpha$ and $\beta$ are badly approximable and $c_1$, $c_2>0$ are arbitrary scalars.

When $n=5$ (equivalently $n=10$) we may take $W$ as a regular pentagon or decagon. Take $\Gamma_< \cong \Z^4$ to be the Diophantine lattice spanned by the 5th roots of unity. There is a homogeneous choice of $W$, with the correct projection to the physical space, giving the Penrose cut and project patterns, up to MLD equivalence. The window for the Penrose tilings is not the decagon, but the $4$-to-$2$ version (using the $A_4$ lattice in the total space) has window which is a union of $4$ pentagons, see \cite[Example7.11]{AOI}, \cite{BKSZ90} and \cite[Figure 7]{NamWer16}. Translates under $\Gamma_<$ of the Penrose window, and of the regular decagon, may be used to construct each other under unions and intersections, which implies an MLD equivalence between the resulting patterns, see \cite[Remark 7.6]{BG03}. Linear repetitivity is preserved under MLD equivalence, so this shows that the Penrose patterns are {\bf LR} (which is already known via their construction by substitution rules).

The Ammann--Beenker patterns give {\bf LR} examples for $n=8$. We may also find an {\bf LR} example for $n=12$, by taking $\Gamma_<$ as the integer span of the $12$th roots of unity. Indeed, the primitive $12$th root of unity $\xi = \sqrt{3}/2 + 0.5i$ satisfies $\xi^4-t^2+1 = 0$, so that $\Gamma_< = \langle 1, \xi, \xi^2, \xi^3\rangle_\Z \cong \Z^4$. The line $H$ through $1$ is stabilised by $1$ and $\xi + \overline{\xi} = \sqrt{3}$, which is badly approximable (and similarly for the other lines by symmetry), so each stabiliser has rank $2$ and is Diophantine.

The above shows that we may construct $4$-to-$2$ {\bf LR} schemes with regular $n$-gon window for $n=2,\ldots,6$, $8$, $10$ and $12$. If one insists that the lattice has corresponding symmetry, then there are no other possibilities. This follows from the fact that the minimal polynomials of the corresponding roots of unity have too large of a degree, which would make $\rk(\Gamma_<) > 4$. We are unsure if there are asymmetric choices of Diophantine $\Gamma_<$ for other values of $n$, intersecting each line with rank $2$.

\subsubsection{Cubical cut and project sets}
Our main Theorem \ref{thm: main} generalises a previous result \cite{HaynKoivWalt2015a} on cubical cut and project sets. We recall that a {\bf cubical cut and project scheme} $\mathcal{S}$ has $\tot = \R^k$, $\intl = \{0\}^d \times \R^n$, $W = [0,1]^n$, $\Gamma = \Z^k$ and
\[
\phy = \{(x_1,x_2,\ldots,x_d,L_1(x_1,\ldots,x_d),L_2(x_1,\ldots,x_d),\ldots,L_n(x_1,\ldots,x_d)) \mid x_i \in \R \},
\]
where each $L_i \colon \R^d \to \R$ is a linear form of $d$ variables. That is, we have that $L_i(x_1,\ldots,x_d) = \sum_{j=1}^d \alpha_{ij} x_j$ for numbers $\alpha_{ij} \in \R$. We call $L = \{L_1,\ldots,L_n\}$ a system of $n$ linear forms in $d$ variables. We assume that if $L_i(r) = 0$ for each $i$, then $r \in \Z^d$ is the $0$ vector, and that $(L_1(\Z^d),L_2(\Z^d),\ldots,L_n(\Z^d))$ is dense in $\R^n$.

Technically, the internal space should be slanted slightly to make $\pi_\vee$ injective on $\Gamma$, otherwise the cut and project sets produced will be periodic. However, the choice of such a projection does not affect linear repetitivity, and this choice of `reference space' is useful, since $\intl$ contains the subgroup $\{0\} \times \Z^n$, for which $W$ is a fundamental domain.

This cut and project scheme is as decomposable as a scheme can be, as the window is a product of $1$-dimensional intervals. Theorem 1.1 of \cite{HaynKoivWalt2015a} characterised {\bf LR} cubical cut and project sets as equivalent to properties called (LR1) and (LR2), stated in terms of the ranks of the kernels of the $L_i$ and a Diophantine condition for each, respectively. It is not difficult to verify using the results here that (LR1) is equivalent to {\bf C} and (LR2) is then equivalent to {\bf D}.

\subsubsection{Canonical cut and project sets}

We call a cut and project scheme {\bf canonical} if $\tot = \R^k$, $\Gamma = \Z^k$ and $W = [0,1]^k_<$, the projection of the unit hypercube. In addition to our usual assumptions, we will always assume the following {\bf non-degeneracy} condition: the projections of any size $n$ subset of the standard basis vectors of $\R^k$ to the internal space are linearly independent. The parallelepiped spanned by such vectors in the internal space can be regarded (after a change of basis) as a cubical window, as above, so this condition says that each cubical sub-window is of full dimension in the internal space. We note that this also guarantees indecomposability of the window:

\begin{proposition} \label{prop: canonicals indecomposable}
A canonical window satisfying the non-degeneracy condition is indecomposable.
\end{proposition}

\begin{proof}
Since the window is canonical, each $H \in \sH_0$ is of the form
\[
H = \langle u_1,u_2\ldots,u_{n-1}\rangle_\R,
\]
where each $u_i = (e_j)_<$ for some standard basis vector $e_j$ in the total space. We claim that such an $H$ is connected to the supporting subspace
\[
H' = \langle v,u_2,\ldots,u_{n-1}\rangle_\R
\]
in the flag-graph $G(W)$ of Theorem \ref{thm: disconnected decomposition graph}, where we obtain $H'$ from $H$ by swapping the spanning vector $u_1$ with a projection $v$ of some standard basis vector distinct from all other $u_i$.

To see this, first take any $x \in \intl$ which is a projected basis vector, not equal to any $u_i$ nor $v$ (since there are $k = d+n > n$ choices of basis vector, we can find such an $x$). Consider the span $Q_i$ of the projected basis vectors $\{x\} \cup \{u_j\}_{j \neq i}$ (that is, we replace $u_i$ in the definition of $H$ with $x$). By non-degeneracy, $\dim(Q_i) = n-1$ and $Q_i \in \sH_0$. Moreover, $\{x,u_1,\ldots,u_{n-1}\}$ is a basis for $\intl$. So we see that $Q_1 \cap \cdots \cap Q_{n-1} = \langle x \rangle_\R$, since with respect to this basis any given vector must have trivial $u_i$ component to belong to $Q_i$.

Suppose that $z \in H$. Then, with respect to the above basis, its $x$ component is trivial. If $z$ also belongs to each $Q_i$ then we see that $z = 0$ and hence $H \cap (Q_1 \cap \cdots \cap Q_{n-1}) = \{0\}$. Similarly, suppose that $z \in H' \cap (Q_1 \cap\cdots \cap Q_{n-1})$. Since $z$ belongs to each $Q_i$, we see that $z = \lambda x$ for some $\lambda \in \R$. So $z$ has all components except for the $x$ component trivial, with respect to the basis $\{x,v,u_2,\ldots,u_{n-1}\}$. However, since $z \in H'$, its $x$ component must be trivial too, so $z = 0$ and hence $H' \cap (Q_1 \cap \cdots \cap Q_{n-1}) = \{0\}$.

The above shows that the $Q_i$ determine a pre-flag connecting $H$ to $H'$ in the flag-graph $G(W)$. Since we may connect any hyperplane to any other by a chain of such edges, successively replacing a single spanning vector at each step, we see that this graph is connected and hence, by Theorem \ref{thm: disconnected decomposition graph}, $W$ is indecomposable.
\end{proof}

We give a characterisation of linear repetitivity of canonical cut and project sets in terms of the directions of the projections of the standard basis vectors to the internal space. We define the one-dimensional subspaces $X_i \coloneqq \langle (e_i)_< \rangle_\R$ in $\intl$. Let $G_i \coloneqq \Gamma_< \cap X_i$ and $r_i \coloneqq \rk(G_i)$.

\begin{theorem} \label{thm: canonical characterisation}
The following are equivalent for a (non-degenerate) canonical cut and project scheme:
\begin{enumerate}
	\item {\bf LR}
	\item The following both hold:
	\begin{enumerate}
		\item $r_i = k/n$ for each $i = 1$, \ldots, $k$ (equivalently, each $r_i \geq k/n$),
		\item each $G_i$ is a Diophantine lattice in $X_i$.
	\end{enumerate}
\end{enumerate}

Condition 2b may be replaced with the following weaker condition:

\begin{itemize}
	\item[(2b')] at least $n$ distinct $G_i$ are Diophantine in $X_i$.
\end{itemize}

\end{theorem}

\begin{proof}
Suppose that {\bf LR} holds, hence we have both {\bf C} and {\bf D}. Choose $n$ distinct indices $\alpha(i)$ between $1$ and $k$ and consider the projections $v_1$, \ldots, $v_n$ of the corresponding basis vectors of $\R^k$ into the internal space. We have a naturally associated flag $f = \{H_1, \ldots, H_n\} \subset \sH_0$, where $H_i$ is the span of $\{v_j\}_{j \neq i}$. The intersection of all $H_j$, except $H_i$, is the subspace $X(\alpha(i))$, and so $\widehat{\Gamma}^i = G_{\alpha(i)}$, as defined in Lemma \ref{lem: lattice spanned by lines}. By {\bf C}, it follows from that lemma that
\begin{equation} \label{eq: line ranks equal}
r_{\alpha(1)} + r_{\alpha(2)} + \cdots + r_{\alpha(n)} = k.
\end{equation}
Since this applies to any such subset of $n$ elements (and there are at least $k = n+d \geq n+1$ basis vectors), by swapping just one of the $n$ elements with any other in Equation \ref{eq: line ranks equal} and taking the difference of terms, we see that $r_i = r_j$ for any $i$, $j$, and hence $r_i = k/n$ for all $i$, proving 2a.

Since the window is indecomposable by Proposition \ref{prop: canonicals indecomposable}, we may apply Theorem \ref{thm: main} without decomposing the lattice (so that {\bf D} is equivalent to $\Gamma_<$ being Diophantine). By Equation \ref{eq: line ranks equal}, the sum of any $n$ distinct $G_i$ is finite rank in $\Gamma_<$, and hence Diophantine if and only if each $G_i$ is Diophantine by Lemmas \ref{lem: Diophantine finite index} and \ref{lem: Diophantine sums}, establishing 2b (and hence 2b').

Conversely, let us suppose that 2b' holds, and each $r_i \geq k/n$. Every supporting subspace $H \in \sH_0$ is spanned by $n-1$ of the lines $X_i$, so we see that the corresponding sum of linearly independent groups $G_i$ are contained in $\Gamma^H$. It follows that $\rk(H)$ is at least the sum of the ranks of the $G_i$, which is at least $(n-1)\cdot(k/n)$ by 2a. In fact, it can be no more, as by Equation \ref{eq: line ranks equal} we can complete this sum to a full rank subgroup of $\Gamma_<$ by summing an additional complementary $G_j$, so the canonical window has constant stabiliser rank. Hence by Theorem \ref{thm: generalised complexity} the cut and project sets have complexity exponent
\[
n\cdot(k-(n-1)\cdot(k/n)-1) = kn - (n-1)k - n = k-n = d,
\]
so {\bf C} holds. By the argument above for the converse direction (which only required {\bf C}), we have that the sum of any $n$ distinct $G_i$ is full rank in $\Gamma_<$. Choose $n$ distinct $G_i$ which are Diophantine (which we can, by 2b'). Their sum is Diophantine by Lemma \ref{lem: Diophantine sums}, and so $\Gamma_<$ is Diophantine by Lemma \ref{lem: Diophantine finite index}. It follows that {\bf D} holds, so the scheme is {\bf LR} by Theorem \ref{thm: main}.
\end{proof}

The following question was raised in \cite{HaynKoivWalt2015a}:

{\bf Problem 4.5.} \emph{Is it true that a canonical cut and project set will be LR if and only if all of the cubical cut and project sets obtained from taking different parametrizations of $\phy$, with respect to different orderings of the standard basis vectors, are also LR?}

The cubical cut and project sets above may be considered (again assuming non-degeneracy) as those coming from the original scheme after replacing the canonical window with the cubical sub-windows, the parallelepipeds spanned by the projections of $n$ arbitrary basis vectors in the internal space. The previous theorem establishes an answer to the above in the affirmative; in fact, we have the following stronger result:

\begin{theorem}
The following are equivalent for a non-degenerate canonical cut and project scheme:
\begin{enumerate}
	\item the canonical window satisfies {\bf LR}
	\item each cubical sub-window satisfies {\bf C} and one satisfies {\bf D}
	\item each cubical sub-window satisfies {\bf LR}.
\end{enumerate}
\end{theorem}

\begin{proof}
Assume that the canonical scheme is {\bf LR}. By Theorem \ref{thm: canonical characterisation}, each $G_i$ has rank $k/n$, making each $\rk(H) = (n-1)\cdot(k/n)$, as in the argument above. It follows from Theorem \ref{thm: generalised complexity} that each cubical sub-window satisfies {\bf C} and has constant stabiliser rank. Since the canonical window is indecomposable by Proposition \ref{prop: canonicals indecomposable}, we have that $\Gamma_<$ is Diophantine. So {\bf D} holds for the cubical windows too, which thus satisfy {\bf LR} by Theorem \ref{thm: main}.

Conversely, suppose the weaker condition that each cubical sub-window satisfies {\bf C}, and that one satisfies {\bf D}. Take any set of distinct $G_i$. Their sum of ranks is equal to $k$ by Lemma \ref{lem: lattice spanned by lines} and our assumption that each cubical window satisfies {\bf C}. Hence, identically to in the proof of Theorem \ref{thm: canonical characterisation}, each $r_i = k/n$ and the canonical window satisfies {\bf C}. Since it is indecomposable (Proposition \ref{prop: canonicals indecomposable}), it has constant stabiliser rank, so the same is true of the cubical windows. In particular, the cubical window satisfying {\bf D} has constant stabiliser rank, so that $\Gamma_<$ is Diophantine. Hence {\bf D} is satisfied for the canonical window, which is {\bf LR} by Theorem \ref{thm: main}
\end{proof}

\bibliographystyle{amsalpha}
\bibliography{biblio}

\end{document}